\documentclass[11pt]{article}

\usepackage{amssymb}
\usepackage{mathtools}
\usepackage{tikz}
\usetikzlibrary{matrix,arrows,decorations.pathmorphing,trees}
\usepackage{amsmath,color}

\usepackage[T1]{fontenc} % Use 8-bit encoding that has 256 glyphs
\usepackage{lmodern}\usepackage{graphicx}
\usepackage{microtype} % Slightly tweak font spacing for aesthetics

\usepackage[hmarginratio=1:1,top=32mm,columnsep=20pt]{geometry} % Document margins
\usepackage{hyperref}
\usepackage{amsthm}
\usepackage{url}
\usepackage{tikz}

\usepackage{graphicx}
 \usepackage[all]{xy}
\newtheorem{thm}{Theorem}[section]

\newtheorem{lem}[thm]{Lemma}
\newtheorem{prop}[thm]{Proposition}

\newtheorem{defn}[thm]{Definition}
\newtheorem{theorem}{Theorem}[section]

\newtheorem{remark}[theorem]{Remark}

\newtheorem{corollary}[theorem]{Corollary}
\newtheorem{example}[theorem]{Example}

\let \bb=\mathbb{}

\def\rm{\textrm{}}
\def\fudge{\mathchoice{}{}{\mkern.5mu}{\mkern.8mu}}
\def\bbc#1#2{{\rm \mkern#2mu\vbar\mkern-#2mu#1}}
\def\bbb#1{{\rm I\mkern-3.5mu #1}}
\def\bba#1#2{{\rm #1\mkern-#2mu\fudge #1}}
\def\bb#1{{\count4=`#1 \advance\count4by-64 \ifcase\count4\or\bba
A{11.5}\or \bbb B\or\bbc C{5}\or\bbb D\or\bbb E\or\bbb F \or\bbc
G{5}\or\bbb H\or \bbb I\or\bbc J{3}\or\bbb K\or\bbb L \or\bbb
M\or\bbb N\or\bbc O{5} \or \bbb P\or\bbc C{5}\or\bbb B\or\bbc
S{4.2}\or\bba T{10.5}\or\bbc U{5}\or \bba V{12}\or\bba
W{16.5}\or\bba X{11}\or\bba Y{11.7}\or\bba Z{7.5}\fi}}

\usepackage{amssymb}
\usepackage{relsize}
\usepackage{layout}

\begin{document}
%-----------------------------------------------------------------------------------

\title{\vspace{-15mm}\fontsize{17pt}{10pt}\selectfont\textbf{ Morphisms Cohomology and Deformations
of  Hom-algebras}} % Article title

\author{
%\large
\textsc{Anja Arfa\thanks{arfaanja.mail@gmail.com}}\\[2mm] %and Nizar Benfraj\thanks{benfraj\_nizar@yahoo.fr}}\\[2mm] % Your name
\normalsize Universit\'e de Sfax (Tunisia) \\
\normalsize Facult\'{e} des Sciences \\ % Your institution
\vspace{-5mm}
\and
\textsc{Nizar Ben Fraj\thanks{benfraj\_nizar@yahoo.fr}}\\[2mm] %and Nizar Benfraj\thanks{benfraj\_nizar@yahoo.fr}}\\[2mm] % Your name
\normalsize Universit\'{e} de Carthage (Tunisia) \\
\normalsize Institut Pr\'eparatoire aux Etudes d'Ing\'enieur de Nabeul\\ % Your institution
\vspace{-5mm}
\and
\textsc{Abdenacer Makhlouf\thanks{abdenacer.makhlouf@uha.fr}}\\[2mm] % Your name
\normalsize Universit\'e de Haute Alsace (France) \\
\normalsize Laboratoire de Math\'ematiques, Informatique et Applications\\ % Your institution
\vspace{-5mm}
}

%\title{ Morphisms Cohomology and Deformations
%of  Hom-algebras }
%%\label{firstpage}
%\author{Arfa Anja\thanks{D\'{e}partement de math\`{e}matiques,
%Facult\'{e} des sciences de Sfax BP 802, 3038 Sfax, Tunisie. E.mail:
%arfaanja.mail@gmail.com}
%\and Ben fraj Nizar\thanks{Institut Preparatoire aux Etudes D'injenieur de Nabeul
%University of Carthage--Tunisie. E.mail: benfraj\_nizar@yahoo.fr,}
%\and Abdenacer Makhlouf\thanks{Facult\'{e} des sciences et technique,
%Universit\'{e} de Haute alsace, France. E.mail: abdenacer.makhlouf@uha.fr}}

\date{}
%--------------------------------------------------------------------
\maketitle
% ----------------------------------------------------------------

\begin{abstract}
The purpose of this paper is to study  deformation theory
of
Hom-associative algebra morphisms and Hom-Lie algebra morphisms. We introduce a suitable cohomology and discuss  Infinitesimal deformations, equivalent deformations and obstructions. Moreover, we provide some examples.
\end{abstract}
\bigskip
\thispagestyle{empty}
%%%%%%%%%%%%%%%%%%%%%%%%%%%%%%%%%%%%%%%%%%%%%%%%%%%%%%%%%%%%%%%%%%%%%%%%%%%%

  \section*{Introduction}

The first instance of Hom-type algebras appeared in Physics literature when looking for quantum deformations of some algebras of vector fields, like Witt and Virasoro algebras, in connection with oscillator algebras. A quantum deformation consists of replacing the usual derivation by a $\sigma$-derivation. The main examples use Jackson derivation, it turns out that  the obtained algebras no longer satisfy Jacobi identity but  a modified version involving a homomorphism.
 These algebras
were called Hom-Lie algebras and studied by Hartwig, Larsson and Silvestrov in \cite{JDS} and \cite{DS}.
Hom-associative algebras play the role of associative algebras in the Hom-Lie
setting.
They were introduced by the last author  and Silvestrov in \cite{ms}, where it is shown
that the commutator bracket defined by the multiplication in a Hom-associative
algebra leads naturally to a Hom-Lie algebra. The adjoint functor was considered by D. Yau \cite{D2}. Usually, we call these type of algebras Hom-algebras because of the homomorphism involving in their structure.

The original deformation theory  was developed
by Gerstenhaber for ring and algebras using formal power series
in \cite{m}. It  is closely related to Hochschild cohomology. Then, it was extended to Lie algebras, using Chevalley-Eilenberg cohomology, by Nijenhuis and Richardson  \cite{ar}. Notice that a more general setting was considered by Fialowski and her collaborators in \cite{Fialow}.
Deformation theory of associative algebra morphisms have been   studied by Gerstenhaber and Schack in a series of papers \cite{ms1,ms2,ms3}. Deformations of Lie algebra morphisms have been considered  by
 Nijenhuis and Richardson  in \cite{ar},
and more recently  by Fr\'egier \cite{f}, see also \cite{fmy,f2}. Cohomology and Deformations of Hom-associative algebras and Hom-Lie algebras were studied first in \cite{ms} then completed in \cite{aem}.

The purpose of this paper is to  provide  first  a Hochschild cohomology of Hom-associative
algebras and a Chevalley-Eilenberg cohomology  of Hom-Lie algebras which are compatible with
Hom-algebra morphisms and then study their deformations.
 We aim to generalize the cohomology theory associated
to deformations of Lie algebra morphisms given  by  Fr\'{e}gier for Lie algebras in \cite{f}
and that introduced by  Gerstenhaber and Schack in \cite{ms3} for associative
algebra morphisms. Moreover, we generalize the algebra valued cohomology theory in \cite{aem} to any  bimodule.

The paper is organized as follows. In Section $1$, we review some basic definitions about Hom-algebras and 
Hom-type Hochschild cohomology (resp. Hom-type Chevalley-Eilenberg cohomology).
In Section $2$, we introduce a cohomology complex with values in an adjoint bimodule related to
deformations of multiplicative Hom-associative algebras and give an explicit
formula for the coboundary operator.
In Section $3$, we construct a cohomology complex of
Hom-algebra morphisms $\phi:\mathcal{A}\rightarrow \mathcal{B}$ for both Hom-associative algebras and Hom-Lie algebras. We define the module of $n$-cochains
of a morphism $\phi$ by $C^{n}(\phi,\phi)=C^{n}(\mathcal{A},\mathcal{A})
\otimes C^{n}(\mathcal{B},\mathcal{B})\otimes C^{n-1}(\mathcal{A},\mathcal{B})$
and give the  coboundary operator formula related to these triples. The corresponding cohomology is denoted by $H^{*}(\mathcal{A},\mathcal{B})$. Moreover, we provide some examples.
 Section $4$ deals with  deformations of Hom-associative algebra
morphisms. We  generalize Gerstenhaber and Schack Theorems given in    \cite{ms1}.
We have observed that the infinitesimal is a $2$-cocycle in
the deformation complex of the Hom-algebra morphism. Also,
it is shown that if $H^{2}(\mathcal{A},\mathcal{B})=0$,
then  every formal deformation is equivalent to a trivial deformation. Furthermore, we prove
that the obstruction to extend a deformation of order $N$ to a deformation of order $N+1$
is a 3-cocycle. One can derive as a consequence that if $H^{3}(\mathcal{A},\mathcal{B})=0$ then any infinitesimal
deformation can be extended. The paper is ended with a section where we explicitly compute  for  several examples   a cohomology of Hom-Lie algebra morphisms and discuss some deformations.

%\thispagestyle{empty}

%%%%%%%%%%%%%%%%%%%%%%%%%%%%%%%%%%%%%%%%%%%%%%%%%%%%%%%%%%%%%%%%%%%%%%%%%%%%%%%
%%%%%%%%%%%%%%%%%%%%%%%%%%%%%%%%%%%%%%%%%%%%%%%%%%%%%%%%%%%%%%%%%%%%%%%%%%%%%%%

  \section{Preliminaries}

%%%%%%%%%%%%%%%%%%%%%%%%%%%%%%%%%%%%%%%%%%%%%%%%%%%%%%%%%%%%%%%%%%%%%%%%%%%%%%%
%%%%%%%%%%%%%%%%%%%%%%%%%%%%%%%%%%%%%%%%%%%%%%%%%%%%%%%%%%%%%%%%%%%%%%%%%%%%%%%
In this  section, we recall some basic definitions and summarize
the Hom-type Hochschild cohomology and the Hom-type
Chevalley-Eilenberg cohomology. We refer to a Hom-algebra as a triple consisting of  a $\mathbb{K}$-vector  space or a module together with a bilinear map (a multiplication) and a linear map. We assume that $\mathbb{K}$ is an algebraically  closed field of characteristic 0, even if most of the result are valid for any field. These Hom-algebras aim to generalize classical algebraic structures and
the main feature  is that the identities
defining the structures are twisted by homomorphisms. In the sequel, we will write $\otimes$ for $\otimes_{\mathbb{K}}$,
$\mathcal{A}^{\otimes n}$ for the $n$-fold tensor product $\mathcal{A}
\otimes\cdots\otimes\mathcal{A}$ and $\mathcal{A}^{\times n}$ for $\mathcal{A}
\times\cdots\times\mathcal{A}$.

\begin{defn}
A Hom-associative algebra over $\mathcal{A}$ is a triple $(\mathcal{A},\mu,\alpha)$ consisting
of a $\mathbb{K}$-vector space $\mathcal{A}$, a bilinear map $\mu:\mathcal{A}\times \mathcal{A} \rightarrow \mathcal{A}$ and a linear map  $\alpha:\mathcal{A}\rightarrow \mathcal{A}$
satisfying
\begin{equation*}
\mu(\alpha(x),\mu(y,z))=\mu(\mu(x,y),\alpha(z)),~~\hbox{for all}~~x,y,z\in \mathcal{A}
\quad (\hbox{Hom-associativity}).
\end{equation*}
A Hom-associative algebra is called multiplicative
if $\alpha$ is an algebra morphism.
\end{defn}
%\begin{example}
%Let $\mathcal{A}$ be a $3$-dimensional vector space over $\mathbb{K}$, generated
%by $\{x_{1} ,x_{2},x_{3}\},~\mu:\mathcal{A} \times\mathcal{A} \rightarrow\mathcal{A}$
%be a multiplication defined by
%$$
%\begin{array}{llll}
%&\mu(x_{1},x_{1})=x_{1},
%&\mu(x_{1},x_{2})=\mu(x_{2},x_{1})=x_{2}, & \mu(x_{1},x_{3})=\mu(x_{3},x_{1})=x_{3},\\[0.1pt]
%&\mu(x_{2},x_{2})=x_{2},&\mu(x_{2},x_{3})=x_{3}, &\mu(x_{3},x_{2})=\mu(x_{3},x_{3})=x_{3},
%\end{array}
%$$
%and $\alpha:V\rightarrow V$ be a linear map defined by $\alpha(x_{1})=\lambda x_{1}
%+\gamma x_{3},~\alpha(x_{2} )=\lambda x_{2} +\gamma x_{3},~\alpha(x_{3})=(\lambda+\gamma)x_{3}$
%where $\lambda,\gamma\in\mathbb{K}^{*}$.
%Then $\mathcal{A}=(V,\mu,\alpha)$ is a Hom-associative algebra.
%\end{example}

\begin{defn}
A Hom-Lie algebra is a triple $(\mathcal{L},[\cdot, \cdot],\alpha)$ consisting of a $\mathbb{K}$-vector space $\mathcal{L}$, a bilinear map
$[\cdot, \cdot] :\mathcal{L}\times \mathcal{L} \rightarrow \mathcal{L}$ and a linear map  $\alpha :
\mathcal{L}\rightarrow \mathcal{L}$ satisfying
$$
[x,y]=-[y,x] ~\hbox{for all} ~x,y\in \mathcal{L}\quad \hbox{(skew-symmetry)},$$
$$\circlearrowleft_{x,y,z}[\alpha(x),[y,z]]=0 ~\hbox{for all}~ x,y,z\in \mathcal{L}\quad
\hbox{(Hom-Jacobi identity)},$$
 where $\circlearrowleft_{x,y,z}$
denotes summation over the cyclic permutation on $x, y, z$.
\end{defn}
\begin{defn}
Let $(\mathcal{A},\mu,\alpha)$ and $(\mathcal{A}',\mu',\alpha')~~
(\hbox{resp.}~(\mathcal{L},[\cdot, \cdot],\alpha) ~
\hbox{and}~(\mathcal{L}',[\cdot, \cdot]',\alpha'))$ be two Hom-associative
$(\hbox{resp. ~Hom-Lie })$ algebras. A linear map $\phi:\mathcal{A}\rightarrow
\mathcal{A}'$ (resp.  $\phi:\mathcal{L}\rightarrow
\mathcal{L}'$) is a
Hom-associative (resp. Hom-Lie) algebra morphism if
$$\mu'\circ(\phi\otimes\phi)=\phi\circ\mu~~(\hbox{resp. } [\cdot, \cdot]'\circ
(\phi\otimes\phi)=\phi\circ[\cdot, \cdot])
\quad\hbox{and}\quad\phi\circ\alpha=\alpha'\circ\phi.$$
\end{defn}

\begin{thm}[\cite{D1}]
Let $A=(A,\mu)$ be an associative algebra (resp. a Lie algebra) and $\alpha:A\rightarrow A$ be an algebra  morphism  with respect to
$\mu$, i.e. $\alpha\circ\mu=\mu\circ\alpha^{\otimes 2}$. Then $A_{\alpha}=(A,\mu_{\alpha}=\alpha\circ\mu,\alpha)$
is a Hom-associative algebra (resp. a Hom-Lie algebra).
\end{thm}
This  theorem provides  an easy way to deform a usual associative algebra (resp.  Lie algebra) to a Hom-associative algebra (resp. Hom-Lie algebra). It could be also generalized to provide a new Hom-associative algebra (resp. Hom-Lie algebra) from a given Hom-associative algebra (resp. Hom-Lie algebra) along an algebra morphism.
%%%%%%%%%%%%%%%%%%%%%%%%%%%%%%%%%%%%%%%%%%%%%%%%%%%%
\subsection{Representations of Hom-associative algebras}
%%%%%%%%%%%%%%%%%%%%%%%%%%%%%%%%%%%%%%%%%%%%%%%%%%%%%%
\begin{defn}
Let $(\mathcal{A},\mu,\alpha)$ be a Hom-associative algebra. A (left)
$\mathcal{A}$-module  is a triple $(M,f,\gamma)$ where $M$ is
$\mathbb{K}$-vector space and $f:M\rightarrow M$, $\gamma:\mathcal{A}\otimes M\rightarrow M$,
are $\mathbb{K}$-linear maps, such that the
following identity is satisfied
$$\gamma\circ(\mu\otimes f)=\gamma\circ(\alpha\otimes\gamma).$$
\end{defn}
\begin{prop}[Left adjoint $\mathcal{A}$-module]\label{left}
Let $(\mathcal{A},\mu,\alpha)$ and $(\mathcal{A}',\mu',\alpha')$
be two Hom-associative algebras and $\phi:\mathcal{A}\rightarrow\mathcal{A}'$
be a Hom-associative algebra morphism.\\
We consider the triple $(M,f,\gamma)$, where $M=\mathcal{A}'$, $\gamma=\rho_{l}=\mu'
(\phi\otimes id)$ and   $f=\alpha'$. Then $M$ is an
$\mathcal{A}$-module called adjoint representation of $\mathcal{A}$ induced
by the Hom-associative algebra morphism $\phi$.
\end{prop}
\begin{proof}
Indeed, for
%\begin{eqnarray*}
$ \rho_{l}: \mathcal{A}\times\mathcal{A}'\rightarrow\mathcal{A}'$ such that
$(a,a')\rightarrow \mu'(\phi(a),a'),
$
%\end{eqnarray*}
we prove that the following diagram commutes
$$
\xymatrix{
\mathcal{A}\otimes\mathcal{A}\otimes\mathcal{A}'
\ar[r]^{\; \; \;  \;  \mu\otimes\alpha'} \ar[d] _{\alpha\otimes
\rho_{l}} & \mathcal{A}\otimes\mathcal{A}'\ar[d]^{\rho_{l}} \\
\mathcal{A}\otimes\mathcal{A}' \ar[r] ^{\rho_{l}} & \mathcal{A}'
}
$$
\end{proof}
\begin{remark}\label{right} Similarly, we define the right adjoint module by considering the triple
$(\mathcal{A}',\rho_{r}=\mu'(id\otimes\phi),\alpha')$.
These
$\mathcal{A}$-modules are also called  adjoint representations.
\end{remark}
A bimodule structure is given by  left and right  actions with maps $\rho_{r}$ and $\rho_{l}$
 that satisfy the following additional condition
$\rho_{r}(\rho_{l}(x,y),\alpha(z))=\rho_{l}(\alpha(x),\rho_{r}(y,z)).$
Actually,  left and right modules are special cases of bimodules, one may set
$\rho_{r}=0 $ (resp. $\rho_{l}=0$).
%%%%%%%%%%%%%%%%%%%%%%%%%%%%%%%%%%%%%%%%%%%%%%%%%%%%
\subsection{Representation of Hom-Lie algebras}
%%%%%%%%%%%%%%%%%%%%%%%%%%%%%%%%%%%%%%%%%%%%%%%%%%%%%%
Let $(\mathcal{L},[\cdot, \cdot],\alpha)$ be a Hom-Lie algebra and
$\beta\in\mathfrak{g}l(V)$ be an arbitrary linear self map
on $V$, where $V$ is an arbitrary vector space. We denote  a left action
of $\mathcal{L}$ on $V$ by the following bracket
$
  [\cdot, \cdot]_{V}:\mathcal{L}\times V\rightarrow V$ such that
 $ (g,V)\rightarrow[g,v]_{V}$.

\begin{defn}
A triple $(V,[\cdot, \cdot]_{V},\beta)$ is called a left Hom-module
on the Hom-Lie algebra $\mathcal{L}$ or $\mathcal{L}$-Hom-module
$V$, with respect to $\beta\in\mathfrak{g}l(V)$ if it satisfies
$[\alpha(u),\beta(v)]_{V}=\beta([u,v]_{V})$
and
\begin{equation}\label{condition}
  [[u,v],\beta(z)]_{V}=[\alpha(u),[v,z]_{V}]_{V}-[\alpha(v),[u,z]_{V}]_{V}.
\end{equation}
We say that $(V,[\cdot, \cdot]_{V},\beta)$ is a representation of $\mathcal{L}$.
\end{defn}
\begin{prop}\label{representation}
Let $(\mathcal{L},[\cdot, \cdot],\alpha)$ and $(\mathcal{L}',[\cdot, \cdot]',\alpha')$ be
two Hom-Lie algebras and $\phi:\mathcal{L}\rightarrow\mathcal{L}'$
be a Hom-Lie algebra morphism.
Let  $(\Pi,\beta,\rho_{l})$ be a triple  where $\Pi=\mathcal{L}', \beta=\alpha',
\rho_{l}=[\cdot, \cdot]_{V}=[\phi, id]'$. Then, $\Pi$ is a left $\mathcal{L}$-module
called left adjoint representation of $\mathcal{L}$ via $\phi$.
\end{prop}
%%%%%%%%%%%%%%%%%%%%%%%%%%%%%%%%%%%%%%%%%%%%%%%%%%%%%%%%%%%%%%%%%%%%%%%%
\begin{remark}
[Coadjoint representation of $\mathcal{L}$]
Let $(\mathcal{L},[\cdot, \cdot],\alpha)$ be a Hom-Lie algebra and
$(V,[\cdot, \cdot]_{V},\beta)$ be a representation of $\mathcal{L}$. Let
$V^{*}$ be the dual vector space of $V$. We define a  bilinear
map $[\cdot, \cdot]_{V^{*}}:\mathcal{L}\times V^{*}\rightarrow V^{*}$ by
$[x,f]_{V^{*}}(v)=-f([x,v]_{V}), \forall x\in \mathcal{L}, f\in V^{*} \hbox{ and } v\in V.$
Let $f\in V^{*}$, $x,y\in\mathcal{L}$ and $v\in V$, we compute the right
hand side of the identity \eqref{condition}.
\begin{eqnarray*}
[\alpha(x),[y,f]_{V^{*}}]_{V^{*}}-[\alpha(y),[x,f]_{V^{*}}]_{V^{*}}(v)&=&
-[y,f]([\alpha(x),v]_{V})+[x,f]([\alpha(y),v]_{V})\\[1pt]
&=&f([y,[\alpha(x),v]_{V}]_{V})-f([x,[\alpha(y),v]_{V}]_{V}).
\end{eqnarray*}
On the other hand, we set  $\beta^{*}=^{t}\beta$ then the left
hand side of \eqref{condition} gives
\begin{eqnarray*}
[ [x,y],\beta^{*}(f)]_{V^{*}}(v)=-\beta^{*}(f)([[x,y],v]_{V})
 =-^{t}\beta(f)([[x,y],v]_{V})
 =-f\circ\beta([[x,y],v]_{V}).
\end{eqnarray*}
\end{remark}
\begin{prop}
Let $(\mathcal{L},[\cdot, \cdot],\alpha)$ be a Hom-Lie algebra
and $(V,[\cdot, \cdot]_{V},\beta)$ be a representation of $\mathcal{L}$.
The triple $(V^{*},[\cdot, \cdot]_{V^{*}},\beta^{*})$, where
$[x,f]_{V^{*}}(v)=-f([x,v]_{V}), \forall x \in\mathcal{L}, f\in V^{*}, v\in V$,
defines a representation of the Hom-Lie algebra $(\mathcal{L},[\cdot, \cdot],\alpha)$
if and only if
$$[[x,y],\beta(v)]_{V}=[x,[\alpha(y),v]_{V}]_{V}-[y,[\alpha(x),v]_{V}]_{V}.$$
\end{prop}

%%%%%%%%%%%%%%%%%%%%%%%%%%%%%%%%%%%%%%%%%%%%%%%%%%%%%%%%%%%%%%%%%%%%%%%%%%
\subsection{Hom-type Hochschild Cohomology}
%%%%%%%%%%%%%%%%%%%%%%%%%%%%%%%%%%%%%%%%%%%%%%%%%%%%%%%%%%%%%%%%%%%%%%%%%%
Let $(\mathcal{A},\mu,\alpha )$ be a multiplicative Hom-associative algebra.
We call a $n$-cochain of $\mathcal{A}$, with coefficients
in itself, any  $n$-linear map  from $\mathcal{A}^{\times n}$ to $\mathcal{A}$.
We  denote by $C^{n}_{Hom}(\mathcal{A},\mathcal{A})$ the space of $n$-cochains
defined to be the set of elements  $f\in Hom(\mathcal{A}^{\times n},\mathcal{A})$,
satisfying
$$\small{\alpha\circ f(x_{0},\cdots,x_{n-1})=f(\alpha(x_{0}),\cdots,
\alpha(x_{1}),\cdots,\alpha(x_{n-1}))~\hbox{for all}~x_{0},\cdots,x_{n-1}
\in\mathcal{A}}.$$
For $n=0$,  $ C^{0}_{Hom}(\mathcal{A},\mathcal{A})=\mathcal{A}$.
Let $f\in C^{n}_{Hom}(\mathcal{A},\mathcal{A})$, define
$\delta^{n}_{Hom}f\in C^{n+1}_{Hom}(\mathcal{A},\mathcal{A}) $ by
\begin{eqnarray}\label{hom}
&& \delta^{n}_{Hom}\varphi(x_{0},\cdots,x_{n})=\mu(\alpha^{n-1}(x_{0}),\varphi(x_{1},x_{2}
,\cdots,x_{n}))\\[1pt]\nonumber
&&+\sum\limits_{k=1}^{n}(-1)^{k}\varphi(\alpha(x_{0}),\alpha(x_{1}),\cdots,\alpha(x_{k-2}),
\mu(x_{k-1},x_{k}),\alpha(x_{k+1}),\cdots,\alpha(x_{n}))\\[1pt]\nonumber
&&+(-1)^{n+1}\mu(\varphi(x_{0},\cdots,x_{n-1}),\alpha^{n-1}(x_{n})).
\end{eqnarray}
So $(C^*_{Hom}(\mathcal{A},\mathcal{A})=\oplus_{n\geq 0}C^{n}_{Hom}(\mathcal{A},\mathcal{A}),\delta_{Hom})$ is a cochain complex defining a so called  Hom-type Hochschild
complex of $\mathcal{A}$ with coefficients in itself.
%We agree that $q<0,~C^{q}(\mathcal{A},\mathcal{A})=0$ and $\delta^{q}=0$.\\
%One checks that $\delta^{q+1}\circ\delta^{q}=0$ and that $Im\delta^{q-1}\subset
%Ker\delta^{q}$.\\
%The Hom-associative cochain complex look like this:
%$$
%0\longrightarrow M\longrightarrow Hom_{k}(\mathcal{A},\mathcal{A})
%\stackrel{\delta}\longrightarrow Hom_{k}(\mathcal{A}\otimes\mathcal{A},\mathcal{A})
%\stackrel{\delta^{2}}\longrightarrow
%\cdots$$
%For $n\in \mathbb{N}^{*}$, we denoted by $H^{n}(\mathcal{A},\mathcal{A})$ the $n^{th}$
%group of Hochschild  cohomology,
%$$H^{n}(\mathcal{A},\mathcal{A}):=
%\frac{Z^{n}(\mathcal{A},\mathcal{A})}{B^{n}(\mathcal{A},\mathcal{A})}$$
%\begin{prop}
%Let $(\mathcal{A},\mu,\alpha)$ be a Hom-associative algebra. Let
%$D_{\mu}^{\alpha}:C_{\alpha}(\mathcal{A},\mathcal{A})\rightarrow C_{\alpha}
%(\mathcal{A},\mathcal{A})$ for a linear map defined by
%$$D_{\mu}^{\alpha}\phi=[\mu,\phi]^{\Delta}_{\alpha}~~\hbox{for all}~\phi
%\in C_{\alpha}(\mathcal{A},\mathcal{A})$$
%\end{prop}

%%%%%%%%%%%%%%%%%%%%%%%%%%%%%%%%%%%%%%%%%%%%%%%%%%%%%%%%%%%%%%%%%%%%%%%%%%%%%%%%%
%\subsubsection{The algebra $C_{\alpha}(\mathcal{A},\mathcal{A})$}
%%%%%%%%%%%%%%%%%%%%%%%%%%%%%%%%%%%%%%%%%%%%%%%%%%%%%%%%%%%%%%%%%%%%%%%%%%%%%%%%%%%

More generally, let $M$ be a vector space and $\beta:M\rightarrow M$ be a linear map.
We denote by $C_{\alpha,\beta}(\mathcal{A},M)$ the space of all
$(n+1)$-linear maps $\varphi:\mathcal{A}^{\times n+1}\rightarrow M$, satisfying
$$\small{\beta(\varphi(x_{0},\cdots,x_{n}))=\varphi(\alpha(x_{0}),\cdots,\alpha(x_{n}))
\quad\hbox{for all}~x_{0},\cdots,x_{n}\in \mathcal{A}}.$$
Let $M$ be a $\mathcal{A}$-bimodule.  For
$\psi\in C_{\alpha,\beta}^{b}(\mathcal{A},M)$ and $\varphi\in C_{\alpha}^{a}(\mathcal{A},\mathcal{A})$,
where $a\geq0, b\geq0$, we define $j^{\alpha}_{\varphi}\psi \in C^{a+b+1}_{\alpha}(\mathcal{A},M)$ to be
 the composition product  given by the operator
\begin{eqnarray}
&& j_{\varphi}^{\alpha}(\psi)(x_{0},\cdots,x_{a+b})=  \nonumber \\ &&  \sum\limits_{k=0}^{b}
(-1)^{ak}\psi(\alpha^{a}(x_{0}),\cdots,\alpha^{a}(x_{k-1}),\varphi(x_{k},\cdots,
x_{k+a}),
\alpha^{a}(x_{a+k+1}),\cdots,\alpha^{a}(x_{a+b})).\label{cup produit 1}
\end{eqnarray}
When $M=\mathcal{A}$, we set,
$$[\varphi,\psi]^{\Delta}_{\alpha}=j_{\psi}^{\alpha}(\varphi)
-(-1)^{ab}j_{\varphi}^{\alpha}(\psi).$$
The bracket $[\varphi,\psi]^{\Delta}_{\alpha}$ is called Gerstenhaber bracket. It satisfies
$j_{[\varphi,\psi]^{\Delta}_{\alpha}}=[j_{\varphi}^{\alpha},j^{\alpha}_{\psi}]$. Moreover,
if $(\mathcal{A},\mu,\alpha)$  is a Hom-algebra, then
$[\mu,\mu]^{\Delta}_{\alpha}=0$ if and only if $(\mathcal{A},\mu,\alpha)$ is a
Hom-associative algebra.
\begin{remark}\cite{aem}.
The pair $(C_{\alpha}(\mathcal{A},\mathcal{A}),[\cdot, \cdot]_{\alpha}^{\Delta})$ is
a graded Lie algebra.
\end{remark}
%%%%%%%%%%%%%%%%%%%%%%%%%%%%%%%%%%%%%%%%%%%%%%%%%%%%%%%%%%%%%%%%%%%%%%%%%%%%%%%%%%%%%%%%%%%
%%%%%%%%%%%%%%%%%%%%%%%%%%%%%%%%%%%%%%%%%%%%%%%%%%%%%%%%%%%%%%%%%%%%%%%%%%%%%%%%%%%%%%%%
\subsection{Hom-type Chevalley-Eilenberg Cohomology}
\begin{defn}
Let $(\mathcal{L},[\cdot, \cdot],\alpha])$ be a Hom-Lie algebra.
We call a $n$-cochain of the  Hom-Lie  algebra $\mathcal{L}$,
with coefficients in itself, any $n$-linear alternating map
$f:\mathcal{L}^{\times n}\rightarrow \mathcal{L}$
satisfying
$$\small{\alpha\circ f(x_{0},\cdots,x_{n-1})=f(\alpha(x_{0}),\cdots,
\alpha(x_{1}),\cdots,\alpha(x_{n-1}))~\hbox{for all}~x_{0},\cdots,x_{n-1}
\in\mathcal{L}}.$$
The space of $n$-cochains is denoted by $C^{n}_{HL}(\mathcal{L},\mathcal{L})$. We set $C_{HL}(\mathcal{L},\mathcal{L})=\oplus_{n\geq 0} C^{n}_{HL}(\mathcal{L},\mathcal{L})$.\\
Let $\varphi\in C^{n}_{HL}(\mathcal{L},\mathcal{L})$, define $\delta^{n}\varphi
\in C^{n+1}_{HL}(\mathcal{L},\mathcal{L})$ by
\begin{eqnarray}\label{alpha}
\delta^{n}_{HL}\varphi(x_{0},\cdots,x_{n})&=&\sum\limits_{i=0}^{n}
(-1)^{i}[\alpha^{n-1}(x_{i})),\varphi (x_{0},\cdots,\widehat{x_{i}},\cdots,x_{n})]\\[2pt]
\nonumber
&+&\sum\limits_{0\leq i<j\leq n+1}(-1)^{i+j}\varphi([x_{i},x_{j}],\alpha(x_{0})
,\cdots,\widehat{x_{i}},\cdots,\widehat{x_{j}},\cdots,\alpha(x_{n})).
\end{eqnarray}
\end{defn}
Then, we have a cohomology  complex $(C^{*}_{HL}(\mathcal{L},\mathcal{L}),\delta)$ which we call
Hom-type Chevalley-Eilenberg complex.

%%%%%%%%%%%%%%%%%%%%%%%%%%%%%%%%%%%%%%%%%%%%%%%%%%%%%%%%%%%%%%%%%%%%%%%%%%%%%%%%%%%%%%%%%%%%
%\subsubsection{The algebra $\widetilde{C}_{\alpha}(\mathcal{L},\mathcal{L})$}
%%%%%%%%%%%%%%%%%%%%%%%%%%%%%%%%%%%%%%%%%%%%%%%%%%%%%%%%%%%%%%%%%%%%%%%%%%%%%%%%%%%%%%%%%%%%

More generally, let $A$ be a vector space and  $\alpha:A \rightarrow A$ be a linear map. Let $\varphi:A^{\times(n+1)}\rightarrow A$ be a $(n+1)$-linear alternating map satisfying
for all $x_{0},\cdots,x_{n}\in A,$
$\alpha(\varphi(x_{0},\cdots,x_{n}))=\varphi(\alpha(x_{0}),\cdots,\alpha(x_{n})).$
We  denote their set by $\widetilde{C}^{n}_{\alpha}(A,A)$  and
$$\widetilde{C}_{\alpha}(A,A)=\bigoplus_{n\geq-1}\widetilde{C}^{n}_{\alpha}(A,A).$$
We define the alternator $\lambda:C_{\alpha}(A,A)\rightarrow\widetilde{C}_{\alpha}
(A,A)$ by
$$(\lambda\varphi)(x_{0},\cdots,x_{a})=\frac{1}{(a+1)!}\sum_{\sigma\in\mathcal{S}_{a+1}}
\varepsilon(\sigma)\varphi(x_{\sigma(0)},\cdots,x_{\sigma(a)})~\text{for}~\varphi\in
C^{a}_{\alpha}(A,A),$$
where $\mathcal{S}_{a+1}$ is a permutation group and $\varepsilon(\sigma)$ is the signature of the permutation $\sigma$.

We define an operator and a bracket for $\varphi\in\widetilde{C}^{a}_{\alpha}(A,A)$,
 $\psi\in\widetilde{C}^{b}_{\alpha}(A,A)$ by
$$[\varphi,\psi]^{\wedge}_{\alpha}:=i_{\varphi}^{\alpha}(\psi)-
(-1)^{ab}i_{\psi}^{\alpha}(\varphi) \text{ where }
i_{\varphi}(\psi):=\frac{(a+b+1)!}{(a+1)!(b+1)!}\lambda(j_{\varphi}^{\alpha}\psi)
.$$
Thus $i_{\varphi}^{\alpha}(\psi)\in\widetilde{C}_{\alpha}^{a+b+1}$.
The bracket $ [\varphi,\psi]^{\wedge}_{\alpha}$ is called Nijenhuis-Richardson
bracket.
\begin{remark}\cite{aem}.
The pair $(\widetilde{C}_{\alpha}(A,A),[\cdot, \cdot]^{\wedge}_{\alpha})$
is a graded Lie algebra.
\end{remark}
%%%%%%%%%%%%%%%%%%%%%%%%%%%%%%%%%%%%%%%%%%%%%%%%%%%%%%%%%%%%%%%%%%%%%%%%%%%%%

\section{Cohomology of Hom-algebras with values in an adjoint bimodule}
%%%%%%%%%%%%%%%%%%%%%%%%%%%%%%%%%%%%%%%%%%%%%%%%%%%%%%%%%%%%%%%%%%%%%%%%%%%%%
The first and the second cohomology groups of Hom-associative algebras and
Hom-Lie algebras were introduced in \cite{ms4}. An $\mathcal{A}$-valued
cohomology complex were introduced for multiplicative Hom-algebras
in \cite{aem}.  The purpose of this section is to construct a
cochain complex cohomology for  multiplicative Hom-associative algebras (resp. Hom-Lie algebras)
with values in any $\mathcal{A}$-bimodule $M$ (resp. $\mathcal{L}$-left module $\Pi$).
%%%%%%%%%%%%%%%%%%%%%%%%%%%%%%%%%%%%%%%%%%%%%%%%%%%%%%%%%%%%%%%%%%%%%%%%%%%%%%%%%%%%%%%%%%%%%%%%%%%%%
\subsection{Cohomology of Hom-associative algebras with values in an adjoint $\mathcal{A}$-bimodule}
%%%%%%%%%%%%%%%%%%%%%%%%%%%%%%%%%%%%%%%%%%%%%%%%%%%%%%%%%%%%%%%%%%%%%%%%%%%%%%%%%%%%%%%%%

We  construct a cochain complex $C^{*}_{\alpha,\alpha'}(\mathcal{A},M)$ that defines a Hom-type Hochschild cohomology
for multiplicative Hom-associative algebras in an adjoint $\mathcal{A}$-bimodule $M$.
Let $(\mathcal{A},\mu,\alpha)$ and $(\mathcal{A}',\mu',\alpha')$ be two
Hom-associative algebras over $\mathbb{K}$ and $\phi:\mathcal{A}\rightarrow\mathcal{A}'$
be a Hom-associative algebra morphism.
Let $M=(\mathcal{A}',\rho_{l},\rho_{r})$ be a $\mathcal{A}$-bimodule, where
$\rho_{l}~\hbox{and}~\rho_{r}$ are defined in Proposition \ref{left} (resp. Remark \ref{right}).
Regard $\mathcal{A}'$ as a $\mathcal{A}$-bimodule via the adjoint representation
of $\mathcal{A}$ induced by $\phi$.\\
The set of $n$-cochains on $\mathcal{A}$ with values in a $\mathcal{A}$-bimodule
$M$, is defined to be the set of $n$-linear maps which are
 compatible with $\alpha$ and $\alpha'$ in the sense that $\alpha'\circ f=f\circ\alpha^{\otimes n}$ , i.e.
$$\alpha' \circ f(u_{1},\cdots, u_{n}) = f(\alpha(u_{1}),\cdots , \alpha(u_{n}))
\quad \hbox{for all}~~u_{1},\cdots,u_{n}\in\mathcal{A}.$$
We denote by $C^{n}_{Hom}(\mathcal{A},M)$ the set of $n$-linear maps from $\mathcal{A}$ to $M$
and  by $C^{n}_{\alpha,\alpha'}(\mathcal{A},M)$ the set of $n$-Hom-cochains:
$$C^{n}_{\alpha,\alpha'}(\mathcal{A},M)=\{f\in C^{n}_{Hom}(\mathcal{A},M): \alpha'\circ f=f\circ\alpha^{\otimes n}\}.$$
For $n=0$, we have $C^{0}_{\alpha,\alpha'}(\mathcal{A},M)=M$.
\begin{defn}
We call, for $n\geq1$, $n$-coboundary operator associated to the triple
$(\mathcal{A},M,\phi)$, the linear map $\delta^{n}_{Hom}:C^{n}_{\alpha,\alpha'}
(\mathcal{A},M)\rightarrow C^{n+1}_{\alpha,\alpha'}(\mathcal{A},M)$
defined by
\begin{eqnarray}\label{hom-morphism}
&& \delta^{n}_{Hom}\varphi(x_{0},x_{1},\cdots,x_{n})=
 \mu'(\phi(\alpha^{n-1}(x_{0}))),\varphi(x_{1},x_{2},\cdots,x_{n}))\\[2pt]
 \nonumber
  &&+\sum\limits_{k=1}^{n}(-1)^{k}\varphi(\alpha(x_{0}),\alpha(x_{1}),\cdots,
  \alpha(x_{k-2}),\mu(x_{k-1},x_{k}),\alpha(x_{k+1}),\cdots,\alpha(x_{n}))
 \\[2pt] \nonumber
 &&+(-1)^{n+1}\mu'(\varphi(x_{0},\cdots,x_{n-1}),\phi(\alpha^{n-1}(x_{n}))).
\end{eqnarray}
\end{defn}
\begin{lem}\label{lem}
Let $D_{i}:C^{n}_{\alpha,\alpha'}(\mathcal{A},M)\rightarrow C^{n+1}_{\alpha,\alpha'}
(\mathcal{A},M)$ be  linear operators defined for $\varphi\in C^{n}
 _{\alpha,\alpha'}(\mathcal{A},M)$ and $x_{0},x_{1},\cdots,x_{n}\in\mathcal{A}$ by
\small{\begin{equation*}\begin{array}{lllll}
&D_{0}^{n}\varphi(x_{0},x_{1},\cdots,x_{n})=-\mu'(\phi(\alpha^{n-1}(x_{0})),\varphi(x_{1},\cdots,x_{n}))
+\varphi(\mu(x_{0},x_{1}),\alpha(x_{2}),\cdots,\alpha(x_{n})),\\[2pt]
&D_{i}^{n}\varphi(x_{0},x_{1},\cdots,x_{n})=\varphi(\alpha(x_{0}),\cdots,\mu(x_{i},x_{i+1}),\cdots,
\alpha(x_{n}))\quad\hbox{for}\quad1\leq i\leq n-2,\\[2pt]
&D_{n-1}^{n}\varphi(x_{0},\cdots,x_{n})=\varphi(\alpha(x_{0}),\cdots,
\alpha(x_{n-2}),\mu(x_{n-1},x_{n}))-\mu'(\varphi(x_{0},\cdots,x_{n-1}),\phi(\alpha^{n-1}
(x_{n}))),\\[2pt]
&D_{i}^{n}\varphi=0\quad\hbox{for}\quad i\geq n.
\end{array}
\end{equation*}}
\end{lem}
Then
$D_{i}^{n+1}D_{j}^{n}=D_{j}^{n+1}D_{i-1}^{n} \text{ for }{ 0\leq j< i\leq n},\quad\hbox{and}\quad
\delta^{n}_{Hom}=\sum\limits_{i=0}^{n}(-1)^{i+1}D_{i}^{n}.$
\begin{prop}
Let $(\mathcal{A},\mu,\alpha)$ be a Hom-associative algebra and $
\delta_{Hom}^{n}:C^{n}_{\alpha,\alpha'}(\mathcal{A},M)\rightarrow C^{n+1}_{\alpha,\alpha'}
(\mathcal{A},M)$ be the operator defined in \eqref{hom-morphism}.  Then, 
$\delta^{n+1}_{Hom}\circ\delta^{n}_{Hom}=0\quad\hbox{for}~~n\geq1.$
\end{prop}
\begin{proof}
Indeed
\begin{small}
\begin{eqnarray*}
%\nonumber to remove numbering (before each equation)
\delta^{n+1}_{Hom}\circ\delta^{n}_{Hom}&=&\sum\limits_{0\leq i,j\leq n}(-1)^{i+j}
D_{i}^{n+1}D_{j}^{n}=\sum\limits_{0\leq j<i\leq n}(-1)^{i+j}D_{i}^{n+1}D_{j}^{n}+
\sum_{0\leq i\leq j\leq n}(-1)^{i+j}D_{i}^{n+1}D_{j}^{n}\\[2pt]
&=&\sum_{0\leq j<i\leq n}(-1)^{i+j}D_{j}^{n+1}D_{i-1}^{n}+\sum\limits_{0\leq i\leq j\leq n}
(-1)^{i+j}D_{i}^{n+1}D_{j}^{n}\\[2pt]
&=&\sum_{0\leq j\leq k \leq n}(-1)^{k+j+1}D_{j}^{n+1}D_{k}^{n}+\sum_{0\leq i\leq j\leq n}
(-1)^{i+j}D_{i}^{n+1}D_{j}^{n}
=0.
\end{eqnarray*}
\end{small}
\end{proof}
\begin{lem}
With respect to the above notation, for $\varphi\in C^{n}_{\alpha,\alpha'}(\mathcal{A},M)$
and by the multiplicative property of these algebras,  we have
$\delta^{n}_{Hom}\varphi\circ\alpha^{\otimes{n+1}}=\alpha'\circ\delta^{n}_{Hom}\varphi.$
Thus, the  map
$\delta^{n}:C^{n}_{\alpha,\alpha'}(\mathcal{A},M)
\rightarrow C^{n+1}_{\alpha,\alpha'}(\mathcal{A},M)$ is well defined.
\end{lem}
\begin{remark}
General case: Let $\mathcal{A}$ and $\mathcal{A}'$ be two Hom-associative
algebras. Let $M = (\mathcal{A}',\rho_{l},\rho_{r})$ be a $\mathcal{A}$-bimodule, where $\rho_{l}$ and $\rho_{r}$ are left
$\mathcal{A}$-module and 
right $\mathcal{A}$-module respectively. For $\varphi\in :C^{n}_{\alpha,\alpha'}(\mathcal{A},M)$, we set 
\begin{eqnarray*}
% % \nonumber to remove numbering (before each equation)
&& \delta^{n}_{Hom}\varphi(x_{0},x_{1},\cdots,x_{n})=
\rho_{l}(\alpha^{n-1}(x_{0}),\varphi(x_{1},x_{2},\cdots,x_{n}))\\[2pt]\nonumber
&&+\sum\limits_{k=1}^{n}(-1)^{k}\varphi(\alpha(x_{0}),\alpha(x_{1}),\cdots,
\alpha(x_{k-2}),\mu(x_{k-1},x_{k}),\alpha(x_{k+1}),\cdots,\alpha(x_{n}))
\\[2pt]\nonumber
&&+(-1)^{n+1}\rho_{r}(\varphi(x_{0},\cdots,x_{n-1}),\alpha^{n-1}(x_{n})).
\end{eqnarray*}
Then $\delta^{n+1}_{Hom}\circ\delta^{n}_{Hom}=0$.\\
The proof is similar to  that of  Lemma \ref{lem}.
and for $D_{i}^{n+1}\circ D_{j}^{n}=D_{j}^{n+1}\circ D_{i-1}^{n}$,
we use the compatibility between the left $\mathcal{A}$-module
and the right $\mathcal{A}$-module
%$$\rho_{r}(\rho_{l}(x,y),\alpha(z))=\rho_{l}(\alpha(x),\rho_{r}(y,z))$$
%and $$\rho_{l}\circ(\mu\otimes f)=\rho_{l}\circ(\alpha\otimes\rho_{l})$$
and the multiplicativity  of the algebras.
\end{remark}
\begin{remark}
In the particular case where  $M=\mathcal{A}$ and $\rho_{l},\rho_{r}=\mu$,
the Hom-associative algebra is a $\mathcal{A}$-bimodule over itself
and $\delta^{n}_{Hom}$ is  the same as in  \eqref{hom}.
\end{remark}
%\begin{defn}
The space of $n$-cocycles is
$Z^{n}_{Hom}(\mathcal{A},M)=\{\varphi\in C^{n}_{\alpha,\alpha'}(\mathcal{A},M)
:\delta^{n}_{Hom}\varphi=0\},$
and the space of $n$-coboundaries is
$B^{n}_{Hom}(\mathcal{A},M)=\{\psi=\delta^{n-1}_{Hom}\varphi:
\varphi\in C^{n-1}_{\alpha,\alpha'}(\mathcal{A},M)\}.$
%\end{defn}
%\begin{lem}
Obviously $B^{n}_{Hom}(\mathcal{A},M)\subset Z^{n}_{Hom}(\mathcal{A},M).$
%\end{lem}
%\begin{defn}
We call the $n^{th}$ Hochschild cohomology group of the Hom-associative
algebra $\mathcal{A}$ with values in an adjoint $\mathcal{A}$-bimodule the quotient
$H^{n}_{Hom}(\mathcal{A},M)=\frac{Z^{n}_{Hom}(\mathcal{A},M)}{B^{n}_{Hom}(\mathcal{A},M)}.$
%\end{defn}
%%%%%%%%%%%%%%%%%%%%%%%%%%%%%%%%%%%%%%%%%%%%%%%%%%%%%%%%%%%%%%%%%%%%%%%%%%%%%%%%%%%%%%%%%%%%%%

\subsection{Cohomology complex of multiplicative Hom-Lie algebras with values
in a left $\mathcal{L}$-module}

%%%%%%%%%%%%%%%%%%%%%%%%%%%%%%%%%%%%%%%%%%%%%%%%%%%%%%%%%%%%%%%%%%%%%%%%%%%%%%%%%%%%%%%%%%%%%
Now, we construct a cochain complex $C^{*}_{HL}(\mathcal{L},\Pi)$
that defines a Chevalley-Eilenberg cohomology for multiplicative Hom-Lie algebras
with values in a left $\mathcal{L}$-module $\Pi$.\\
Let $\mathcal{L}$ and $\mathcal{L}'$ be two Hom-Lie algebras and
$\phi:\mathcal{L}\rightarrow\mathcal{L}'$ be a Hom-Lie morphism.
Regard $\mathcal{L}$ as a representation $\Pi$ of $\mathcal{L}$ via $\phi$ defined by
\eqref{representation}.
The set $C^{k}_{HL}(\mathcal{L},\Pi)$ of $k$-cochains on $\mathcal{L}$ with values in $\Pi$
is the set of skewsymmetric
$\mathbb{K}$-linear maps from $\mathcal{L}^{\times k}$ to $\Pi:$\\
$$C^{k}_{HL}(\mathcal{L};\Pi)=\{f:\wedge^{k}\mathcal{L}\rightarrow \Pi~~\hbox{ linear map}\}.$$
A $k$-Hom-cochain on $\mathcal{L}$ with values in $\Pi$ is defined to be a $k$-cochain
 $f \in C^{k}_{HL}(\mathcal{L},\Pi)$ such that it
is compatible with $\alpha$ and $\alpha'$ in the sense that $\alpha'\circ f=f\circ\alpha^{\otimes k}$ , i.e.
\begin{equation}\label{condition cochain}
\alpha'(f(u_{1}, \cdots, u_{k})) = f(\alpha(u_{1}),\cdots , \alpha(u_{k}))\quad\hbox{for all }~
u_{1},\cdots,u_{n}\in\mathcal{L}.
\end{equation}
We denote  by $\widetilde{C}^{k}_{\alpha,\alpha'}(\mathcal{L},\Pi)$ the set of $k$-Hom-cochains:
$$\widetilde{C}^{k}_{\alpha,\alpha'}(\mathcal{L},\Pi)=\{f\in C^{k}_{HL}(\mathcal{L},\Pi):
\alpha'\circ f=f\circ\alpha\}.$$
For $k=0$, we have $\widetilde{C}^{0}_{\alpha,\alpha'}(\mathcal{L},\Pi)=\Pi$.
\begin{defn}
Let $(\mathcal{L},[\cdot, \cdot],\alpha)$ and  $(\mathcal{L},[\cdot, \cdot]',\alpha')$ be two
Hom-Lie algebras. Let  $\phi:\mathcal{L}\rightarrow \mathcal{L}'$ be a Hom-Lie
algebra morphism.
Regard $\mathcal{L}'$ as a representation of $\mathcal{L}$ via $\phi$ wherever appropriate.
We call, for $n\geq1$, n-coboundary operator associated to the triple $(\mathcal{L},\Pi,\phi)$
the linear map
$
\delta^{n}:C^{n}_{HL}(\mathcal{L},\Pi)\rightarrow C^{n+1}_{HL}(\mathcal{L},\Pi)
$
defined by
\begin{eqnarray}\label{phi}
\delta^{n}_{HL}\varphi(x_{0},\cdots,x_{n})&=&\sum\limits_{i=0}^{n}
(-1)^{i}[\phi(\alpha^{n-1}(x_{i})),\varphi (x_{0},\cdots,\widehat{x_{i}},\cdots,x_{n})]^{'}\\[2pt]
\nonumber
&+&\sum\limits_{0\leq i<j \leq n} (-1)^{i+j}\varphi([x_{i},x_{j}],\alpha(x_{0})
,\cdots,\widehat{x_{i}},\cdots,\widehat{x_{j}},\cdots,\alpha(x_{n})).
\end{eqnarray}
\end{defn}
We show that  $(\widetilde{C}^{*}_{\alpha,\alpha'}(\mathcal{L},\Pi),\delta_{HL})$ is a
cochain complex. The corresponding cohomology denoted by $H^{*}_{HL}(\mathcal{L},\Pi)$,
is called the cohomology of the Hom-Lie algebra
$\mathcal{L}$ with coefficients in the representation $\Pi$.
\begin{lem}
With respect to the above notation, for any
$f\in \widetilde{C}^{n}_{\alpha,\alpha'}(\mathcal{L},\Pi)$, we have
$$\delta^{n}_{HL}(f)\circ\alpha^{\otimes n+1}=\alpha'\circ\delta^{n}_{HL}(f).$$
Thus $\delta^{n}:\widetilde{C}^{n}_{\alpha,\alpha'}(\mathcal{L},\Pi)
\rightarrow \widetilde{C}^{n+1}_{\alpha,\alpha'}(\mathcal{L},\Pi)$ is well defined.
\end{lem}
\begin{proof}
Let $f\in \widetilde{C}^{n}_{\alpha,\alpha'}(\mathcal{L},\Pi)$ and
$(x_{0},\cdots,x_{n})\in\mathcal{L}^{n+1}$, then
$$\begin{array}{lllll}
&\delta^{n}_{HL}f\circ\alpha(x_{0},\cdots,x_{n})
=\delta_{HL}^{n}f(\alpha(x_{0}),\cdots,\alpha(x_{n}))\\[2pt]
&=\sum\limits_{i=0}^{n}
(-1)^{i}[\phi(\alpha^{n}(x_{i})),f(\alpha(x_{0}),\cdots,\widehat{x_{i}},\cdots,\alpha(x_{n}))]^{'}\\[2pt]
&+\sum\limits_{0\leq i<j \leq n} (-1)^{i+j}f([\alpha(x_{i}),\alpha(x_{j})],\alpha^{2}(x_{0})
,\cdots,\widehat{x_{i}},\cdots,\widehat{x_{j}},\cdots,\alpha^{2}(x_{n}))\\[2pt]
&=\sum\limits_{i=0}^{n}
(-1)^{i}[\phi(\alpha^{n}(x_{i})),f\circ\alpha(x_{0},\cdots,\widehat{x_{i}},\cdots,x_{n})]^{'}\\[2pt]
&+\sum\limits_{0\leq i<j \leq n} (-1)^{i+j}f\circ\alpha([x_{i},x_{j}],\alpha(x_{0})
,\cdots,\widehat{x_{i}},\cdots,\widehat{x_{j}},\cdots,\alpha(x_{n}))\\[2pt]
&=\sum\limits_{i=0}^{n}
(-1)^{i}\alpha'([\phi(\alpha^{n-1}(x_{i})),f(x_{0},\cdots,\widehat{x_{i}},\cdots,x_{n})]^{'})\\[2pt]
&+\sum\limits_{0\leq i<j \leq n} (-1)^{i+j}\alpha'\circ f([x_{i},x_{j}],\alpha(x_{0})
,\cdots,\widehat{x_{i}},\cdots,\widehat{x_{j}},\cdots,\alpha(x_{n}))\\[2pt]
&=\alpha'\circ\delta^{n}_{HL}(f)(x_{0},\cdots,x_{n}).
\end{array}$$
\end{proof}
\begin{thm} We have
$\delta^{n+1}_{HL}\circ\delta^{n}_{HL}=0$.
\end{thm}
\begin{proof}
\small{\begin{eqnarray}
\nonumber
 &&  \delta^{n+1}_{HL}\circ\delta^{n}_{HL}\varphi(x_{0},\cdots,x_{n+1})
 =\sum\limits_{i=0}^{n+1}(-1)^{i}[\phi(\alpha^{n}(x_{i})),\delta_{HL}^{n}
  \varphi(x_{0},\cdots,\widehat{x}_{i},\cdots,x_{n+1})]'\\
  \nonumber &&+\sum\limits_{0\leq i<j\leq n+1}(-1)^{i+j}\delta^{n}_{HL}\varphi([x_{i},x_{j}],
  \alpha(x_{0}),\cdots,\widehat{x}_{i},\cdots,\widehat{x}_{j},\cdots,\alpha(x_{n+1}))
    \end{eqnarray}
 with \begin{eqnarray}
 &&\nonumber \sum\limits_{i=0}^{n+1}(-1)^{i}[\phi(\alpha^{n}(x_{i})),\delta_{HL}^{n}
  \varphi(x_{0},\cdots,\widehat{x}_{i},\cdots,x_{n+1})]'\\[1pt]
  &&=[\phi(\alpha^{n}(x_{i})),\sum_{p=0}^{i-1}(-1)^{p}[\phi(\alpha^{n-1}
  (x_{p})),\varphi(x_{0},\cdots,\widehat{x}_{p},\cdots,\widehat{x}_{i},\cdots,x_{n+1})]'\label{terme1}\\[0.5pt]
  &&+[\phi(\alpha^{n}(x_{i}),\sum_{p=i+1}^{n+1}(-1)^{p-1}
  [\phi(\alpha^{n-1}(x_{p})),\varphi(x_{0},\cdots,\widehat{x}_{i},\cdots,\widehat{x}_{p}
  ,\cdots,x_{n+1})]'\label{terme2}\\[0.5pt]
  &&+[\phi(\alpha^{n}(x_{i})),\sum_{i<p<q}(-1)^{p+q}\varphi([x_{p},x_{q}]
  ,\alpha(x_{0}),\cdots,\widehat{x}_{i},\cdots,\widehat{x}_{p},\cdots,\widehat{x}_{q},
  \cdots,\alpha(x_{n+1}))]'\label{terme3}
  \\[0.5pt]
  &&+[\phi(\alpha^{n}(x_{i})),\sum_{p<i<q}(-1)^{p+q+1}\varphi
  ([x_{p},x_{q}],\alpha(x_{0}),\cdots,\widehat{x}_{p},\cdots,\widehat{x}_{i},\cdots,\widehat{x}_{q}
  ,\cdots,\alpha(x_{n+1}))]'\label{terme4}\\[0.5pt]
  &&+[\phi(\alpha^{n}(x_{i})),\sum_{p<q<i}(-1)^{p+q}\varphi([x_{p},x_{q}]
  ,\alpha(x_{1}),\cdots,\widehat{x}_{p},\cdots,\widehat{x}_{q},\cdots,
  \widehat{x}_{i},\cdots,\alpha(x_{n+1}))]'\label{terme5}
\end{eqnarray}}
and
$$\delta^{n}_{HL}\varphi([x_{i},x_{j}],\alpha(x_{0}),\cdots,\widehat{x}_{i}
,\cdots,\widehat{x}_{j},\cdots,\alpha(x_{n+1})$$
\small{\begin{eqnarray}
&=&[\phi(\alpha^{n-1}([x_{i},x_{j}])),\varphi(\alpha(x_{0}),\cdots
,\widehat{x}_{i},\cdots,\widehat{x}_{j},\cdots,\alpha(x_{n+1}))]\label{terme6}\\[0.5pt]
&+&\sum_{p=0}^{i-1}(-1)^{p-1}[\phi(\alpha^{n}(x_{p})),\varphi([x_{i},x_{j}]
,\alpha(x_{0}),\cdots,\widehat{x}_{p},\cdots,\widehat{x}_{i},\cdots,
\widehat{x}_{j},\cdots,\alpha(x_{n+1}))]'\label{terme7}\\[0.5pt]
&+&\sum_{p=i+1}^{j-1}(-1)^{p}[\phi(\alpha^{n}(x_{p}))
,\varphi([x_{i},x_{j}],\alpha(x_{0}),\cdots,\widehat{x}_{i},\cdots,
\widehat{x}_{p},\cdots,\widehat{x}_{j},\cdots,\alpha(x_{n+1}))]'\label{terme8}\\[0.5pt]
&+&\sum_{p=j+1}^{n+1}(-1)^{p-1}[\phi(\alpha^{n}(x_{p})),\varphi([x_{i},x_{j}],
\alpha(x_{0}),\cdots,\widehat{x}_{i},\cdots,\widehat{x}_{j},\cdots,\widehat{x}_{p}
,\cdots,\alpha(x_{n+1}))]'\label{terme9}\\[0.5pt]
&+&\sum_{p=1}^{i-1}(-1)^{p}\varphi([[x_{i},x_{j}],\alpha(x_{p})],
 \alpha^{2}(x_{0}),\cdots,\widehat{x}_{p},\cdots,\widehat{x}_{i},\cdots,\widehat{x}_{j},
 \cdots,\alpha^{2}(x_{n+1}))\label{terme10}\\[0.5pt]
&+&\sum_{p=i+1}^{j-1}(-1)^{p-1}\varphi([[x_{i},x_{j}],\alpha(x_{p})],\alpha^{2}(x_{0}),
 \cdots,\widehat{x}_{i},\cdots,\widehat{x}_{p},\cdots,\widehat{x}_{j},\cdots,\alpha^{2}(x_{n+1}))
 \label{terme11}
 \\[0.5pt]
&+&\sum_{p=j+1}^{n+1}(-1)^{p}\varphi([[x_{i},x_{j}],\alpha(x_{p})],
 \alpha^{2}(x_{0}),\cdots,\widehat{x}_{i},\cdots,\widehat{x}_{j},\cdots,\widehat{x}_{p},\cdots,
 \alpha^{2}(x_{n+1}))\label{terme12}\\[0.5pt]
&+&\sum_{p<q<i<j}(-1)^{p+q-1}\varphi([\alpha(x_{p}),\alpha(x_{q})],\alpha([x_{i},x_{j}]),
 \alpha^{2}(x_{0}),\cdots,\widehat{x_{i,j,p,q}},\cdots,\alpha^{2}(x_{n+1})))\label{15}
 \\[0.5pt]
&+&\sum_{i<p<q<j}(-1)^{p+q}\varphi([\alpha(x_{p}),\alpha(x_{q})],\alpha([x_{i},x_{j}]),
 \alpha^{2}(x_{0}),\cdots,\widehat{x_{i,j,p,q}},\cdots,\alpha^{2}(x_{n+1})))\label{16}\\[0.5pt]
&+&\sum_{i<j<p<q}(-1)^{p+q}\varphi([\alpha(x_{p}),\alpha(x_{q})],\alpha([x_{i},x_{j}]),
 \alpha^{2}(x_{0}),\cdots,\widehat{x_{i,j,p,q}},\cdots,\alpha^{2}(x_{n+1})))\label{17}
 %\\[0.5pt]
   \end{eqnarray}
 \begin{eqnarray}
 &+&\sum_{i<p<j<q}(-1)^{p+q-1}\varphi([\alpha(x_{p}),\alpha(x_{q})],\alpha([x_{i},x_{j}]),
 \alpha^{2}(x_{0}),\cdots,\widehat{x_{i,j,p,q}},\cdots,\alpha^{2}(x_{n+1})))\label{18}\\[0.5pt]
 &+&\sum_{p<i<q<j}(-1)^{p+q}\varphi([\alpha(x_{p}),\alpha(x_{q})],\alpha([x_{i},x_{j}]),
 \alpha^{2}(x_{0}),\cdots,\widehat{x_{i,j,p,q}},\cdots,\alpha^{2}(x_{n+1})))\label{19}\\[0.5pt]
 &+&\sum_{p<i<j<q}(-1)^{p+q-1}\varphi([\alpha(x_{p}),\alpha(x_{q})],\alpha([x_{i},x_{j}]),
\alpha^{2}(x_{0}),\cdots,\widehat{x_{i,j,p,q}},\cdots,\alpha^{2}(x_{n+1}))). \label{20}
\end{eqnarray}}
By the fact that $\alpha$ is an algebra morphism, we get
$$\sum_{0\leq i<j\leq n}(-1)^{i+j}[\eqref{15}+\eqref{17}+\eqref{16}+\eqref{20}+\eqref{19}+\eqref{18}]=0.$$
 Using Hom-Jacobi identity, the fact that $\phi\circ[\cdot, \cdot]=[\cdot, \cdot]'\circ(\phi\otimes\phi)$
and  $(\ref{condition cochain})$, we obtain
$$\sum_{i=0}^{n+1}(-1)^{i}(\eqref{terme1}+\eqref{terme2})+\sum_{0\leq i<j\leq n+1}(-1)^{i+j}\eqref{terme6}=0.$$
Hom-Jacobi identity leads also to
$\sum_{0\leq i<j\leq n+1}(-1)^{i+j}(\eqref{terme12}+\eqref{terme11}+\eqref{terme10})=0$,
and  simple calculation to  the following equality
$$\sum_{i=0}^{n+1}(-1)^{i}(\eqref{terme3}+\eqref{terme4}+\eqref{terme5})=
-\sum_{0\leq i<j\leq n+1}(-1)^{i+j}(\eqref{terme7}+\eqref{terme8}+\eqref{terme9}).$$
\end{proof}
\begin{defn}
The space of $n$-cocycles is defined by
$$Z^{n}_{HL}(\mathcal{L},\Pi)=\{\varphi\in\widetilde{C}_{\alpha,\alpha'}^{n}
(\mathcal{L},\Pi):\delta^{n}_{HL}\varphi=0\},$$
and the space of $n$-coboundaries is defined by
$$B^{n}_{HL}(\mathcal{L},\Pi)=\{\psi=\delta^{n-1}_{HL}\varphi:\varphi
\in\widetilde{C}^{n-1}_{\alpha,\alpha'}(\mathcal{L},\Pi)\}.$$
One has $B^{n}_{HL}(\mathcal{L},\Pi)\subset Z^{n}_{HL}(\mathcal{L},\Pi)$. Then, we call the $n^{th}$ cohomology group of the Hom-Lie algebra $\mathfrak{g}$ with coefficients in $\Pi$,
the quotient
$H^{n}_{HL}(\mathcal{L},\Pi)=\frac{Z^{n}_{HL}(\mathcal{L},\Pi)}{B^{n}_{HL}(\mathcal{L},\Pi)}.$
\end{defn}

%%%%%%%%%%%%%%%%%%%%%%%%%%%%%%%%%%%%%%%%%%%%%%%%%%%%%%%%%%%%%%%%%%%%%%

\section{Cohomology Complex of Hom-algebra morphisms}

%%%%%%%%%%%%%%%%%%%%%%%%%%%%%%%%%%%%%%%%%%%%%%%%%%%%%%%%%%%%%%%%%%%
In this section, we construct the module $C^{n}(\phi,\phi)$ of $n$-cochains of the
Hom-algebra morphism $\phi$ and we give an explicit formula for
the coboundary operator.
%%%%%%%%%%%%%%%%%%%%%%%%%%%%%%%%%%%%%%%%%%%%%%%%%%%%%%%%%%%%%%%%%%%%%%%%%
\subsection{Cohomology Complex of Hom-associative algebra morphisms}
%%%%%%%%%%%%%%%%%%%%%%%%%%%%%%%%%%%%%%%%%%%%%%%%%%%%%%%%%%%%%%%%%%%%%%%%%%
The original cohomology theory associated to deformation of associative
algebra morphism was introduced by M. Gerstenhaber in \cite{ms3}.
In this section, we will discuss this theory for Hom-associative algebra morphisms.\\
Let $\mathcal{A}, \mathcal{B}$ be two Hom-associative algebras and $\phi:\mathcal{A}\rightarrow\mathcal{B}$ be a Hom-associative algebra morphism.
Regard $\mathcal{B}$ as a representation of $\mathcal{A}$ via $\phi$ wherever
appropriate.\\
Define the module of $n$-cochains of $\phi$ by
\begin{equation*}
C^{n}_{Hom}(\phi,\phi)=C^{n}_{Hom}(\mathcal{A},\mathcal{A})\times C^{n}_{Hom}(\mathcal{B},\mathcal{B})
\times C^{n-1}_{\alpha,\alpha'}(\mathcal{A},\mathcal{B}).
\end{equation*}
The coboundary operator $\delta^{n}:C^{n}_{Hom}(\phi,\phi)\rightarrow C^{n+1}_{Hom}(\phi,\phi)$ is
defined by
$$\delta^{n}(\varphi_{1},\varphi_{2},\varphi_{3})=
(\delta^{n}_{Hom}\varphi_{1},\delta^{n}_{Hom}\varphi_{2},\phi\circ\varphi_{1}-\varphi_{2}\circ\phi^{\otimes n}
-\delta^{n-1}_{Hom}\varphi_{3}),$$
where $\delta^{n}_{Hom}\varphi_{1}$ and
$\delta^{n}_{Hom}\varphi_{2}$ are defined by  \eqref{hom}  and $\delta^{n}_{Hom}\varphi_{3}$
 is defined by \eqref{hom-morphism}.
\begin{thm} We have $\delta^{n+1}\circ\delta^{n}=0$. Hence
$(C^{*}_{Hom}(\phi,\phi),\delta^{n})$ is a cochain complex.
\end{thm}
\begin{proof}
The most-right component of $(\delta^{n+1}\circ\delta^{n})
(\varphi_{1},\varphi_{2},\varphi_{3})$ is
$\phi\circ(\delta^{n}_{Hom}\varphi_{1})-
(\delta^{n}_{Hom}\varphi_{2})\circ\phi-\delta^{n}_{Hom}(\phi\circ\varphi_{1}-
\varphi_{2}\circ\phi^{\otimes n}-\delta^{n-1}_{Hom}\varphi_{3}) =\phi\circ(\delta^{n}_{Hom}\varphi_{1})-
(\delta^{n}_{Hom}\varphi_{2})
\circ\phi^{\otimes n+1}-\delta^{n}_{Hom}(\phi\circ\varphi_{1})-\delta^{n}_{Hom}(\varphi_{2}\circ\phi^{\otimes n})$.
To finish the proof, one checks that $\phi\circ(\delta^{n}_{Hom}\varphi_{1})=\delta^{n}_{Hom}(\phi\circ\varphi_{1})$
and $(\delta^{n}_{Hom}\varphi_{2})\circ\phi^{\otimes n+1}=\delta^{n}_{Hom}(\varphi_{2}\circ\phi^{\otimes n})$.
Indeed,
$\phi\circ\varphi_{1}$ is defined as follows:
$(\phi\circ\varphi_{1})(x_{0},\cdots,x_{n})=\phi\circ(\varphi_{1}(x_{0},\cdots,x_{n}))$
and
$\varphi_{2}\circ\phi^{\otimes n}$  as 
$\varphi_{2}\circ\phi(x_{0},\cdots,x_{n})=\varphi_{2}\circ(\phi(x_{1}),\cdots,\phi(x_{n})).$
\end{proof}
%\begin{prop}

%\end{prop}
\begin{prop}
If $H^{n}_{Hom}(\mathcal{A},\mathcal{A})$, $H^{n}_{Hom}(\mathcal{B},\mathcal{B})$
and $H^{n-1}_{\alpha,\alpha'}(\mathcal{A},\mathcal{B})$ are all trivial then so is $H^{n}_{Hom}(\phi,\phi)$.
\end{prop}
The proof is similar to that of Proposition $3.3$ in \cite{d}.\\
We call the $n^{th}$ Hochschild cohomology group of the Hom-associative
algebra morphism $\phi:\mathcal{A}\rightarrow \mathcal{B}$,
$$H^{n}_{Hom}(\phi,\phi)=H^{n}_{Hom}(\mathcal{A},\mathcal{A})\times H^{n}_{Hom}(\mathcal{B},\mathcal{B})\times
H^{n-1}_{\alpha,\alpha'}(\mathcal{A},\mathcal{B}).$$

The corresponding cohomology modules of the cochain complex $(C^{*}_{Hom}(\phi,\phi),\delta^{n})$
are denoted by
$H^{n}_{Hom}(\phi,\phi):=H^{n}_{Hom}(C^{*}_{Hom}(\phi,\phi),\delta).$\\

The next task is to define a composition product on cochains of   Hom-associative algebra morphisms.
%%%%%%%%%%%%%%%%%%%%%%%%%%%%%%%%%%%%%%%%%%%%%%%%%%%%%%%%%%%%%%%%%%%%%%%%%%%%%%%%%%%%%
%\subsection{Composition product }
%%%%%%%%%%%%%%%%%%%%%%%%%%%%%%%%%%%%%%%%%%%%%%%%%%%%%%%%%%%%%%%%%%%%%%%%%%%%%%%%%%
Regard $\mathcal{B}$ as a $\mathcal{A}$-bimodule and let $C^{p}(\mathcal{A},\mathcal{B})$
be the vector space of all $p$-linear maps of $\mathcal{A}$ into $\mathcal{B}$.
Set $C^{p}=0$ if $p<0$ and let $C(\mathcal{A},\mathcal{B})=\bigoplus\limits_{p\in \mathbb{Z}}C^{p}(\mathcal{A},\mathcal{B})$.
\begin{defn}
Let $\phi:\mathcal{A}\longrightarrow\mathcal{B}$ be a Hom-associative algebra
morphism, $\varphi_{1}
\in C^{a}(\mathcal{B},\mathcal{B})$ and $\varphi_{2}\in C^{b}(\mathcal{A},\mathcal{B})$.
We define the composition product
$\varphi_{1}\overline{\circ}\varphi_{2} \in C^{a+b}(\mathcal{B},\mathcal{B})$ by
\begin{equation}\label{cup}
\varphi_{1}\overline{\circ}\varphi_{2}=\sum\limits_{i=0}^{b}
(-1)^{i(a-1)}\varphi_{1}(\phi(x_{1}),\cdots,\phi(x_{i-1}),\varphi_{2}
(x_{i},\cdots,x_{i+b-1}),\phi(x_{i+b}),\cdots,\phi(x_{a+b})).
\end{equation}
\end{defn}
\begin{defn}
Define the homogenous derivation $D:C(\mathcal{A},\mathcal{B})\rightarrow C(\mathcal{A},\mathcal{B})$ of degree $1$ as follows. Let
$\varphi\in C^{n}(\mathcal{A},\mathcal{B})$, we set
\begin{equation*}
  D\varphi(x_{0},\cdots,x_{n})=\sum\limits_{i=1}^{n}(-1)^{i}\varphi
  (\alpha(x_{0}),\cdots,\mu_{\mathcal{A}}(x_{i-1},x_{i}),
  \cdots,\alpha(x_{n})).
\end{equation*}
\end{defn}
\begin{defn}
The cup product is a  bilinear map $\smile :C(\mathcal{A},\mathcal{B})\times
C(\mathcal{A},\mathcal{B})\rightarrow C(\mathcal{A},\mathcal{B})$ defined by
$\varphi\smile\psi(x_{0},\cdots,x_{a+b-1})=\mu_{B}(\varphi(x_{0},\cdots,x_{a-1}),\psi(x_{a},
\cdots,x_{a+b-1}))$
for all  $\varphi\in C^{a}(\mathcal{A},\mathcal{B}),~\psi\in C^{b}(\mathcal{A},\mathcal{B})$
and for all $x_{0},\cdots,x_{a+b-1}\in\mathcal{A}$.
\end{defn}
\begin{prop}
Let $[\varphi,\psi]^{\smile}=\varphi\smile\psi-(-1)^{ab}\psi\smile\varphi$.
Then, the pair $(C(\mathcal{A},\mathcal{B}),[\cdot, \cdot]^{\smile})$ defines a graded Lie algebra.
\end{prop}

%%%%%%%%%%%%%%%%%%%%%%%%%%%%%%%%%%%%%%%%%%%%%%%%%%%%%%%%%%%%%%%%%%%%%%%%%%%%%%%%%%%%%%%
\begin{example}
%%%%%%%%%%%%%%%%%%%%%%%%%%%%%%%%%%%%%%%%%%%%%%%%%%%%%%%%%%%%%%%%%%%%%%%%%%%%%%%%%%%%%%%%
We consider a $3$-dimensional  Hom-associative algebra $A$ defined in \cite{Ms}, with respect to a basis
 $\{e_{1},e_{2},e_{3}\}$, by the multiplication $\mu_A$ and the linear map $\alpha_A$ such that
$$\small{
\begin{array}{llll}
&\mu_A(e_{1},e_{1})=ae_{1},~~~~~~&\mu_A(e_{2},e_{2})=ae_{2},\\[0.1pt]
&\mu_A(e_{1},e_{2})=\mu_A(e_{2},e_{1})=ae_{2},~~&\mu_A(e_{2},e_{3})=be_{3},\\[0.1pt]
&\mu_A(e_{1},e_{3})=\mu_A(e_{3},e_{1})=be_{3},~~&\mu_A(e_{3},e_{2})=\mu_A(e_{3},e_{3})=0,
\end{array}}
$$
$$\alpha_A(e_{1})=ae_{1},~\alpha_A(e_{2})=ae_{2},~\alpha_A(e_{3})=be_{3},$$
where $a,b,c$ are parameters.

We consider also a $2$-dimensional  Hom-associative algebra $B$ defined, with respect to a basis  $\{f_{1},f_{2}\}$,  by
the multiplication $\mu_{B}$ and the linear map $\alpha_{B}$  such that
$$ \mu_{B}(f_{1},f_{1})=f_{1},\quad
  \mu_{B}(f_{i},f_{j})=f_{2} \text{ for }\{i,j\}\not\in\{1,1\},$$
$$\alpha_B(f_1)=\beta f_1-\beta f_2, \quad \alpha_B (f_2)=0,$$
where $\beta$ is a parameter.

Let $\phi:\mathcal{A}\rightarrow \mathcal{B}$ be a  Hom-associative algebra morphism. It is
wholly determined by a set of structure constants $\lambda_{i,j}$, such  that $\phi(e_{j})=\sum\limits_{j=1}^{3}\lambda_{i,j}f_{i}$.
It turns out that it is defined as
$$
\phi(e_{1})=f_{1}-f_{2}, \quad
\phi(e_{2})=f_{1}-f_{2}, \quad
\phi(e_{3})=0,$$ with $a=\beta=1$.\\
In the following, we compute the second cohomology spaces  $H^{2}_{Hom}(\mathcal{A},\mathcal{A})$
and $H^{2}_{Hom}(\mathcal{B},\mathcal{B})$.
Let $\psi\in Z^{2}_{Hom}(\mathcal{A},\mathcal{A})$. The $2$-cocycle $\psi : \mathcal{A}\otimes\mathcal{A}\rightarrow\mathcal{A}$
is a linear map satisfying $\delta^{2}\psi(e_{i}, e_{j}, e_{k}) = 0$ and
$\psi(\alpha(e_{i}),\alpha(e_{j}))= \alpha(\psi(e_{i},e_{j}))$, for $1\leq i, j,k\leq 3$.
Therefore, with $a=1$, we get the following 2-cocycles
\begin{small}
$$\left\{\begin{array}{ll}
\psi(e_{1},e_{1})=x_{1}e_{1}+(\frac{x_{2}}{b}-x_{1})e_{2}\\
\psi(e_{2},e_{1})=\frac{x_{2}}{b}e_{2}\\
\psi(e_{3},e_{1})=bx_{1}e_{3}
\end{array}\right.\quad \left\{\begin{array}{ll}
\psi(e_{1},e_{2})=\frac{x_{2}}{b}e_{2}\\
\psi(e_{2},e_{2})=x_{3}e_{1}+x_{4}e_{2}\\
\psi(e_{3},e_{2})=-bx_{3}e_{3}
\end{array}\right.\quad\left\{\begin{array}{ll}
\psi(e_{1},e_{3})=x_{2}e_{3}\\
\psi(e_{2},e_{3})=b(x_{3}+x_{4})e_{3}e_{3}\\
\psi(e_{3},e_{3})=0
\end{array}\right.
$$
\end{small}
We prove that all the 2-cocycles are 2-coboundaries. 
Let $\varphi_{A}\in B_{Hom}^{2}(\mathcal{A},\mathcal{A})$,
then there is a $1$-cochain $f \in C^{1}_{Hom}(\mathcal{A},\mathcal{A})= Hom(\mathcal{A},\mathcal{A})$
such that $\varphi_{A}=\delta _{A}f$ and $f(\alpha(e_{i}))=\alpha(f(e_{i}))$. Set
$f(e_{1}) = x_{1}e_{1} + y_{1}e_{2};~f(e_{2})=x_{2}e_{1} + y_{2}e_{2};~f(e_{3}) =z_{3}e_{3}$ and
$
 \delta_{A} f(e_{i},e_{j})=\mu_{A}(f(e_{i}),e_{j}))+\mu_{A}
 (e_{i},f(e_{j}))-f(\mu(e_{i},e_{j}))
$.
Therefore
\begin{align*}
 & \delta_{A}f(e_{1},e_{1})=x_{1}e_{1}+y_{1}e_{2},\delta_{A}f(e_{1},e_{2})=(x_{1}+y_{1})e_{2},
\delta_{A}f(e_{1},e_{3})=b(x_{1}+y_{1})e_{3}, \\
& \delta_{A}f(e_{2},e_{1})=(x_{1}+y_{1})e_{2},\delta_{A}f(e_{2},e_{2})=-x_{2}e_{1}+(2x_{2}+y_{2})e_{2},
 \delta_{A}f(e_{2},e_{3})=b(x_{2}+y_{2})e_{3},\\
& \delta_{A}f(e_{3},e_{1})=bx_{1}e_{3},\delta_{A}f(e_{3},e_{2})=b x_{2}e_{3},
\delta_{A}f(e_{3},e_{3})=0.
\end{align*}
Using the fact
$\varphi_{A}(\alpha(e_{i}), \alpha(e_{j})) = \alpha(\varphi_{A}(e_{i}, e_{j}))$, we obtain
\begin{small}
$$\left\{\begin{array}{l}
\varphi_{A}(e_{1},e_{1})=x_{1}e_{1}+y_{1}e_{2}\\
\varphi_{A}(e_{2},e_{1})=(x_{1}+y_{1})e_{2}\\
\varphi_{A}(e_{3},e_{1})=bx_{1}e_{3}
\end{array}\right.\quad\left\{\begin{array}{l}
\varphi_{A}(e_{1},e_{2})=(x_{1}+y_{1})e_{2}\\
\varphi_{A}(e_{2},e_{2})=-x_{2}e_{1}+(2x_{2}+y_{2})e_{2}\\
\varphi_{A}(e_{3},e_{2})=bx_{2}e_{3}
\end{array}\right.\quad\left\{\begin{array}{l}
\varphi_{A}(e_{1},e_{3})=b(x_{1}+y_{1})e_{3}\\
\varphi_{A}(e_{2},e_{3})=b(x_{2}+y_{2})e_{3}\\
\varphi_{A}(e_{3},e_{3})=0
\end{array}\right.
$$ \end{small}
Therefore $H^{2}_{Hom}(\mathcal{A},\mathcal{A})=0.$
Now, we do similar computations for $\mathcal{B}$. We get
$$B^{2}_{Hom}(\mathcal{B},\mathcal{B})=\{\psi=\delta^{1}_{B}\varphi_{2}\mid \psi(f_{1},f_{1})=cf_{1}+df_{2},
\psi(f_{i},f_{j})=(c+d)f_{2},(i,j)\neq(1,1)\},$$
$$Z^{2}_{Hom}(\mathcal{B},\mathcal{B})=\{\psi\mid\psi(f_{1},f_{1})=cf_{1}+df_{2},\psi(f_{i},f_{j})=zf_{2},(i,j)\neq(1,1)\},$$
$$H^{2}_{Hom}(\mathcal{B},\mathcal{B})=\{\psi\mid\psi(f_{1},f_{1})=cf_{1}+df_{2}, \psi(f_{i},f_{j})=zf_{2}, z\neq(c+d), (i,j)\neq(1,1)\},$$
$$Z^{1}_{Hom}(\mathcal{A},\mathcal{B})=\{\phi\in C^{1}_{Hom}(\mathcal{A},\mathcal{B})
\mid\phi\circ\psi_{1}-\psi_{2}\circ\phi-\delta^{1}\phi=0\mid
\psi_{1}\in Z^{2}_{Hom}(\mathcal{A},\mathcal{A}),~\psi_{2}\in Z^{2}_{Hom}(\mathcal{B},\mathcal{B})\}.$$
%Let $\phi_{1}\in C^{1}(\mathcal{A},\mathcal{B})$ a linear map that satisfy $\phi_{1}\circ\alpha_{A}
%=\alpha_{B}\circ\phi_{1}$ and defined as follow
%$$
%\phi_{1}(e_{1})=p_{1}f_{1}-p_{1}f_{2};~\phi_{1}(e_{2})=k_{1}f_{1}-k_{1}f_{2};~\phi_{1}(e_{3})=0
%$$
Let $\phi_{1}\in H_{Hom}^{1}(\mathcal{A},\mathcal{B})$, it turns out that it  is defined by
$$\phi_{1}(e_{1})=(\frac{x_{2}}{b}-c)f_{1}+(c-\frac{x_{2}}{b})f_{2};~\phi_{1}(e_{2})=
(x_{3}+x_{4}-c)f_{1}+(c-x_{3}-x_{4})f_{2};~\phi_{1}(e_{3})=0,$$
 with $c+d=z$ and $x_{3}+x_{4}=c$.\\
%Let $\psi\in C^{1}_{Hom}(\mathcal{A},\mathcal{A})$
%$$\psi(e_{1})=u_{1}e_{1}+v_{1}e_{2};~\psi(e_{2})=u_{2}e_{1}+v_{2}e_{2};~\psi(e_{3})=z_{3}e_{3}$$
%Let $\varphi\in C^{1}_{Hom}(\mathcal{B},\mathcal{B})$
%$$\varphi(f_{1})=c_{1}f_{1};~\varphi(f_{2})=c_{1}f_{2}$$
%Then
% $B^{1}_{Hom}(\mathcal{A},\mathcal{B})=\big<\phi_{1}\big>$ where $\phi_{1}$ is defined as follow
% $$\phi(e_{1})=(u_{1}+v_{1}-c_{1})f_{1}-(u_{1}+v_{1}-c_{1})f_{2};~
%\phi_{1}(e_{2})=(u_{2}+v_{2}-c_{1})f_{1}-(u_{2}+v_{2}-c_{1})f_{2};~\phi_{1}(e_{3})=0$$

In summary, we have
\begin{align*}
& H^{2}_{Hom}(\mathcal{A},\mathcal{A})=\{0\},\\
& H^{2}_{Hom}(\mathcal{B},\mathcal{B})=\{\psi\mid\psi(f_{1},f_{1})=cf_{1}+df_{2}, \psi(f_{i},f_{j})=zf_{2}, z\neq(c+d), (i,j)\neq(1,1)\}\\
& H_{Hom}^{1}(\mathcal{A},\mathcal{B})=\{\phi_{1}\mid\phi_{1}(e_{1})=(\frac{x_{2}}{b}-c)f_{1}+(c-\frac{x_{2}}{b})f_{2}, \\
& \hspace{3cm} \phi_{1}(e_{2})=(x_{3}+x_{4}-c)f_{1}+(c-x_{3}-x_{4})f_{2};~\phi_{1}(e_{3})=0
\}.
\end{align*}
\end{example}

%%%%%%%%%%%%%%%%%%%%%%%%%%%%%%%%%%%%%%%%%%%%%%%%%%%%%%%%%%%%%%%%%%%%%%%%%%%%%%%%%%%%%%%
%%%%%%%%%%%%%%%%%%%%%%%%%%%%%%%%%%%%%%%%%%%%%%%%%%%%%%%%%%%%%%%%%%%%%%%%%%%%%%%%%%%%
\subsection{Cohomology complex of Hom-Lie algebra morphisms }
%%%%%%%%%%%%%%%%%%%%%%%%%%%%%%%%%%%%%%%%%%%%%%%%%%%%%%%%%%%%%%%%%%%%%%%%%%%%%%%%%%%%%%%%%
The original cohomology theory of  Lie algebra morphisms was  developed by Fr\'{e}gier in \cite{f}.
In this section, we will generalize this theory  to Hom-Lie algebras. We adopt the same notations as in in \cite{f}.
Consider the product  $\diamond$ defined  for  $\lambda\in \widetilde{C}^{n}_{\alpha,\alpha'}
(\mathcal{L}',\mathcal{L}')$ and $\phi\in Hom(\mathcal{L},\mathcal{L}')$, by  $\lambda\diamond\phi\in C^{n}_{HL}(\mathcal{L},\mathcal{L}')$ by
 $\lambda\diamond\phi(x_{1},\cdots,x_{n})=\lambda(\phi(x_{1}),\cdots,\phi(x_{n}))$,
where $x_{1},\cdots,x_{n}\in \mathcal{L}$.\\
Let $\phi:\mathcal{L}\rightarrow\mathcal{L}'$ be  a Hom-Lie algebra morphism.
Regard $\mathcal{L}'$ as a representation of $\mathcal{L}$ via $\phi$ wherever
appropriate.
Define the module of $n$-cochains of $\phi$ by
\begin{equation*}
C^{n}_{HL}(\phi,\phi)=C^{n}_{HL}(\mathcal{L},\mathcal{L})\times C^{n}_{HL}(\mathcal{L}',\mathcal{L}')
\times \widetilde{C}^{n-1}_{\alpha,\alpha'}(\mathcal{L},\mathcal{L}').
\end{equation*}
The coboundary operator $\delta^{n}(\phi,\phi):C^{n}_{HL}(\phi,\phi)\rightarrow C^{n+1}_{HL}(\phi,\phi)$ is
defined by 
$$\delta^{n}(\varphi_{1},\varphi_{2},\varphi_{3})=(\delta^{n}_{HL}\varphi_{1},
\delta^{n}_{HL}\varphi_{2},\delta^{n-1}_{HL}\varphi_{3}+(-1)^{n-1}(\phi\circ\varphi_{1}-
\varphi_{2}\diamond\phi)),$$
where $\delta^{n}\varphi_{1}$ and $\delta^{n}\varphi_{2}$ are  given  by formula  \eqref{alpha} and $\delta^{n}\varphi_{3}$  by formula
 \eqref{phi}.\\
\begin{lem}We have
$\delta^{n+1}\circ\delta^{n}=0$.
 Hence $(C^{n}_{HL}(\phi,\phi),\delta^{n})$ is  a cohomology complex.
\end{lem}
\begin{proof}
The most right  component of $(\delta^{n+1}\circ\delta^{n})
(\varphi_{1},\varphi_{1},\varphi_{3})$ is $(-1)^{n-1}\delta^{n}_{HL}[
\phi\circ\varphi_{1}-\varphi_{2}\diamond\phi]+(-1)^{n}
[\phi\circ\delta^{n}_{HL}(\varphi_{1})-\delta^{n}_{HL}(\varphi_{2})\diamond\phi]$.
To finish the proof, one checks that
$\delta^{n}_{HL}(\phi\circ\varphi_{1})(x_{0},\cdots,x_{n})
=-\phi\circ\delta^{n}_{HL}(\varphi_{1})(x_{0},\cdots,x_{n})$ and
$\delta^{n}_{HL}(\varphi_{2}\diamond\phi)=-\delta^{n}_{HL}(\varphi_{2})\diamond\phi$.
Indeed,\\
$\phi\circ\varphi_{1}$ is defined as
$(\phi\circ\varphi_{1})(x_{0},\cdots,x_{n})=\phi\circ(\varphi_{1}(x_{0},\cdots,x_{n}))$
and $\varphi_{2}\circ\phi^{\otimes n}$ as\\
$\varphi_{2}\diamond\phi(x_{0},\cdots,x_{n})=\varphi_{2}\circ(\phi(x_{1}),\cdots,\phi(x_{n})).$
\end{proof}
\begin{prop}
The corresponding cohomology modules of the cochain complex $(C^{*}_{HL}(\phi,\phi),\delta^{n})$
are denoted by
$$H^{n}_{HL}(\phi,\phi):=H^{n}_{HL}(C^{*}_{HL}(\phi,\phi),\delta).$$
If $H^{n}_{HL}(\mathcal{L},\mathcal{L})$, $H^{n}_{HL}(\mathcal{L}',\mathcal{L}')$
and $H^{n-1}_{HL}(\mathcal{L},\mathcal{L}')$ are all trivial then so is $H^{n}_{HL}(\phi,\phi)$.
\end{prop}
The proof is similar to that of Proposition $3.3$ in \cite{d}.\\
We call $H^{n}_{HL}(\phi,\phi)$ the $n^{th}$ Chevalley-Eilenberg cohomology group of the Hom-Lie
algebra morphism $\phi$.\\

%%%%%%%%%%%%%%%%%%%%%%%%%%%%%%%%%%%%%%%%%%%%%%%%%%%%%%%%%
%\subsection{Composition product }
%%%%%%%%%%%%%%%%%%%%%%%%%%%%%%%%%%%%%%%%%%%%%%%%%%%%%%%%%
Now, we define a composition product that lead to a structure of a graded Lie algebra on  cochains.

Regard $\mathcal{L}'$ as an $\mathcal{L}$-bimodule. Let $C^{p}(\mathcal{L},\mathcal{L}')$
be the vector space of all skew-symmetric $p$-linear maps from $\mathcal{L}$ to $\mathcal{L}'$.
Set $C^{p}(\mathcal{L},\mathcal{L}')=0$ if $p<0$ and let $C=\bigoplus\limits_{p\in Z}C^{p}(\mathcal{L},\mathcal{L}')$.
\begin{defn}
We define a product on $C$, denoted by $[\cdot, \cdot]^{\smile}$ as follows. Let
$\varphi\in C^{p}(\mathcal{L},\mathcal{L}')$ and $\psi\in C^{q}(\mathcal{L},\mathcal{L}')$, then $[\varphi,\psi]^{\smile}\in C^{p+q}(\mathcal{L},\mathcal{L}')$
is given by
$$[\varphi,\psi]^{\smile}(x_{0},\cdots,x_{p+q-1})=\sum_{\sigma\in S_{p+q}}\varepsilon(\sigma)[\varphi(x_{\sigma(0)}
,\cdots,x_{\sigma(p-1)}),\psi(x_{\sigma(p)},\cdots,x_{\sigma(p+q-1)})]'.$$
\end{defn}
This product defines on $C$ a structure of a graded Lie algebra.
\begin{defn}
Define the homogenous derivation $D:C\rightarrow C$ of degree $1$.  For
$\varphi\in C^{n}(\mathcal{A},\mathcal{B})$,  we set
\begin{equation*}
  D\varphi(x_{0},\cdots,x_{n})=\sum\limits_{i<j}(-1)^{i+j}\varphi
  ([x_{i},x_{j}],\alpha(x_{0}),\cdots,\widehat{x}_{i},\cdots,\widehat{x}_{j},\cdots,\alpha(x_{n})).
\end{equation*}
\end{defn}
%%%%%%%%%%%%%%%%%%%%%%%%%%%%%%%%%%%%%%%%%%%%%%%%%%%%%%%%%%%%%%%%%%%%%%%%%%%%%
%%%%%%%%%%%%%%%%%%%%%%%%%%%%%%%%%%%%%%%%%%%%%%%%%%%%%%%%%%%%%%%%%%%%%%%%%%%%%
\begin{example}\ 
%%%%%%%%%%%%%%%%%%%%%%%%%%%%%%%%%%%%%%%%%%%%%%%%%%%%%%%%%%%%%%%%%%%%%%%%%%%%

%%%%%%%%%%%%%%%%%%%%%%%%%%%%%%%%%%%%%%%%%%%%%%%%%%%%%%%%%%%%%%%%%%%%%%%%%%%%
We consider the $4$-dimensional Hom-Lie algebras $(\mathfrak{g}_{1},[\cdot, \cdot]_{1},\alpha_{1})$ 
 and $(\mathfrak{g}_{2},[\cdot, \cdot]_{2},\alpha_{2})$
defined, with respect to the basis $(e_{i})_{1\leq i\leq4}$ and $(f_{i})_{1\leq i\leq4}$ respectively,  by
$$\small{\left\{\begin{array}{lll} & [e_{1},e_{2}]_{1}=be_{4}\\
 &  [e_{3},e_{4}]_{1}=de_{2}\end{array}\right.
 ;\  \left\{\begin{array}{lll}
\alpha_{1}(e_{1})&=&e_{3}+ae_{4}\\
\alpha_{1}(e_{2})&=&be_{4}\\
\alpha_{1}(e_{3})&=&e_{1}+ce_{2}\\
\alpha_{1}(e_{4})&=&de_{2}\end{array}\right.;~ [f_{1},f_{2}]_{2}=df_{4}\  ; \left\{\begin{array}{lll}
\alpha_{2}(f_{1})&=&af_{1}+bf_{2}+f_{3}+cf_{4}\\
\alpha_{2}(f_{2})&=&df_{4}\\
\alpha_{2}(f_{3})&=&ef_{1}+\frac{be}{a}f_{2}\\
\alpha_{2}(f_{4})&=&0\end{array}\right.}.
$$
We compute now the second cohomology space for the first Hom-Lie algebra. The 2-cocycle space 
$Z^{2}_{HL}(\mathfrak{g}_{1},\mathfrak{g}_{1})$ is $7$-dimensional and is generated by
$$\small{\left\{\begin{array}{ll}
\psi_{1}(e_{1},e_{2})=\frac{-1+bd}{c+ad}e_{2}\\
\psi_{1}(e_{1},e_{3})=e_{2}+\frac{-ad-bcd}{d(c+ad)}f_{4}\\
\psi_{1}(e_{2},e_{3})=\frac{-1+bd}{c+ad}e_{2}\\
\psi_{1}(e_{1},e_{4})=\frac{1-bd}{c+ad}e_{4}\\
\psi_{1}(e_{2},e_{4})=0\\
\psi_{1}(e_{3},e_{4})=\frac{-1+bd}{c+ad}e_{4}
\end{array}\right.\quad \left\{\begin{array}{ll}
\psi_{2}(e_{1},e_{2})=\frac{-a-bc}{c+ad}e_{2}\\
\psi_{2}(e_{1},e_{3})=\frac{bc^{2}-a^{2}d}{d(c+ad)}e_{4}\\
\psi_{2}(e_{2},e_{3})=\frac{-a-bc}{c+ad}e_{2}-\frac{b}{d}e_{4}\\
\psi_{2}(e_{1},e_{4})=e_{2}+\frac{a+bc}{c+ad}e_{4}\\
\psi_{2}(e_{2},e_{4})=0\\
\psi_{2}(e_{3},e_{4})=\frac{-a-bc}{c+ad}e_{4}
\end{array}\right.\quad\left\{\begin{array}{ll}
\psi_{3}(e_{1},e_{2})=e_{4}\\
\psi_{3}(e_{1},e_{3})=0\\
\psi_{3}(e_{2},e_{3})=0\\
\psi_{3}(e_{1},e_{4})=0\\
\psi_{3}(e_{2},e_{4})=0\\
\psi_{3}(e_{3},e_{4})=\frac{d}{b}e_{2}
\end{array}\right.}
$$
Furthermore $H_{HL}^{2}(\mathfrak{g}_{1},\mathfrak{g}_{1})=0$.\\
For the second Hom-Lie algebra, $Z^{2}_{HL}(\mathfrak{g}_{2},\mathfrak{g}_{2})$ is
$9$-dimensional and is generated by
$$\small{\left\{\begin{array}{ll}
\psi_{1}(f_{1},f_{2})=\frac{1-a}{e}f_{2}\\
\psi_{1}(f_{1},f_{3})=-\frac{b}{a}f_{2}+\frac{ac+db}{ae}e_{4}\\
\psi_{1}(f_{2},f_{3})=f_{2}\\
\psi_{1}(f_{1},f_{4})=0\\
\psi_{1}(f_{2},f_{4})=-\frac{a}{be}f_{4}\\
\psi_{1}(f_{3},f_{4})=\frac{1}{e}f_{4}
\end{array}\right.\quad \left\{\begin{array}{ll}
\psi_{2}(f_{1},f_{2})=0\\
\psi_{2}(f_{1},f_{3})=0\\
\psi_{2}(f_{2},f_{3})=0\\
\psi_{2}(f_{1},f_{4})=f_{4}\\
\psi_{2}(f_{2},f_{4})=-\frac{a}{b}f_{4}\\
\psi_{2}(f_{3},f_{4})=0
\end{array}\right.\quad\left\{\begin{array}{ll}
\psi_{3}(f_{1},f_{2})=f_{4}\\
\psi_{3}(f_{1},f_{3})=0\\
\psi_{3}(f_{2},f_{3})=0\\
\psi_{3}(f_{1},f_{4})=0\\
\psi_{3}(f_{2},f_{4})=0\\
\psi_{3}(f_{3},f_{4})=0
\end{array}\right.\left\{\begin{array}{ll}
\psi_{4}(f_{1},f_{2})=0\\
\psi_{4}(f_{1},f_{3})=-\frac{b}{a}f_{4}\\
\psi_{4}(f_{2},f_{3})=f_{4}\\
\psi_{4}(f_{1},f_{4})=0\\
\psi_{4}(f_{2},f_{4})=0\\
\psi_{4}(f_{3},f_{4})=0
\end{array}\right.}
$$
One distinguishes the case where $e=1-a, ~Z^{2}_{HL}(\mathfrak{g}_{1},\mathfrak{g}_{1})$
is $9$-dimensional and generated by
$$\small{\left\{\begin{array}{ll}
\varphi_{1}(f_{1},f_{2})=0\\
\varphi_{1}(f_{1},f_{3})=f_{2}+\frac{b}{a}f_{2}+\frac{1}{a-1}f_{3}\\
\varphi_{1}(f_{2},f_{3})=\frac{bd+ac}{b(1-a)}f_{4}\\
\varphi_{1}(f_{1},f_{4})=0\\
\varphi_{1}(f_{2},f_{4})=0\\
\varphi_{1}(f_{3},f_{4})=0
\end{array}\right.\quad \left\{\begin{array}{ll}
\varphi_{2}(f_{1},f_{2})=0\\
\varphi_{2}(f_{1},f_{3})=0\\
\varphi_{2}(f_{2},f_{3})=\\
\varphi_{2}(f_{1},f_{4})=f_{4}\\
\varphi_{2}(f_{2},f_{4})=-\frac{a}{b}f_{4}\\
\varphi_{2}(f_{3},f_{4})=0
\end{array}\right.\quad\left\{\begin{array}{ll}
\varphi_{3}(f_{1},f_{2})=f_{2}\\
\varphi_{3}(f_{1},f_{3})=-\frac{b}{a}f_{2}\\
\varphi_{3}(f_{2},f_{3})=f_{2}+\frac{-bd-ac}{b(1-a)}f_{4}\\
\varphi_{3}(f_{1},f_{4})=0\\
\varphi_{3}(f_{2},f_{4})=\frac{-2a^{2}+a}{b(a-1)}f_{4}\\
\varphi_{3}(f_{3},f_{4})=-\frac{1}{a-1}
\end{array}\right.}$$
$$ \left\{\begin{array}{ll}
\varphi_{4}(f_{1},f_{2})=f_{4}\\
\varphi_{4}(f_{1},f_{3})=0\\
\varphi_{4}(f_{2},f_{3})=\\
\varphi_{4}(f_{1},f_{4})=0\\
\varphi_{4}(f_{2},f_{4})=0\\
\varphi_{4}(f_{3},f_{4})=0
\end{array}\right.\quad\left\{\begin{array}{ll}
\varphi_{5}(f_{1},f_{2})=0\\
\varphi_{5}(f_{1},f_{3})=f_{4}\\
\varphi_{5}(f_{2},f_{3})=\frac{a-a^{2}}{b(1-a)}f_{4}\\
\varphi_{5}(f_{1},f_{4})=0\\
\varphi_{5}(f_{2},f_{4})=0\\
\varphi_{5}(f_{3},f_{4})=0
\end{array}\right.$$
Therefore
\begin{equation*}
dim~H^{2}_{HL}(\mathfrak{g}_{2},\mathfrak{g}_{2})
= \left\{
\begin{array}{llllll}
&7 \quad &\hbox{ if } e\neq1-a &\hbox{ and generated by } \psi_{1},\psi_{2}\\
&6 \quad &\hbox{ if } e=1-a  &\hbox{ and generated by } \varphi_{1},\varphi_{2},\varphi_{3}.\\
\end{array}\right.
\end{equation*}
One refers to $a_{1},b_{1},c_{1},d_{1}$ for the parameters of $\alpha_{1}$ and
$a_{2},b_{2},c_{2},d_{2}$ for that of $\alpha_{2}$.\\
Now, let us construct the Hom-Lie algebra morphisms. Straightforward calculations lead to the following examples:

\noindent If $e=1+a_{2}$
$$\left\{\begin{array}{llll}
&\phi(e_{1})=\lambda_{1,1}f_{1}+\frac{b_{2}\lambda_{11}}{a_{2}}f_{2}-\lambda_{11}f_{3}+
\frac{(-a_{2}c_{2}-b_{2}d_{2})\lambda_{11}}{a_{2}}f_{4}\\
&\phi(e_{2})=0\\
&\phi(e_{3})=-\lambda_{11}f_{1}-\frac{\lambda_{11}b_{2}}{a_{2}}f_{2}+\lambda_{11}f_{3}+
\frac{(a_{2}c_{2}+b_{2}d_{2})\lambda_{11}}{a_{2}}f_{4}\\
&\phi(e_{4})=0
\end{array}\right.$$
If $e=1-a_{2}$
$$\left\{\begin{array}{llll}
&\phi(e_{1})=\lambda_{1,1}f_{1}+\frac{b_{2}\lambda_{11}}{a_{2}}f_{2}+\lambda_{11}f_{3}+
\frac{(a_{2}c_{2}+b_{2}d_{2})\lambda_{11}}{a_{2}}f_{4}\\
&\phi(e_{2})=0\\
&\phi(e_{3})=\lambda_{11}f_{1}+\frac{\lambda_{11}b_{2}}{a_{2}}f_{2}+\lambda_{11}f_{3}+
\frac{(a_{2}c_{2}+b_{2}d_{2})\lambda_{11}}{a_{2}}f_{4}\\
&\phi(e_{4})=0
\end{array}\right.$$
Under the condition $a_{2}=b_{2}$, the first space of cocycles $Z_{HL}^{1}(\mathfrak{g}_{1},\mathfrak{g}_{2})$
related to the second morphisms family  is generated by
$$\left\{\begin{array}{llll}
&\phi_{1}(e_{1})=\kappa f_{1}+\kappa f_{2} + \kappa f_{3}+ (c_{2}+d_{2})\kappa f_{4}\\
&\phi_{1}(e_{2})=0\\
&\phi_{1}(e_{3})=\kappa f_{1}+\kappa f_{2} + \kappa f_{3}+ (c_{2}+d_{2})\kappa f_{4}\\
&\phi_{1}(e_{4})=0
\end{array}\right.$$ where $\kappa$ is a parameter.

Then $H_{HL}^{1}(\mathfrak{g}_{1},\mathfrak{g}_{2})$ is $3$-dimensional.
%%%%%%%%%%%%%%%%%%%%%%%%%%%%%%%%%%%%%%%%%%%%%%%%%%%%%%%%%%%%%%%%%%%%%%%%%%%%
\end{example}
  \section{Deformations of Hom-associative algebra morphisms}

%%%%%%%%%%%%%%%%%%%%%%%%%%%%%%%%%%%%%%%%%%%%%%%%%%%%%%%%
%The deformation theory of associative algebra morphisms,
%is studied by Gerstenhaber and Schack in a series of papers \cite{ms1,ms2,ms3}.
%Deformation of Lie algebras morphism have been studied by Nijenhuis and Richardson \cite{ar},
%and more recently, in $2003$ by Fregier \cite{f}.
In this section, we study one-parameter formal deformations of Hom-algebra morphisms using the approach introduced by Gerstenhaber. Recall that the main idea  is to change the scalar field $\mathbb{K}$ to a formal power series ring  $\mathbb{K}[\![t]\!]$, in one variable $t$,  and the main results provide cohomological interpretations.

Let  $A[\![t]\!]$ be the set of formal power
series whose coefficients are elements of the vector space $A$, $(A[\![t]\!]$ is
obtained by extending the coefficients domain of $A$ from $\mathbb{K}$ to $\mathbb{K}[\![t]\!])$.
%%%%%%%%%%%%%%%%%%%%%%%%%%%%%%%%%%%%%%%%%%%%%%%%%%%%%%%%%%%%
\subsection{Deformation of Hom-associative algebra morphisms}
%%%%%%%%%%%%%%%%%%%%%%%%%%%%%%%%%%%%%%%%%%%%%%%%%%%%%%%%%
%\subsubsection{Infinitesimal Deformations}
%%%%%%%%%%%%%%%%%%%%%%%%%%%%%%%%%%%%%%%%%%%%%%%%%%%%%%%%%%%%
First, we recall the definition and  the main results for Hom-associative algebras, involving Hom-type Hochschild cohomology.
\begin{defn}
A one-parameter formal deformation of a Hom-associative algebra $(\mathcal{A},\mu_0,\alpha) $
is a  Hom-associative $\mathbb{K}[\![t]\!]$-algebra $(\mathcal{A}[\![t]\!],\mu_{t},\alpha)$,  where $\mu_{t}=\sum_{i\geq 0}^{N}\mu_{i}t^{i}$, which is a $\mathbb{K}[\![t]\!]$-bilinear map satisfying
the  condition
$\mu_{t}\circ (\mu_{t}\otimes \alpha)=\mu_{t}\circ (\alpha \otimes \mu_{t}).$

\end{defn}
The deformation is said to be of order $N$ if  $\mu_{t}=\sum_{i\geq 0}^N\mu_{i}t^{i}$ and infinitesimal if  $N=1$.\\
We have the following properties:
\begin{enumerate}
  \item If the second Hom-type Hochschild cohomology group, $H^{2}_{Hom}(\mathcal{A},\mathcal{A})$
  vanishes, then the algebra $\mathcal{A}$ is rigid, in the sens that every deformation is equivalent to a trivial deformation.
  \item On the other hand, if $\mu_{t}=\sum_{i}\mu_{i}t^{i}$ is a deformation of $\mathcal{A}$
   and $\mathcal{A}$ is rigid, then the $2$-cocycle $\mu_{1}$ is always a $2$-coboundary of the Hom-type
   Hochschild cohomology.
  \item The third cohomology group is the space of obstructions to  extensions of a deformation of  order
   $N$ to a deformation of order $N+1$.
\end{enumerate}

Now, we discuss deformations of Hom-associative algebra morphisms. Let $(\mathcal{A},\mu_{A},\alpha)$ and $(\mathcal{B},\mu_{B},\beta)$ be  two Hom-associative
algebras.
\begin{defn}
Let $\phi:\mathcal{A}\rightarrow \mathcal{B}$ be a Hom-associative algebra morphism.
A deformation of $\phi$ is given by  a triple $\Theta_{t}=(\mu_{A,t},\mu_{B,t},\phi_{t})$
where
\begin{itemize}
    \item $\mu_{A,t}=\sum\limits_{n\geq0}\mu_{A,n}t^{n}$ is a deformation of $\mathcal{A}$,
    \item $\mu_{B,t}=\sum\limits_{n\geq0}\mu_{B,n}t^{n}$ is a deformation of $\mathcal{B}$,
    \item $\phi_{t}:\mathcal{A}[\![t]\!]\rightarrow \mathcal{B}[\![t]\!]$ is
    a Hom-associative algebra morphism of the form $\phi_{t}=\sum\limits_{n\geq0}\phi_{n}t^{n}$,
    where each $\phi_{n} :\mathcal{A}\rightarrow \mathcal{B}$ is a $\mathbb{K}$-linear map and
    $\phi_{0}=\phi$.
\end{itemize}
\end{defn}
\begin{prop}\label{2cocycle}
The linear coefficient $\theta_{1}=(\mu_{A,1}, \mu_{B,1},\phi_{1})$,
called the infinitesimal of the deformation $\Theta_{t}$ of $\phi$,
is a $2$-cocycle in $C_{Hom}^{2}(\phi,\phi)$.
\end{prop}
\begin{proof}
The proof is similar to the that one for  dialgebras in \cite{d}
and for  associative algebras in \cite{m}.\\

 Let $\mu_{A,t}$ (resp.~$\mu_{B,t}$) be  a formal deformation of
the Hom-associative algebra $\mathcal{A}$~(resp. $\mathcal{B}$)
which is given by a $\mathbb{K}[\![t]\!]$-bilinear map
$$\mu_{A,t}:\mathcal{A}[\![t]\!]\times\mathcal{A}
[\![t]\!]\rightarrow\mathcal{A}[\![t]\!]~~(\text{resp.}~\mu_{B,t}:\mathcal{B}[\![t]\!]\times\mathcal{B}
[\![t]\!]\rightarrow\mathcal{B}[\![t]\!])$$
of the form $\mu_{*,t}=\sum_{n\geq0}\mu_{*,n}t^n$, where each
$\mu_{*,t}$ extended to be $\mathbb{K}[\![t]\!]$-bilinear map, such that  for
$x,y,z\in\mathcal{A}~(\text{resp.}~\mathcal{B})$ the following condition holds
\begin{equation}\label{deformation}
\mu_{*,t}(\mu_{*,t}(x,y),\alpha(z))=\mu_{*,t}(\alpha(x),\mu_{*,t}(y,z)),
\end{equation}
where $\ast\in\{1,2\}$ and refers to either the multiplication in $\mathcal{A}$ or $\mathcal{B}$.
The identity \eqref{deformation} is called deformation equation of the Hom-associative algebra
and may be written as
$$\sum\limits_{i\geq0,j\geq0}t^{i+j}(\mu_{*,i}(\mu_{*,j}(x,y),\alpha(z))-\mu_{*,i}
(\alpha(x),\mu_{*,j}(y,z)))=0,$$
%or
%$$\sum\limits_{s\geq 0}t^{s}\sum_{i\geq0}(\mu_{*,i}(\mu_{*,s-i}(x,y),\alpha(z))-\mu_{*,i}(\alpha(x),\mu_{*,s-i}
%(y,z)))=0,$$
which is equivalent to the following infinite system of equations
$$\sum\limits_{i=0}^{s}(\mu_{*,i}(\mu_{*,s-i}(x,y),\alpha(z))-\mu_{*,i}(\alpha(x),\mu_{*,s-i}(y,z)))=0,~~for
~s=0,1,2,\cdots$$
\begin{defn}\label{associator}
We call $\alpha$-associator the map
$$Hom(\mathcal{A}^{\times2},\mathcal{A})\times Hom(\mathcal{A}^{\times2},\mathcal{A})\rightarrow
Hom(\mathcal{A}^{\times3},\mathcal{A})
,\quad
(\mu_{i},\mu_{j})\mapsto\mu_{i}\circ_{\alpha}\mu_{j},$$
defined for all $x,y,z\in\mathcal{A}$ by
$\mu_{i}\circ_{\alpha}\mu_{j}(x,y,z)=\mu_{i}(\alpha(x),\mu_{j}(y,z))-\mu_{i}
(\mu_{j}(x,y),\alpha(z)).$
\end{defn}

By using $\alpha_{k}$-associator, the deformation equation may be written as 
\begin{equation}\label{rong}
 \sum_{i=0}^{s}\mu_{*,i}\circ_{\alpha}\mu_{*,s-i}=0,~~for~s=0,1,2,\cdots
\end{equation}
For $s=1$, we have $\mu_{*,0}\circ_{\alpha}\mu_{*,1}+\mu_{*,1}\circ_{\alpha}\mu_{*,0}=0$ which is equivalent
to $\delta^{2}_{Hom}\mu_{*,1}=0$. It turns out that $\mu_{*,1}$ is always a $2$-cocycle.\\
For $s\geq2$, the identity \eqref{rong} is equivalent to :
$
\delta^{2}_{Hom}\mu_{*,s}=-\sum_{p+q=s,p>0,q>0}\mu_{*,p}\circ_{\alpha}\mu_{*,q}.
$\\
Let $\phi$ be a Hom-associative algebra  morphism, we have the following deformation equation
$$\phi_{t}(\mu_{A,t}(x,y))=\mu_{B,t}(\phi_{t}(x),\phi_{t}(y)),$$
which may be written
$\sum\limits_{i,j\geq0}t^{i+j}\phi_{i}(\mu_{A,j}(x,y))=\sum\limits_{i,j,k}
t^{i+j+k}\mu_{B,k}(\phi_{i}(x),\phi_{j}(y))$.
It  is equivalent to the following infinite system of equations
$$\sum_{i=0}^{s}\phi_{i}(\mu_{A,s-i}(x,y))=\sum_{i,j\geq 0; \; i+j\leq s}\mu_{B,i}(\phi_{j}(x),\phi_{s-i-j}(y)),~~for
~s=0,1,2,\cdots$$
For $s=1$, we have\\
\begin{small}
 $\phi_{0}\circ\mu_{A,1}(x,y)-\mu_{B,1}(\phi_{0}(x),\phi_{0}(y))+\phi_{1}\circ\mu_{A,0}(x,y)-\mu_{B,0}(\phi_{1}(x),
 \phi_{0}(y))-\mu_{B,0}(\phi_{0}(x),\phi_{1}(y))
=0.$
\end{small}\\ It is equivalent to 
$$
\phi\circ\mu_{A,1}-\mu_{B,1}\circ(\phi\times\phi)-\delta^{1}\phi_{1}=0.$$
This means that the $2$-cochain
$\phi\circ\mu_{A,1}-\mu_{B,1}\circ(\phi\otimes\phi)-\delta^{1}\phi_{1}\in C_{Hom}^{2}(\mathcal{A},\mathcal{B})$
is equal to $0$.\\
More generally, if $\theta_{i}=0$ for $1\leq i\leq n$, then $\theta_{n+1}$ is a $2$-cocycle.
Indeed, we have the following identity
$\sum\limits_{p+q=s}\mu_{*,p}\circ_{\alpha}\mu_{*,q}=0.$\\
For $s=n+1$, we have
$$\mu_{*,n+1}\circ_{\alpha}\mu_{*,0}+\mu_{*,0}\circ_{\alpha}\mu_{*,n+1}+
\sum\limits_{p+q=n+1,p>0,q>0}\mu_{*,p}\circ_{\alpha}\mu_{*,q}=0.$$
It is equivalent to 
$$ \delta^{2}\mu_{*,n+1}=-\sum_{p+q=n+1,p>0,q>0}\mu_{*,p}\circ_{\alpha}\mu_{*,q}.$$
As well, from the deformation equation of $\phi$, we have
$\phi\circ\mu_{A,n+1}-\mu_{B,n+1}\circ(\phi\otimes\phi)-\delta^{1}\phi_{n+1}=0$.
This is equivalent to say that the $2$-cochain
$\phi\circ\mu_{A,n+1}-\mu_{B,n+1}\circ(\phi\otimes\phi)-\delta^{1}\phi_{n+1}\in
C_{Hom }^{2}(\mathcal{A},\mathcal{B})$ is equal to $0$.
\end{proof}
%%%%%%%%%%%%%%%%%%%%%%%%%%%%%%%%%%%%%%%%%%%%%%%%%%%%%%%%%%%%%%%%%%%%%%%%%%%%%%%
\subsubsection{Equivalent deformations}
%%%%%%%%%%%%%%%%%%%%%%%%%%%%%%%%%%%%%%%%%%%%%%%%%%%%%%%%%%%%%%%%%%%%%%%%%%%%%%%
Let $\mu_{A,t}$ and $\widetilde{\mu}_{A,t}$ be two formal deformations of a Hom-associative algebra $\mathcal{A}$.
A \emph{formal automorphism} $\psi_{t}:\mathcal{A}[\![t]\!]\rightarrow\mathcal{A}[\![t]\!]$ is a
power series $\psi_{t}=\sum\limits_{n\geq0}\psi_{n}t^{n}$ in which each
$\psi_{n}\in End(\mathcal{A})$
and $\psi_{0}=Id_{\mathcal{A}}$ such that
$\psi_{t}(\mu_{t,A}(x,y))=\widetilde{\mu}_{t,A}(\psi_{t}(x),\psi_{t}(y))$
and
$\psi_{t}\circ\alpha=\widetilde{\alpha}\circ\psi_{t}$
  for all $x,y\in\mathcal{A}$.  Two deformations $\mu_{\mathcal{A},t}$ and
$\widetilde{\mu}_{\mathcal{A}t}$ are said \emph{equivalent} if and only if there exists
a formal automorphism which transforms  $\mu_{A,t}$ to $\widetilde{\mu}_{A,t}$, 
 $\mu_{A,t}\rightarrow \widetilde{\mu}_{A,t}$.

\begin{defn}\label{DefMorphEquiv}
Let $\Theta_{t}=(\mu_{A,t},\mu_{B,t},\phi_{t})$ and $\widetilde{\Theta}=
(\widetilde{\mu}_{A,t},\widetilde{\mu}_{B,t},\widetilde{\phi}_{t})$ be two deformations of
a Hom-associative algebra morphism $\phi:\mathcal{A}\rightarrow\mathcal{B}$.\\
A formal automorphism $\phi_{t}:\Theta_{t}\rightarrow\widetilde{\Theta}_{t}$ is
a pair $(\psi_{A,t},\psi_{B,t})$, where $\psi_{A,t}:\mathcal{A}[\![t]\!]\rightarrow A[\![t]\!]$
and $\psi_{B,t}:\mathcal{B}[\![t]\!]\rightarrow \mathcal{B}[\![t]\!]$ are formal automorphisms such that
$\widetilde{\phi}_{t}=\psi_{B,t}~\phi_{t}~\psi^{-1}_{A,t}$.\\
 Two deformations $\Theta_{t}$ and $\widetilde{\Theta}_{t}$ are equivalent
if and only if there exists a formal automorphism such that $\theta_{t}\rightarrow\widetilde{\theta}_{t}$.
\end{defn}
Given a deformation $\Theta_{t}$ and a pair of power series
$\psi_{t}=(\psi_{A,t}=\sum\limits_{n}\psi_{A,n}t^{n},\psi_{B,t}=\sum\limits_{n}\psi_{B,n}t^{n})$,
one can define a deformation $\widetilde{\Theta}_{t}$ which
is automatically equivalent to $\Theta_{t}$.

\begin{prop}\label{deformation equivalente}\cite{aem}
If $\mu_{A,t}$ and $\widetilde{\mu}_{A,t}$ are equivalent deformations of
$\mathcal{A}$ given by the automorphism $\phi_{A,t}:\mathcal{A}[\![t]\!]\rightarrow \mathcal{A}[\![t]\!]$, then 
the infinitesimals of $\mu_{A,t}$ and $\widetilde{\mu}_{A,t}$ belong to the same cohomology class
with respect to the cohomology of the algebra $\mathcal{A}$.
\end{prop}
\begin{theorem}\label{theoreme}
The infinitesimal of a deformation $\Theta_{t}$ of $\phi$ is a $2$-cocycle in
$C_{Hom}^{2}(\phi,\phi)$ whose cohomology class is determined by the equivalence class
of the first term of $\Theta_{t}$.
\end{theorem}
\begin{proof}
In view of Proposition $\ref{2cocycle}$, it remains to show that if
$\psi_{t}:\Theta_{t}\rightarrow\Theta_{t}$ is a formal automorphism,
then the $2$-cocycle $\theta_{1}$ and $\widetilde{\theta}_{1}$
differ by a $2$-coboundary. Write $\psi_{t}=(\psi_{A,t},\psi_{B,t})$
and $\widetilde{\Theta}_{t}=(\widetilde{\mu}_{A,t},\widetilde{\mu}_{B,t},
\widetilde{\phi}_{t})$. 
In view of Proposition \ref{deformation equivalente}, we have
 \begin{eqnarray}\label{mu}
   \delta^{1}\psi_{*,1}=\mu_{*,1}-\widetilde{\mu}_{*,1}\in C_{Hom}^{2}
(\mathcal{*},\mathcal{*})
 \end{eqnarray}
for $\ast$ denoting either $A$ or $B$.
To finish the proof, we develop both sides of
$\widetilde{\phi}_{t}=\psi_{B,t}\phi_{t}\psi^{-1}_{A,t}$ and collecting the
coefficients of $t^{n}$  yield for $n=1$ the equality
\begin{equation}\label{equivalence}
  \phi_{1}-\widetilde{\phi}_{1}=\phi\psi_{A,1}-\psi_{B,1}\phi.
\end{equation}
It follows that a $1$-cochain $\alpha=(\psi_{A,1},\psi_{B,1},0)\in C_{Hom}^{1}(\phi,\phi)$
satisfies $\delta^{1}\alpha=\theta_{1}-\widetilde{\theta}_{1}$.
\end{proof}
\begin{defn}
Let $(\mathcal{A},\mu,\alpha)$ be a Hom-associative algebra and
$\theta_{1}=(\mu_{A,1},\mu_{B,1},\phi_{1})$ be an element of $Z^{2}_{Hom}(\phi,\phi)$. The
$2$-cocycle $\theta_{1}$ is said to be integrable if there exists a family
$ (\mu_{A,t}=\sum_{n}\mu_{A,n},\mu_{B,t}=\sum_{n}\mu_{B,n},\phi_{t}=\sum_{n}\phi_{n})$
defining a formal deformation
$\Theta_{t}$ of $\phi$.
\end{defn}
According to \eqref{equivalence} and \eqref{mu} identities, the integrability of $\theta_{1}$
depends only on its cohomology class. Thus, we have the  following result.

\begin{theorem}
Let $(\mathcal{A},\mu_{A},\alpha)$ and $(\mathcal{B},\mu_{B},\beta)$ be two Hom-associative
algebras and  $\Theta_{t}=(\mu_{t,A},\mu_{t,B},\phi_{t})$ be a deformation of a
Hom-associative algebra morphism $\phi:\mathcal{A}\rightarrow \mathcal{B}$.
Then,  there exists an equivalent deformation $\widetilde{\Theta}_{t}=(\widetilde{\mu}_{A,t},
\widetilde{\mu}_{B,t},\widetilde{\phi})$ such that
$\widetilde{\theta}_{1}\in Z_{Hom}^{2}(\phi,\phi)$ and $\widetilde{\theta}_{1}
\not\in B_{Hom}^{2}(\phi,\phi)$. Hence, if $H_{Hom}^{2}(\phi,\phi)=0$ then every formal deformation is
equivalent to a trivial deformation.
\end{theorem}
\begin{proof}
Define a pair of power series $\psi_{t}=(\psi_{A,t},\psi_{B,t})$.  According to Definition \ref{DefMorphEquiv}, we  define an equivalent
deformation $\widetilde{\Theta}_{t}=(\widetilde{\mu}_{A,t},\widetilde{\mu}_
{B,t},\widetilde{\phi}_{t})
=\sum\limits_{n}\widetilde{\theta}_{n}t^{n}$.\\
We have $\mu_{*,1}\in Z_{Hom}^{2}(*,*)$ and also $\mu_{*,1}-\widetilde{\mu}_{*,1}
\in Z_{Hom}^{2}(*,*)$ for $*\in\{\mathcal{A},\mathcal{B}\}$. Moreover  $\phi_{1}\in Z_{Hom}^{1}(\mathcal{A},\mathcal{B})$
leads to   $\phi_{1}-\widetilde{\phi}_{1}\in Z_{Hom}^{1}(\mathcal{A},\mathcal{B})$. If $\widetilde{\theta}_{1}
\in B_{Hom}^{2}(\phi,\phi)$ then
$\theta_{1}-\widetilde{\theta}_{1}=\delta^{1}\varphi$ for $\varphi\in C_{Hom}^{1}(\phi,\phi)$.
\end{proof}
\begin{theorem}
Let $\Theta_{t}=\sum_{i\geq0}\theta_{i}t^{i}$ be a deformation of a Hom-associative algebra morphism,
in which $\theta_{i}=0$ for $i=1,\cdots,n$ and $\theta_{n+1}$ is a coboundary in $C_{Hom}^{2}(\phi,\phi)$,
then there exists a deformation $\widetilde{\Theta}_{t}$ equivalent to $\Theta_{t}$ and
a formal automorphism $\psi_{t}:\Theta_{t}\rightarrow\widetilde{\Theta}_{t}$ such that
$\widetilde{\theta}_{i}=0$, for $i=1,\cdots,n+1$.
\end{theorem}
\begin{proof}
The proof is similar to that in \cite{d}.
\end{proof}
A Hom-associative algebras morphism for which every formal deformation is
equivalent to a trivial deformation $(\mu_{A,0},\mu_{B,0},\psi)$ is said to be analytically rigid.
The vanishing of the second cohomology group $(H^{2}_{Hom}(\phi,\phi)=0)$
gives a sufficient criterion for rigidity.
%%%%%%%%%%%%%%%%%%%%%%%%%%%%%%%%%%%%%%%%%%%
\subsubsection{Obstructions }
%%%%%%%%%%%%%%%%%%%%%%%%%%%%%%%%%%%%%%%%%%%
A deformation of order $N$ of $\phi$ is  a triple, $\Theta_{t}=(\mu_{A,t},\mu_{B,t},\phi_{t})$
satisfying $\phi_{t}(\mu_{A,t}(a,b))=\mu_{B,t}(\phi_{t}(a),\phi_{t}(b))$
or equivalently
$\sum\limits_{i=0}^{n}\phi_{i}(\mu_{A,n-i}(a,b))=\sum_{i+j+k=n}\mu_{B,i}(\phi_{j}(a),\phi_{k}(b))$, for $n\leq N$.\\
Given a deformation $\Theta_{t}$ of order $N$, it is said  extended to order $N+1$
if and only if there exists a $2$-cochain $\theta_{N+1}=(\mu_{A,N+1},\mu_{B,N+1},\phi_{N+1})
\in C_{Hom}^{2}(\phi,\phi)$ such that $\overline{\Theta}_{t}=\Theta_{t}+t^{N+1}\theta_{N+1}$
is a deformation of order $N+1$.
$\overline{\Theta}_{t}$ is called an order $N+1$ extension of $\Theta_{t}$.\\
Let $\mu_{\mathcal{A},t}=\sum\limits_{i=0}^{N+1}\mu_{i}t^{i}$
be a deformation of $\mathcal{A},$ the primary obstruction is $$-\sum\limits_{p+q=N+1,
p>0,q>0} \mu_{p}\circ_{\alpha}\mu_{q}$$ which is a $3$-cocycle.\\
The analogous primary obstruction to an infinitesimal deformation of a morphism
$\phi:\mathcal{A}\rightarrow\mathcal{B}$ is obtained by calculating the second
order term in the deformation equation of $\phi$. Then
\begin{equation*}\begin{array}{llll}
&\mu_{\mathcal{B},2}(\phi(a),\phi(b))+\mu_{\mathcal{B},1}(\phi_{1}(a),\phi(b))
+\mu_{\mathcal{B},1}(\phi(a),\phi_{1}(b))\\[1pt]
&+\mu_{\mathcal{B},0}(\phi_{1}(a),\phi_{1}(b))+\mu_{\mathcal{B},0}(\phi_{2}(a),\phi(b))
+\mu_{\mathcal{B},0}(\phi(b),\phi_{2}(b)\\[1pt]
&=\phi_{2}\mu_{\mathcal{A},0}(a,b)+\phi_{1}\mu_{1,\mathcal{A}}(a,b)+\phi
\mu_{\mathcal{A},2}(a,b)
\end{array}
\end{equation*}
giving
\begin{equation*}\begin{array}{llll}
\mu_{\mathcal{B},1}\overline{\circ}\phi_{1}-\phi_{1}\circ\mu_{\mathcal{A},1}
+\phi_{1}\smile\phi_{1}
=\phi\mu_{\mathcal{A},2}(a,b)-\mu_{\mathcal{B},2}(\phi(a),\phi(b))-\delta_{Hom}^{1}\phi_{2}.
\end{array}
\end{equation*}
Since $\mu_{*,1}\circ_{\alpha}\mu_{*,1}=-\delta_{Hom}^{2}\mu_{*,2}$, for $*\in \{\mathcal{A},\mathcal{B}\}$,
we have
\begin{equation*}
\small{(\mu_{\mathcal{A},1}\overline{\circ}\mu_{\mathcal{A},1},\mu_{\mathcal{B},1}\overline{\circ}
\mu_{\mathcal{B},1},~\mu_{1,\mathcal{B}}\widehat{\circ}\phi_{1}-\phi_{1}\circ
\mu_{\mathcal{A},1}+\phi_{1}\smile\phi_{1})=\delta^{2}(\mu_{\mathcal{A},2},\mu_{\mathcal{B},2},\phi_{2})},
\end{equation*}
where $\overline{\circ}$ is defined in  \eqref{cup produit 1} and $\widehat{\circ}$
in \eqref{cup}.\\
Set $\mathcal{O}b_{\mathcal{A}}=\sum\limits_{\begin{subarray}{l} p+q=N+1\\
p>0,q>0
\end{subarray}} \frac{1}{2}[\mu_{A,p},\mu_{A,q}]^{\Delta}_{\alpha}$ for  the obstruction of a
deformation of a Hom-associative algebra $\mathcal{A}$,
$\mathcal{O}b_{\mathcal{B}}=\sum\limits_{\begin{subarray}{l} p+q=N+1\\
p>0,q>0
\end{subarray}} \frac{1}{2}[\mu_{B,p},\mu_{B,q}]^{\Delta}_{\alpha}$ for the obstruction
of a deformation of Hom-associative algebra $\mathcal{B}$ and
$$\mathcal{O}b_{\phi}=\sum\limits_{\begin{subarray}{l} p+q=N+1\\
p>0,q>0
\end{subarray}} \mu_{B,p}\overline{\circ}\phi_{q}-
\sum\limits_{\begin{subarray}{l} p+q=N+1\\
p>0,q>0
\end{subarray}}\phi_{p}\circ\mu_{A,q}+
\sum\limits_{\begin{subarray}{l} p+q=N+1\\
p>0,q>0
\end{subarray}}\phi_{p}\smile\phi_{q}+\sum\limits_{\begin{subarray}{l} p+q=N+1\\
p>0,q>0,k>0
\end{subarray}}\mu_{B,p}\circ(\phi_{q},\phi_{k})$$
for the obstruction of the extension of the Hom-associative algebra morphism $\phi$.
\begin{thm}
Let $(\mathcal{A},\mu_{A,0},\alpha_{A})$ and $(\mathcal{B},\mu_{B,0},\alpha_{B})$
be two Hom-associative algebras.
Let $\phi:\mathcal{A}\rightarrow\mathcal{B}$ be a Hom-associative  algebra morphism
and $\Theta_{t}=(\mu_{A,t},\mu_{B,t},\phi_{t})$
be an order $k$ one-parameter formal deformation of $\phi$.\\
Then $\mathcal{O}b=(\mathcal{O}b_{\mathcal{A}},\mathcal{O}b_{\mathcal{B}},
\mathcal{O}b_{\phi})\in Z_{Hom}^{3}(\phi,\phi)$.
Therefore the deformation extends to a deformation of order $k+1$ if and only if
$\mathcal{O}b_{\phi}$ is a coboundary.
\end{thm}
\begin{proof}
We show that the 3-cochain $(\mathcal{O}b_{\mathcal{A}},
\mathcal{O}b_{\mathcal{B}},\mathcal{O}b_{\phi})$ is
a cocycle,  i.e.
$\delta^{3}(\mathcal{O}b_{\mathcal{A}},\mathcal{O}b_{\mathcal{B}},
\mathcal{O}b_{\phi})=0$, i.e. $(\delta^{3}\mathcal{O}b_{\mathcal{A}},
\delta^{3}\mathcal{O}b_{\mathcal{B}},\phi\mathcal{O}b_{\mathcal{A}}
-\mathcal{O}b_{\mathcal{B}}\phi-\delta^{2}\mathcal{O}b_{\phi})=0$. One has already that $\delta^{3}\mathcal{O}b_{\mathcal{A}}=0$ and

 $\delta^{3}\mathcal{O}b_{\mathcal{B}}=0$ then it remains to show
that $\phi\mathcal{O}b_{\mathcal{A}}-\mathcal{O}b_{\mathcal{B}}\phi-\delta^{2}
\mathcal{O}b_{\phi}=0$.

\begin{eqnarray*}
\phi\mathcal{O}b_{\mathcal{A}}&=&\phi(\sum\limits_{p+q=N+1,p>0,q>0}\mu_{A,p}\circ_{\alpha}
\mu_{A,q})\\[1pt]
&=&\phi(\sum\limits_{p+q=N+1,p>0,q>0}\mu_{A,p}\circ(\mu_{A,q}\otimes\alpha)-\mu
_{A,p}\circ(\alpha_{A}\otimes\mu_{A,q}))\\[1pt]
&=&(\sum\limits_{p+q=N+1,p>0,q>0}(\phi\circ\mu_{A,p})\circ(\mu_{A,q}
\otimes\alpha_{A})-(\phi\circ\mu_{A,p})
\circ(\alpha_{A}\otimes\mu_{A,q})),
%[10pt]
\end{eqnarray*}
\begin{eqnarray*}
\mathcal{O}b_{\mathcal{B}}\phi&=&\sum\limits_{p+q=N+1,p>0,q>0}(\mu_{B,p}\circ(\mu
_{B,q}\otimes\alpha_{B})-\mu_{B,p}\circ(\alpha_{B}\otimes\mu_{B,q}))\phi\\[1pt]
&=&\sum\limits_{p+q=N+1,p>0,q>0}\mu_{B,p}\circ(\mu_{B,q}\circ(\phi\circ\phi)\otimes\alpha_{B}\circ\phi)\\[1pt]
&-&\sum\limits_{p+q=N+1,p>0,q>0}\mu_{B,p}\circ(\alpha_{B}\circ\phi\otimes\mu_{B,q} \circ(\phi\otimes\phi)),
\end{eqnarray*}
\begin{eqnarray*}
\delta^{2}(\phi_{p}\circ\mu_{A,q})
&=&\mu_{B,0}\circ(\phi\circ\alpha_{A}\otimes\phi_{p}\circ\mu_{A,q})
-\phi_{p}\circ\mu_{A,q}\circ(\mu_{A,0}\otimes\alpha_{A})\\[1pt]
&+&\phi_{p}\circ\mu_{A,q}\circ(\alpha_{A}\otimes\mu_{A,0})-
\mu_{B,0}\circ(\phi_{p}\circ\mu_{A,q}\otimes\phi\circ\alpha_{A})\\[1pt]
&=&\mu_{B,0}\circ(\phi\otimes\phi_{p})\circ(\alpha_{A}\otimes\mu_{A,q})
-\phi_{p}\circ\mu_{A,q}\circ(\mu_{A,0}\otimes\alpha_{A})\\[1pt]
&+&\phi_{p}\circ\mu_{A,q}\circ(\alpha_{A}\otimes\mu_{A,0})-
\mu_{B,0}\circ(\phi_{p}\otimes\phi)\circ(\mu_{A,q}\otimes\alpha_{A}),
\end{eqnarray*}
\begin{eqnarray*}
\delta^{2}(\phi_{p}\smile\phi_{q})&=&
\mu_{B,0}\circ(\phi\circ\alpha_{A}\otimes\phi_{p}\smile\phi_{q})
-\phi_{p}\smile\phi_{q}\circ(\mu_{A,0}\otimes\alpha_{A})\\[1pt]
&+&\phi_{p}\smile\phi_{q}
\circ(\alpha_{A}\otimes\mu_{A,0})-\mu_{B,0}\circ(\phi_{p}\smile\phi_{q}
\otimes\phi\circ\alpha_{A})\\[1pt]
&=&\mu_{B,0}\circ(\phi\circ\alpha_{A}\otimes\mu_{B,0}\circ(\phi_{p}\otimes\phi_{q}))
-\mu_{B,0}\circ(\phi_{p}\otimes\phi_{q})\circ(\mu_{A,0}\otimes\alpha_{A})\\[1pt]
&+&\mu_{B,0}\circ(\phi_{p}\otimes\phi_{q})\circ(\alpha_{A}\otimes\mu_{A,0})
-\mu_{B,0}\circ(\mu_{B,0}\circ(\phi_{p}\otimes\phi_{q})\otimes\phi\otimes\alpha_{A})\\[1pt]
&=&\mu_{B,0}\circ(\phi\otimes\mu_{B,0})\circ(\alpha_{A}\otimes\phi_{p}\otimes\phi_{q})
-\mu_{B,0}\circ(\phi_{p}\otimes\phi_{q})\circ(\mu_{A,0}\otimes\alpha_{A})\\[1pt]
&+&\mu_{B,0}\circ(\phi_{p}\otimes\phi_{q})\circ(\alpha_{A}\otimes\mu_{A,0})
-\mu_{B,0}\circ(\mu_{B,0}\otimes\phi)\circ(\phi_{p}\otimes\phi_{q}\otimes\alpha_{A}),
\end{eqnarray*}
\begin{eqnarray*}
 \delta^{2}(\mu_{B,p}\overline{\circ}\phi_{q})
 &=&\mu_{B,0}\circ(\phi\circ\alpha_{A}\otimes\mu_{B,p}\overline{\circ}\phi_{q})
 -\mu_{B,p}\overline{\circ}\phi_{q}\circ(\mu_{A,0}\otimes\alpha_{A})\\[1pt]
 &+&\mu_{B,p}\overline{\circ}\phi_{q}\circ(\alpha_{A}\otimes\mu_{A,0})
 -\mu_{B,0}\circ(\mu_{B,p}\overline{\circ}\phi_{q}\otimes\psi\circ\alpha_{A})\\[1pt]
 &=&\mu_{B,0}\circ(\phi\circ\alpha_{A}\otimes\mu_{B,p}\circ(\phi_{q}\otimes\phi))
 +\mu_{B,0}\circ(\phi\circ\alpha_{A}\otimes\mu_{B,p}\circ(\phi\otimes\phi_{q}))\\[1pt]
 &-&\mu_{B,p}\circ(\phi_{q}\otimes\phi)\circ(\mu_{A,0}\otimes\alpha_{A})
 -\mu_{B,p}\circ(\phi\otimes\phi_{q}\circ(\mu_{A,0}\otimes\alpha_{A})\\[1pt]
 &+&\mu_{B,p}\circ(\phi_{q}\otimes\phi)\circ(\alpha_{A}\otimes\mu_{A,0})
 +\mu_{B,p}\circ(\phi\otimes\phi_{q})\circ(\alpha_{A}\otimes\mu_{A,0})\\[1pt]
 &-&\mu_{B,0}\circ(\mu_{B,p}\circ(\phi_{q}\otimes\phi)\otimes\phi\circ\alpha_{A})
 -\mu_{B,0}\circ(\mu_{B,p}\circ(\phi\otimes\phi_{q})\otimes\phi\circ\alpha_{A})\\
 &=&\mu_{B,0}\circ(\phi\otimes\mu_{B,p})\circ(\alpha_{A}\otimes\phi_{q}\otimes\phi)
 +\mu_{B,0}\circ(\phi\otimes\mu_{B,p})\circ(\alpha_{A}\otimes\phi\otimes\phi_{q})\\[1pt]
 &-&\mu_{B,p}\circ(\phi_{q}\otimes\phi)\circ(\mu_{A,0}\otimes\alpha_{A})
 -\mu_{B,p}\circ(\phi\otimes\phi_{q})\circ(\mu_{A,0}\otimes\alpha_{A})\\[1pt]
 &+&\mu_{B,p}\circ(\phi_{q}\otimes\phi)\circ(\alpha_{A}\otimes\mu_{A,0})
 +\mu_{B,p}\circ(\phi\otimes\phi_{q})\circ(\alpha_{A}\otimes\mu_{A,0})\\[1pt]
 &-&\mu_{B,0}\circ(\mu_{B,p}\otimes\phi)\circ(\phi_{q}\otimes\phi\otimes\alpha_{A})
 -\mu_{B,0}\circ(\mu_{B,p}\otimes\phi)\circ(\phi\otimes\phi_{q}\otimes\alpha_{A}),
\end{eqnarray*}
\begin{eqnarray*}
\delta^{2}(\mu_{B,p}\circ(\phi_{p}\otimes\phi_{q}))&=&
\mu_{B,0}\circ(\phi\circ\alpha_{A}\otimes\mu_{B,p}\circ(\phi_{p}\otimes\phi_{q}))
-\mu_{B,p}\circ(\phi_{p}\otimes\phi_{q})\circ(\mu_{A,0}\otimes\alpha_{A})\\[1pt]
&+&\mu_{B,p}\circ(\phi_{p}\otimes\phi_{q})
\circ(\alpha_{A}\otimes\mu_{A,0})-\mu_{B,0}\circ(\mu_{B,p}\circ(\phi_{p}\otimes\phi_{q})
\otimes\phi\circ\alpha_{A}),
\end{eqnarray*}
$\phi\mathcal{O}b_{\mathcal{A}}-\mathcal{O}b_{B}\phi-\delta^{2}(\mathcal{O}b(\phi))=0$.\\
The proof is similar to that of Lemma $6.2$ in \cite{d} by setting
$$\begin{array}{llcl}\sum'=\sum\limits_{\begin{subarray}{l} i+j=N+1\\
~~i,j>0\\~~~k=0
\end{subarray}}+\sum\limits_{\begin{subarray}{l} i+k=N+1\\
~~i,k>0\\~~~j=0
\end{subarray}}+\sum\limits_{\begin{subarray}{l} j+k=N+1\\
~~j,k>0\\~~~i=0
\end{subarray}}
\sum\limits_{\begin{subarray}{l} i+k+j=N+1\\
~~~i,k,j>0
\end{subarray}}
\end{array}$$
and
$$\widetilde{\sum}=\sum\limits_{\begin{subarray}{l}
\alpha+\beta+\gamma+\lambda+\mu=N+1\\
~~~1\leq\alpha+\beta+\gamma\leq N\\~~~~\alpha,\beta,\gamma,\lambda,\mu\geq 0
\end{subarray}}+\sum\limits_{\begin{subarray}{l}
\alpha+\beta=N+1\\
\gamma=\lambda=\mu= 0
\end{subarray}}+\sum\limits_{\begin{subarray}{l}
\alpha+\gamma=N+1\\
\beta=\lambda=\mu=0
\end{subarray}}+\sum\limits_{\begin{subarray}{l}
\beta+\gamma=N+1\\
\alpha=\lambda=\mu=0
\end{subarray}}.$$
By  Hom-associativity, we get
$$\widetilde{\sum}[\mu_{B,\lambda}(\phi_{\alpha}\circ\alpha\otimes\mu_{\mu}(\phi_{\beta}
\otimes\phi_{\gamma}))-\mu_{B,\lambda}(\mu_{B,\mu}(\phi_{\alpha}\otimes\phi_{\beta})
\otimes\phi_{\gamma}\circ\alpha))]=0.$$
One has moreover
$$\delta^{2}(\mu_{A,N+1},\mu_{B,N+1},\phi_{N+1})=\sum\limits_{\begin{subarray}{l} p+q=N+1\\
~p>0,q>0
\end{subarray}}(\phi\mathcal{O}b_{\mathcal{A}}-\mathcal{O}b_{\mathcal{B}}\phi,
\mathcal{O}b_{\phi}).$$
Then, the order $N$ formal deformation extends to an order $N+1$ formal deformation
whenever \\
$\sum\limits_{\begin{subarray}{l} p+q=N+1\\
~p>0,q>0
\end{subarray}}(\mu_{A,p}\circ_{\alpha}\mu_{A,q},
\mu_{B,p}\circ_{\alpha}\mu_{B,q},\mu_{B,p}\overline{\circ}\phi_{q}-\phi_{p}\widehat{\circ}
\mu_{A,q}+\phi_{q}\smile\phi_{p}+\sum\limits_{k>0}\mu_{B,p}\circ(\phi_{q}\otimes\phi_{k}))$
is a coboundary.\\
%\end{proof}
%%%%%%%%%%%%%%%%%%%%%%%%%%%%%%%%%%%%%%%%%%%%%%%%%%%%%%%%%%%%%%%%%%%%%%%%%%%%%%%%%%%%%%%%%%%%%%
%%%%%%%%%%%%%%%%%%%%%%%%%%%%%%%%%%%%%%%%%%%%%%%%%%%%%%%%%%%%%%%%%%%%%%%%%%%%%%%%%%%%%%%%%%%%%
%\begin{proof}
Setting
\begin{eqnarray*}
&&-\sum'\mu_{B,p}\circ(\phi_{p}\otimes\phi_{k})\circ(\mu_{A,0}\otimes\alpha_{A})
=-\sum_{\begin{subarray}{l}p,q>0\\~k=0 \end{subarray}}\mu_{B,p}\circ(\phi_{q}\circ\mu_{A,0}\otimes\phi\circ\alpha_{A})\\[1pt]
&&-\sum_{\begin{subarray}{l} p,k>0\\~q=0\end{subarray}}\mu_{B,p}\circ
(\phi\circ\mu_{A,0}\otimes\phi_{k}\circ\alpha_{A})
-\sum_{\begin{subarray}{l} q,k>0\\~p=0\end{subarray}}\mu_{B,0}\circ
(\phi_{q}\circ\mu_{A,0}\otimes\phi_{k}\circ\alpha_{A})\\[1pt]
&&-\sum_{p,q,k>0}\mu_{B,p}\circ(\phi_{p}\circ\mu_{A,0}\otimes\phi_{k}\circ\alpha_{A})
\end{eqnarray*}
and as well decomposing the sum on the right hand side of
$\sum'\mu_{B,p}\circ(\phi_{q}\circ\mu_{A,0}\otimes\phi_{k}\circ\alpha_{A})$
as follows: $$\sum'=\sum\limits_{\begin{subarray}{l}p,q>0\\~k=0 \end{subarray}}+\sum\limits_{\begin{subarray}{l} p,k>0\\~q=0\end{subarray}}
+\sum\limits_{\begin{subarray}{l} q,k>0\\~p=0\end{subarray}}+\sum\limits_{p,q,k>0},$$
we can write
\begin{eqnarray}
  \sum'\mu_{B,p}\circ(\phi_{q}\circ\mu_{A,0}\otimes\phi_{k}\circ\alpha_{A})
  \label{clair1}&=&\sum'_{\begin{subarray}{l} \alpha+\beta+\gamma=q\\ ~\alpha,\beta,\gamma\geq0\end{subarray}}
  \mu_{B,i}\circ(\mu_{\beta,\alpha}\circ(\phi_{\beta}\otimes\phi_{\gamma})\otimes\phi_{k}\circ\alpha_{A})\\[1pt]
  &+&\label{clair2}\sum'_{\begin{subarray}{l} \lambda+\mu=q\\1\leq\mu\leq q\end{subarray}}\mu_{B,i}
  (\phi_{\lambda}\circ\mu_{A,\mu}\otimes\phi_{k}\circ\alpha_{A})
\end{eqnarray}
with
\begin{eqnarray}
 \eqref{clair1}&=&-\sum_{\begin{subarray}{l} i+\alpha+\beta+\gamma+k=N+1\\~~~~~~~~k=0\\~~~~i,\beta+\alpha+\gamma>0\end{subarray}}
 \mu_{B,i}\circ(\mu_{B,\alpha}\circ(\phi_{\beta}\otimes\phi_{\gamma})\otimes\phi\circ\alpha_{A})\label{clair3}\\[1pt]
 \nonumber\label{clair4}&&-\sum_{\begin{subarray}{l}\alpha=\beta=\gamma=0\\~~~i,k>0\\i+k=N+1\end{subarray}}
 \mu_{B,i}\circ(\mu_{B,0}\circ(\phi\otimes\phi)\otimes\phi_{k}\circ\alpha_{A})\\[1pt]
 \nonumber\label{clair5}&&-\sum_{\begin{subarray}{l}k,\alpha+\beta+\gamma>0\\~~~~i=0\end{subarray}}
 \mu_{B,0}\circ(\mu_{B,\alpha}\circ(\phi_{\beta}\otimes\phi_{\gamma})\otimes\phi_{k}\circ\alpha_{A})\\[1pt]
 \nonumber&&\label{clair6}-\sum_{i,\alpha+\beta+\gamma,k>0}\mu_{B,i}\circ(\mu_{B,\alpha}\circ(\phi_{\beta}\otimes\phi_{\gamma})
 \otimes\phi_{k}\circ\alpha_{A})
\end{eqnarray}

The sum on the right hand side of \eqref{clair2} write $\sum'\limits_{\lambda+\mu=q}=\sum\limits_{\begin{subarray}{l}~~~~i,\mu>0\\i+\lambda+\mu=N+1\\~~~~k=0\end{subarray}}+
\sum\limits_{\begin{subarray}{l}\lambda+\mu+k=N+1\\~~ \mu,k>0,\lambda\geq0\end{subarray}}+\sum\limits_{\begin{subarray}{l}~~~~i,\mu,k>0\\i+\lambda+\mu+k=N+1\end{subarray}}$\\
%\begin{eqnarray}
%% \nonumber to remove numbering (before each equation)
%  \eqref{clair2}&&=\sum'_{\lambda+\mu=q}\mu_{B,i}\circ(\phi_{\lambda}\circ\mu_{A,\mu}\otimes\phi_{k}\circ\alpha_{A})\\[1pt]
% \label{stoessel1}&&=\sum_{\begin{subarray}{l}~~~~i,\mu>0\\i+\lambda+\mu=N+1\\~~~~k=0\end{subarray}}\mu_{B,i}\circ(\phi_{\lambda}\circ\mu_{A,\mu}
%  \otimes\phi\circ\alpha_{A})\\[1pt]
%  &&\label{stoessel2}+\sum_{\begin{subarray}{l}\lambda+\mu+k=N+1\\~~ \mu,k>0,\lambda\geq0\end{subarray}}
%  \mu_{B,0}\circ(\phi_{\lambda}\circ\mu_{A,\mu}\otimes\phi_{k}\circ\alpha_{A})\\[1pt]
%  &&\label{stoessel3}+\sum_{\begin{subarray}{l}~~~~i,\mu,k>0\\i+\lambda+\mu+k=N+1\end{subarray}}\mu_{B,i}\circ(\phi_{\lambda}\circ
%  \mu_{A,\mu}\otimes\phi_{k}\circ\alpha_{A})
%\end{eqnarray}
and
\begin{eqnarray}
  \eqref{clair3}=&&-\sum_{\begin{subarray}{l}~~~i,\alpha>0\\i+\alpha=N+1\end{subarray}}
  \mu_{B,i}\circ(\mu_{B,\alpha}\circ(\phi\otimes\phi)\otimes\phi\circ\alpha_{A})\label{daguere1}\\[1pt]
  \nonumber&&\label{daguere2}-\sum_{\begin{subarray}{l}i+\alpha+\beta+\gamma=N+1\\~~~~i,\beta+\gamma>0\\~~~~ \alpha,\beta,\gamma\geq0\end{subarray}}
  \mu_{B,i}\circ(\mu_{B,\alpha}\circ(\phi_{\beta}\otimes\phi_{\gamma})\otimes\phi\circ\alpha_{A}).
\end{eqnarray}
Setting
\begin{eqnarray*}
&&\sum'_{i,j,k}\mu_{B,i}\circ(\phi_{j}\circ\alpha_{A}\otimes\phi_{k}\circ\mu_{A,0})
=\sum_{\begin{subarray}{l}j+k=N+1\\~~~j,k>0\end{subarray}}\mu_{B,0}\circ(\phi_{j}\circ\alpha_{A}\otimes\phi_{k}\circ\mu_{A,0})
\\[1pt]
&&+\sum_{\begin{subarray}{l}i+k=N+1\\~~~i,k>0\end{subarray}}\mu_{B,i}\circ(\phi\circ\alpha_{A}\otimes\phi_{k}\circ\mu_{A,0})
+\sum_{\begin{subarray}{l}i+j=N+1\\~~~i,j>0\end{subarray}}\mu_{B,i}\circ(\phi_{j}\circ\alpha_{A}\otimes\phi\circ\mu_{A,0})\\[1pt]
&&+\sum_{\begin{subarray}{l}~~~i,j,k>0\\i+j+k=N+1\end{subarray}}\mu_{B,i}\circ(\phi_{j}\circ\alpha_{A}\otimes\phi_{k}\circ\mu_{A,0}),
\end{eqnarray*}
equally we have
\begin{eqnarray}
%\nonumber
 \sum'_{i+j+k=N+1}\mu_{B,i}\circ(\phi_{j}\circ\alpha_{A}\otimes\phi_{k}\circ\mu_{A,0})
 \label{mairie1}=\sum'_{i,j,\alpha,\beta,\gamma}\mu_{B,i}\circ(\phi_{j}\circ\alpha_{A}\otimes\mu_{B,\alpha}\circ(
 \phi_{\beta}\otimes\phi_{\gamma}))\\[1pt]
\label{mairie2} -\sum'_{\begin{subarray}{l}\lambda+\mu=k\\1\leq\mu\leq k\end{subarray}}
\mu_{B,i}\circ(\phi_{j}\circ\alpha_{A}\otimes\phi_{\lambda}\circ\mu_{A,\mu})
\end{eqnarray}
\begin{small}
We can write \eqref{mairie1} as
\begin{eqnarray}
  \eqref{mairie1}
  %&=&\sum'\mu_{B,i}(\phi\circ\alpha_{A}\otimes\mu_{B,\alpha}\circ(\phi_{\beta}\otimes\phi_{\gamma}))\\[1pt]
 \label{mairie11} &=&\sum_{\begin{subarray}{l}i+\alpha+\beta+\gamma=N+1\\~~~i,\alpha+\beta+\gamma>0\\ ~~~\alpha,\beta,\gamma\geq0\end{subarray}}
  \mu_{B,i}(\phi\circ\alpha_{A}\otimes\mu_{B,\alpha}\circ(\phi_{\beta}\otimes\phi_{\gamma}))
\\[1pt] &&
  \nonumber+ \label{mairie12}\sum_{\begin{subarray}{l}~~~~~~~~i=0\\ ~~\alpha+\beta+\gamma+j=N+1\\ j,\alpha+\beta+\gamma>0,\alpha,\beta,\gamma\geq0\end{subarray}}
  \mu_{B,0}\circ(\phi_{j}\circ\alpha_{A}\otimes\mu_{B,\alpha}\circ(\phi_{\beta}\otimes\phi_{\gamma}))\\[1pt]
  \nonumber&&+\label{mairie13}\sum_{\begin{subarray}{l}~~i,j>0\\ \alpha=\beta=\gamma=0\\i+j=N+1\end{subarray}}
  \mu_{B,i}\circ(\phi_{j}\circ\alpha_{A}\otimes\mu_{B,0}\circ(\phi\otimes\phi))
\\[1pt]
\nonumber&&+\label{mairie14}\sum_{\begin{subarray}{l}i+j+\alpha+\beta+\gamma=N+1\\~~~i,j,\alpha+\beta+\gamma>0\end{subarray}}
\mu_{B,i}\circ(\phi_{j}\circ\alpha_{A}\otimes\mu_{B,\alpha}\circ(\phi_{\beta}\otimes\phi_{\gamma})).
\end{eqnarray}
\end{small}
The sum on the right hand side of \eqref{mairie2} is $\sum'\limits_{\begin{subarray}{l}\lambda+\mu=k\\1\leq\mu\leq k\end{subarray}}
=\sum\limits_{\begin{subarray}{l}j+\lambda+\mu=N+1\\~~~~~i=0\\ ~~j,\mu>0,\lambda\geq0\end{subarray}}+
\sum\limits_{\begin{subarray}{l}i+\lambda+\mu=N+1,j=0\\ ~~~~~i,\mu >0\\~~~~\lambda+\mu=k\\~~~~~\lambda\geq0 \end{subarray}}
+\sum\limits_{\begin{subarray}{l}~~~\mu,i,j>0\\ \lambda+\mu=k,\lambda\geq0\end{subarray}}$
%\begin{eqnarray}
% \nonumber
%  \eqref{mairie2}&=&\sum'\mu_{B,i}\circ(\phi_{j}\circ\alpha_{A}\otimes\phi_{\lambda}\circ\mu_{A,\mu})\\[1pt]
%  \label{mairie21}&=&\sum_{\begin{subarray}{l}j+\lambda+\mu=N+1\\~~~~~i=0\\ ~~j,\mu>0,\lambda\geq0\end{subarray}}
%  \mu_{B,0}\circ(\phi_{j}\circ\alpha_{A}\otimes\phi_{\lambda}\circ\mu_{A,\mu})\\[1pt]
%  \label{mairie22}&+&\sum_{\begin{subarray}{l}i+\lambda+\mu=N+1,j=0\\ ~~~~~i,\mu >0\\~~~~\lambda+\mu=k\\~~~~~\lambda\geq0 \end{subarray}}
%  \mu_{B,i}\circ(\phi\circ\alpha_{A}\otimes\phi_{\lambda}\circ\mu_{A,\mu})\\[1pt]
%  \label{mairie23}&+&\sum_{\begin{subarray}{l}~~~\mu,i,j>0\\ \lambda+\mu=k,\lambda\geq0\end{subarray}}
%  \mu_{B,i}(\phi_{j}\circ\alpha_{A}\otimes\phi_{\lambda}\circ\mu_{A,\mu})
%\end{eqnarray}

\begin{eqnarray}
\nonumber
 \eqref{mairie11}&=&
 \sum_{\begin{subarray}{l}i+\alpha+\beta+\gamma=N+1\\~~~i,\alpha+\beta+\gamma>0\\~~~~\alpha,\beta,\gamma\geq0\end{subarray}}
  \mu_{B,i}(\phi\circ\alpha_{A}\otimes\mu_{B,\alpha}\circ(\phi_{\beta}\otimes\phi_{\gamma}))\\[1pt]
  \label{mairie111}&=&\sum_{\begin{subarray}{l}i+\alpha=N+1\\~~i,\alpha>0\end{subarray}}
  \mu_{B,i}\circ(\phi\circ\alpha_{A}\otimes\mu_{B,\alpha}\circ(\phi\otimes\phi))\\[1pt]
  \nonumber\label{mairie112}&+&\sum_{\begin{subarray}{l}i+\alpha+\beta+\gamma=N+1\\~~~~i,\beta+\gamma>0\\~~~~ \alpha,\beta,\gamma\geq0\end{subarray}}
\mu_{B,i}\circ(\phi\circ\alpha_{A}\otimes\mu_{A,\alpha}\circ(\phi_{\beta}\otimes\phi_{\gamma})).
\end{eqnarray}
The sum \eqref{daguere1}+\eqref{mairie111} is equal to $\mathcal{O}b_{B}\circ\phi$.
In other hand, we have
\begin{eqnarray}
\nonumber
  \sum_{p+q=N+1}\phi_{p}\circ\mu_{A,q}\circ(\mu_{A,0}\otimes\alpha_{A})&=&
  \sum_{\begin{subarray}{l}p+q=N+1\\~~p,q>0\end{subarray}}\phi_{i}\circ\mu_{A,q}\circ(\alpha_{A}\otimes\mu_{A,0})\\[1pt]
  \label{porte1}&+&\sum_{\begin{subarray}{l}i+j+k=N+1\\~~i,k>0,j\geq0\end{subarray}}
  \phi_{i}\mu_{A,j}\circ(\alpha_{A}\otimes\mu_{A,k})\\[1pt]
  \label{porte2}&-&\sum_{\begin{subarray}{l}p+j+k=N+1\\~~p,k>0,j\geq0\end{subarray}}
  \phi_{p}\circ\mu_{A,j}\circ(\mu_{A,k}\otimes\alpha_{A}).
\end{eqnarray}
The first sum of the right hand side vanishes with the second last sum of
$\delta^{2}(\phi_{p}\circ\mu_{A,q})$.
\begin{eqnarray}
\label{porte11} \eqref{porte1}=&&\sum_{k=1}^{N}[\sum_{\begin{subarray}{l}i+j=N+1-k\\~~~~i,j\geq0\end{subarray}}
 \phi_{i}\mu_{A,j}\circ(\alpha_{A}\otimes\mu_{A,k})]\\[1pt]
 \label{porte12}&&-\phi[\sum_{\begin{subarray}{l}j+k=N+1\\~~~j,k>0\end{subarray}}
 \mu_{A,j}\circ(\alpha_{A}\otimes\mu_{A,k})].
\end{eqnarray}

\begin{eqnarray}
  \label{porte21}\eqref{porte2}=&&\sum_{k=1}^{N}[\sum_{i+j=N+1-k}\phi_{i}\mu_{A,j}(\mu_{A,k}\otimes\alpha_{A})]\\[1pt]
  \label{porte22}&&+\phi[\sum_{\begin{subarray}{l}j+k=N+1\\~~~j,k>0\end{subarray}}\mu_{A,j}\circ(\mu_{A,k}\otimes\alpha_{A})].
\end{eqnarray}
The sum \eqref{porte12}+\eqref{porte22} is equal to $\phi\circ\mathcal{O}b_{A}$.
Furtheremore
\begin{eqnarray}
 \nonumber
  \eqref{porte11}&=&\sum_{k=1}^{N}[\sum_{\begin{subarray}{l}\beta+\alpha+\gamma+k=N+1\\ ~~~\alpha,\beta,\gamma\geq0\end{subarray}}
  \mu_{B,\alpha}(\phi_{\beta}\circ\alpha_{A}\otimes\phi_{\gamma}\circ\mu_{A,k})]\\[1pt]
  \label{porte111}&=&\sum_{\begin{subarray}{l}\gamma+k=N+1\\ ~~~\gamma,k>0\end{subarray}}\mu_{B,0}(\phi\circ\alpha_{A}\otimes\phi_{\gamma}\circ\mu_{A,k})\\[1pt]
  \label{porte112}&+&\sum'_{\begin{subarray}{l}\lambda+\mu=k\\ 1\leq\mu\leq q\\ ~~\lambda\geq 0\end{subarray}}\mu_{B,i}(\phi_{j}\circ\alpha_{A}\otimes\phi_{\lambda}\circ\mu_{A,\mu}).
\end{eqnarray}
The term \eqref{porte111} vanishes with the first sum of $\delta^{2}(\phi_{p}\mu_{A,q})$ and
\eqref{porte112} vanishes with \eqref{mairie2}.\\
Also
\begin{eqnarray*}
 \eqref{porte21}&=&\sum_{k=1}^{N}[\sum_{i+j=N+1-k}\phi_{i}\mu_{A,j}(\mu_{A,k}\otimes\alpha_{A})]=
 -\sum_{\begin{subarray}{l}k+i+j=N+1\\~~~~ k>0 \end{subarray}}\phi_{i}\mu_{A,j}(\mu_{A,k}\otimes\alpha_{A})\\[1pt]
 &=&-\sum_{\begin{subarray}{l}\beta+\gamma+k+\alpha=N+1\\~~ \alpha,\beta,\gamma\geq0,k>0\end{subarray}}
 \mu_{B,\alpha}(\phi_{\beta}\circ\mu_{A,k}\otimes\phi_{\gamma}\circ\alpha_{A})\\[1pt]
 &=&-\sum_{\begin{subarray}{l}\beta+k=N+1\\ ~~~k>0 \end{subarray}}\mu_{B,0}(\phi_{\beta}\circ\mu_{A,k}\otimes\phi\circ\alpha_{A})
 -\sum'_{\begin{subarray}{l}\lambda+\mu=j\\ 1\leq\mu\leq j\end{subarray}}
 \mu_{B,i}(\phi_{\lambda}\circ\mu_{A,\mu}\otimes\phi_{k}\circ\alpha_{A}).
\end{eqnarray*}
The second last sum vanishes with the last sum of $\delta^{2}(\phi_{p}\mu_{A,q})$ and
the last sum vanishes with \eqref{clair2}.\\
The remaining terms can be written in the form of one sum as follows
$$\widetilde{\sum}[\mu_{B,\lambda}(\phi_{\alpha}\circ\alpha\otimes\mu_{\mu}(\phi_{\beta}
\otimes\phi_{\gamma}))-\mu_{B,\lambda}(\mu_{B,\mu}(\phi_{\alpha}\otimes\phi_{\beta}).
\otimes\phi_{\gamma}\circ\alpha))]$$
By  Hom-associativity, it equals $0$.
\end{proof}
%%%%%%%%%%%%%%%%%%%%%%%%%%%%%%%%%%%%%%%%%%%%%%%%%%%%%%%%%%%%%%%%%%%%%%%%%%%%%%%%%%%%%%%%%%%%%%%%%%%%%%%%%%%%
%%%%%%%%%%%%%%%%%%%%%%%%%%%%%%%%%%%%%%%%%%%%%%%%%%%%%%%%%%%%%%%%%%%%%%%%%%%%%%%%%%%%%%%%%%%%%%%%%%%%%%%%%%%%%
\begin{corollary}
If $H^{3}_{Hom}(\phi,\phi)=0$, then every infinitesimal deformation can be extended to a
formal deformation of larger order.
\end{corollary}
%%%%%%%%%%%%%%%%%%%%%%%%%%%%%%%%%%%%%%%%%%%%%%%%%%%%%%%%%%%%%%%
\subsection{Deformations of Hom-Lie algebra morphisms}
%%%%%%%%%%%%%%%%%%%%%%%%%%%%%%%%%%%%%%%%%%%%%%%%%%%%%%%%%%%%%%%
In this section, we discuss deformations of Hom-Lie algebra morphisms.

\subsubsection{Infinitesimal Deformations}
%%%%%%%%%%%%%%%%%%%%%%%%%%%%%%%%%%%%%%%%%%%%%%%%%%%%%%%%%%%%%%%%
\begin{defn}
Let $(\mathcal{L},[\cdot, \cdot],\alpha)$ be a Hom-Lie algebra. A one-parameter formal
Hom-Lie deformation of $\mathcal{L}$ is given by a $\mathbb{K}[\![t]\!]$-bilinear map
$[\cdot, \cdot]_{t}:\mathcal{L}[\![t]\!]\times\mathcal{L}[\![t]\!]\rightarrow\mathcal{L}[\![t]\!]$
of the form
$$[\cdot, \cdot]_{t}=\sum\limits_{i\geq0}t^{i}[\cdot, \cdot]_{i}$$
where each $[\cdot, \cdot]_{i}$ is a bilinear map $[\cdot, \cdot]_{i}:\mathcal{L}\times\mathcal{L}\rightarrow\mathcal{L}$
(extended to  $\mathbb{K}[\![t]\!]$-bilinear map), $ [\cdot, \cdot]=[\cdot, \cdot]_{0}$ and satisfying the following
conditions
$$
[x,y]_{t}=-[y,x]_{t} ~\hbox{for all} ~x,y\in \mathcal{L}\quad \hbox{(skew-symmetry)},$$
$$\circlearrowleft_{x,y,z}[\alpha(x),[y,z]_{t}]_{t}=0 \quad \hbox{(Hom-Jacobi identity)}$$
\end{defn}
\begin{defn}
Let $\phi:\mathcal{L}\rightarrow \mathcal{L}'$ be a Hom-Lie algebra morphism.
A deformation of $\phi $ is to  a triple $\Theta_{t}=([\cdot, \cdot]_{t};[\cdot, \cdot]_{t}';\phi_{t})$
in which :
\begin{itemize}
    \item $[\cdot, \cdot]_{t}=\sum\limits_{i\geq0}t^{i}[\cdot, \cdot]_{i}$ is a deformation of $\mathcal{L}$,
    \item $[\cdot, \cdot]_{t}'=\sum\limits_{i\geq0}t^{i}[\cdot, \cdot]_{i}'$ is a deformation of $\mathcal{L}'$,
    \item $\phi_{t}:\mathcal{L}[\![t]\!]\rightarrow \mathcal{L}'[\![t]\!]$ is a
    Hom-Lie algebra morphism of the form $\phi_{t}=\sum\limits_{n\geq0}\phi_{n}t^{n}$
    where each $\phi_{n} :\mathcal{L}\rightarrow \mathcal{L}'$ is a $\mathbb{K}$-linear map and
    $\phi_{0}=\phi$.
\end{itemize}
\end{defn}
\begin{prop}\label{2cocycles}
The linear coefficient, $\theta_{1}=([\cdot, \cdot]_{1}, [\cdot, \cdot]'_{1},\phi_{1})$,
which is called the infinitesimal of the deformation $\Theta_{t}$,
is a $2$-cocycle in $C^{2}_{HL}(\phi,\phi)$.
\end{prop}
\begin{proof}
%The demonstration is well treated for the first component of the infinitesimal deformation
%$[\cdot, \cdot]_{1}$ (resp.$[\cdot, \cdot]'_1$).
Let $\phi$ be a Hom-Lie algebra morphism, we have the following deformation equation
$$[\phi_{t}(x),\phi_{t}(y)]_{t}'=\phi_{t}([x,y]_{t}),$$
which may be written
$\sum\limits_{i,j\geq0}t^{i+j}\phi_{i}([x,y]_{j})=\sum\limits_{i,j,k\geq 0}
t^{i+j+k}[\phi_{i}(x),\phi_{j}(y)]'_{k}.$
It  is equivalent to the following infinite system of equations
$$\sum_{i=0}^{s}\phi_{i}([x,y]_{s-i}(x,y)=\sum_{i,j\geq0; \; i+j\leq s}[\phi_{i}(x),\phi_{j}(y)]'_{s-i-j},~~for
~s=0,1,2,\cdots$$
For $s=1$, we have
 $$\phi([x,y]_{1})-[\phi(x),\phi(y)]'_{1}-
(-\phi_{1}([x,y]_{0})+[\phi_{1}(x),\phi(y)]'_{0}+ [\phi(x),\phi_{1}(y)]'_{0})=0.$$
%that is $$
%\delta^{1}\phi_{1}(x,y)-\phi([x,y]_{1})+[\phi(x),\phi(y)]'_{1}=0.$$
This is equivalent  to the $2$-cochain
$\delta^{1}\phi_{1}(x,y)-\phi([x,y]_{1})+[\phi(x),\phi(y)]'_{1}
%\in C^{2}(\mathcal{L},\mathcal{L}')
$
 equals  $0$.
\end{proof}
%%%%%%%%%%%%%%%%%%%%%%%%%%%%%%%%%%%%%%%%%%%%%%%%%%%%%%%%%%%%%%%%%%%%%%%%%%%
\subsubsection{Equivalent Deformations}
%%%%%%%%%%%%%%%%%%%%%%%%%%%%%%%%%%%%%%%%%%%%%%%%%%%%%%%%%%%%%%%%%%%%%%%%%%%%%%%%%%
Let $(\mathcal{L},[\cdot, \cdot],\alpha)$ be a multiplicative Hom-Lie algebra. Let  $\mathcal{L}_{t}=
(\mathcal{L},[\cdot, \cdot]_{t},\alpha)$ and $\mathcal{L}'_{t}=(\mathcal{L},[\cdot, \cdot]_{t}',\alpha)$ be  two deformations
of $\mathcal{L}$, where $[\cdot, \cdot]_{t}=\sum_{i\geq0}t^{i}[\cdot, \cdot]_{i}$
and $[\cdot, \cdot]_{t}'=\sum_{i\geq0}t^{i}[\cdot, \cdot]'_{i}$ with $[\cdot, \cdot]_{0}=[\cdot, \cdot]_{0}'=[\cdot, \cdot]$.
We say that $\mathcal{L}_{t}$ and $\mathcal{L}'_{t}$ are equivalent if there exists a formal
automorphism $\psi_{t}:\mathcal{L}[\![t]\!]\rightarrow \mathcal{L}[\![t]\!]$, that may be written in the form $\psi_{t}=\sum_{i\geq0}
\psi_{i}t^{i}$ where $\psi_{i}\in End(\mathcal{L})$ and $\psi_{0}=id$, such that
$\psi_{t}([x,y]_{t})=[\psi_{t}(x),\psi_{t}(y)]'_{t}.$\\
A deformation $\mathcal{L}_{t}$ is said to be trivial if and only if $\mathcal{L}_{t}$
is equivalent to $\mathcal{L}$.
\begin{defn}Let $(\mathcal{L},[\cdot, \cdot]_{\mathcal{L}},\alpha)$, $(\mathcal{G},[\cdot, \cdot]_{\mathcal{G}},\beta)$ be two Hom-Lie
algebras and $\phi:\mathcal{L}\rightarrow\mathcal{G}$ be a Hom-Lie algebra mprphism. 
Let $\Theta_{t}=([\cdot, \cdot]_{\mathcal{L},t},[\cdot, \cdot]_{\mathcal{G},t},\phi_{t})$ and $\widetilde{\Theta}=
([\cdot, \cdot]_{\mathcal{L},t}',[\cdot, \cdot]_{\mathcal{G},t}',\widetilde{\phi}_{t})$ be two deformations of
a Hom-Lie algebra morphism $\phi$.\\
A formal automorphism $\psi_{t}:\Theta_{t}\rightarrow\widetilde{\Theta}_{t}$ is
a pair $(\psi_{\mathcal{L},t},\psi_{\mathcal{G},t})$, where $\psi_{\mathcal{L},t}:
\mathcal{L}[\![t]\!]\rightarrow \mathcal{L}[\![t]\!]$
and $\psi_{\mathcal{G},t}:\mathcal{G}[\![t]\!]\rightarrow \mathcal{G}[\![t]\!]$ are
formal automorphisms, such that
$\widetilde{\phi}_{t}=\psi_{\mathcal{L},t}\circ \phi_{t}\circ \psi^{-1}_{\mathcal{G},t}$.\\
$\bullet$ Two deformation $\Theta_{t}$ and $\widetilde{\Theta}_{t}$ are equivalent
if and only if there exists a formal automorphism  that transforms $\theta_{t}$ in to $\widetilde{\theta}_{t}$,  $\theta_{t}\rightarrow\widetilde{\theta}_{t}$.\\
$\bullet$ Given a deformation $\Theta_{t}$ and a pair of power series
$\psi_{t}=(\psi_{\mathcal{L},t}=\sum\limits_{n}\psi_{\mathcal{L},n}t^{n},\psi_{\mathcal{G},t}
=\sum\limits_{n}\psi_{\mathcal{G},n}t^{n})$,
one can define a deformation $\widetilde{\Theta}_{t}$. The deformation $\widetilde{\Theta}_{t}$
is automatically equivalent to $\Theta_{t}$.
\end{defn}
\begin{prop}\label{deformation equivalentes}
If $[\cdot, \cdot]_{t}$ and $[\cdot, \cdot]_t'$ are equivalent deformations of
$\mathcal{L}$ given by the automorphism $\psi_{t}:\mathcal{L}[\![t]\!]\rightarrow \mathcal{L}[\![t]\!]$,
the infinitesimals of $[\cdot, \cdot]_t$ and $[\cdot, \cdot]_t'$ belong to the same cohomology class.
\end{prop}
\begin{theorem}
The infinitesimal of a deformation $\Theta_{t}$ of $\phi$ is a $2$-cocycle in
$C^{2}_{HL}(\phi,\phi)$ whose cohomology class is determined by the equivalence class
of the first term of $\Theta_{t}$.
\end{theorem}
\begin{proof}
Same proof as for Theorem \ref{theoreme}.
\end{proof}
\begin{defn}
Let $(\mathcal{L},[\cdot, \cdot]_{\mathcal{L}},\alpha)$ and $(\mathcal{G},[\cdot, \cdot]_{\mathcal{G}},\beta)$ be
two Hom-Lie algebras, and $\phi_{1}$ be an element of $Z^{1}_{HL}(\mathcal{L},\mathcal{G})$, the $1$-cocycle
is said to be integrable if there exists a family $(\phi_{t})_{t\geq 0}$ such that $\phi_{t}=\sum_{i\geq 0}
t^{i}\phi_{i}$ defines a formal deformation $\phi_{t}$ of $\phi$.
\end{defn}
\begin{theorem}
Let $(\mathcal{L},[\cdot, \cdot]_{\mathcal{L}},\alpha)$ and $(\mathcal{G},[\cdot, \cdot]_{\mathcal{G}},\beta)$ be two Hom-Lie
algebras. Let $\Theta_{t}=([\cdot, \cdot]_{\mathcal{L},t},[\cdot, \cdot]_{\mathcal{G},t},\phi_{t})$ be a deformation of a
Hom-Lie algebra morphism $\phi$.
Then there exists an equivalent deformation $\widetilde{\Theta}_{t}=([\cdot, \cdot]_{\mathcal{L},t}',
[\cdot, \cdot]_{\mathcal{G},t},\widetilde{\phi})$ such that
$\widetilde{\theta}_{1}\in Z_{HL}^{2}(\phi,\phi)$ and $\widetilde{\theta}_{1}
\not\in B_{HL}^{2}(\phi,\phi)$. Hence, if $H_{HL}^{2}(\phi,\phi)=0$ then every formal deformation is
equivalent to a trivial deformation.
\end{theorem}
%%%%%%%%%%%%%%%%%%%%%%%%%%%%%%%%%%%%%%%%%%%%%%%%%%%%%%%%%%%%%%%%%%%%%%%%%%%%%%%%%%%%%%
\subsubsection{Obstructions}
%%%%%%%%%%%%%%%%%%%%%%%%%%%%%%%%%%%%%%%%%%%%%%%%%%%%%%%%%%%%%%%%%%%%%%%%%%%%%%%%%%%%
A deformation of order $N$ of $\phi$ is  a triple, $\Theta_{t}=([\cdot, \cdot]_{\mathcal{L},t}
,[\cdot, \cdot]_{\mathcal{G},t},\phi_{t})$
satisfying  $\phi_{t}([x,y]_{\mathcal{L},t})=[\phi_{t}(x),\phi_{t}(y)]_{\mathcal{G},t}$
or equivalently
$\sum\limits_{i=0}^{N}\phi_{i}([x,y]_{\mathcal{L},N-i})=\sum_{i+j+k=N}[\phi_{i}(x),
\phi_{j}(y)]_{\mathcal{G},k}.$\\
Given a deformation $\Theta_{t}$ of order $N$, it  extends to a deformation of order $N+1$
if and only if there exists a $2$-cochain $\theta_{N+1}=([\cdot, \cdot]_{\mathcal{L},N+1},[\cdot, \cdot]_
{\mathcal{G},N+1},\phi_{N+1})
\in C_{HL}^{2}(\phi,\phi)$ such that $\overline{\Theta}_{t}=\Theta_{t}+t^{N+1}\theta_{N+1}$
is a deformation of order $N+1$.  Then
$\overline{\Theta}_{t}$ is said to be  an extension of $\Theta_{t}$ of order $N+1$.\\
Let  $\mathcal{O}b_{\mathcal{L}}=\frac{1}{2}\sum\limits_{p+q=N+1,p>0,q>0}[[\cdot, \cdot]_{\mathcal{L},p},
[\cdot, \cdot]_{\mathcal{L},q}]_{\alpha}^{\wedge}$ (resp.
$\mathcal{O}b_{\mathcal{G}}=\frac{1}{2}\sum\limits_{p+q=N+1,p>0,q>0}[[\cdot, \cdot]_{\mathcal{G},p},[\cdot, \cdot]_
{\mathcal{G}q,}]_{\alpha}^{\wedge}]$)  be
the obstruction of a deformation of the Hom-Lie algebra $\mathcal{L}$ (resp. $\mathcal{G}$ ) and\\
$\mathcal{O}b_{\phi}=\sum_{i=0}^{N+1}
\phi_{i}([x,y]_{\mathcal{L},N+1-i})-\sum'[\phi_{i}(x),\phi_{j}(y)]_{\mathcal{G},k}$
 be the obstruction of the extension of the
Hom-Lie algebra morphism $\phi$, where
$$\begin{array}{llcl}\sum'=\sum\limits_{\begin{subarray}{l} i+j=N+1\\
~~i,j>0\\~~~k=0
\end{subarray}}+\sum\limits_{\begin{subarray}{l} i+k=N+1\\
~~i,k>0\\~~~j=0
\end{subarray}}+\sum\limits_{\begin{subarray}{l} j+k=N+1\\
~~j,k>0\\~~~i=0
\end{subarray}}
\sum\limits_{\begin{subarray}{l} i+k+j=N+1\\
~~~i,k,j>0
\end{subarray}}.
\end{array}$$
\begin{thm}
Let $(\mathcal{L},[\cdot, \cdot]_{\mathcal{L},0},\alpha)$ and $(\mathcal{G},[\cdot, \cdot]_{\mathcal{G},0},\beta)$ be
two Hom-Lie algebras and $\phi:\mathcal{L}\rightarrow \mathcal{G}$ be  a Hom-Lie algebra morphism. Let $\Theta_{t}=([\cdot, \cdot]_{\mathcal{L},t},[\cdot, \cdot]_{\mathcal{G},t},\phi_{t})$ be an order
$k$ one-parameter formal deformation of $\phi$. Then $\mathcal{O}b=(\mathcal{O}b_{\mathcal{A}},\mathcal{O}b_{\mathcal{B}},
\mathcal{O}b_{\phi})\in Z_{HL}^{3}(\phi,\phi)$.\\
Therefore, the deformation extends to a deformation of order $k+1$ if and only if
$\mathcal{O}b$ is a coboundary.
\end{thm}
\begin{proof}
We must show that $(\mathcal{O}b_{\mathcal{L}},
\mathcal{O}b_{\mathcal{G}},\mathcal{O}b_{\phi})$ is
a cocycle in $C_{HL}^{3}(\phi,\phi)$ that is
$\delta^{3}(\mathcal{O}b_{\mathcal{L}},\mathcal{O}b_{\mathcal{G}},
\mathcal{O}b_{\phi})=0$ i.e. $(\delta^{3}\mathcal{O}b_{\mathcal{L}},
\delta^{3}\mathcal{O}b_{\mathcal{G}},\delta^{2}\mathcal{O}b_{\phi}+
\phi\mathcal{O}b_{\mathcal{L}}-\mathcal{O}b_{\mathcal{G}}\phi)=0$.
One has  $\delta^{3}\mathcal{O}b_{\mathcal{L}}=0$ and
$\delta^{3}\mathcal{O}b_{\mathcal{G}}=0$,  then it remains to show
that $\delta^{2}\mathcal{O}b_{\phi}+\phi\mathcal{O}b_{\mathcal{L}}
-\mathcal{O}b_{\mathcal{G}}\phi=0$.\\
We have $$\small{(-\sum\limits_{p+q=s,p,q>0}\circlearrowleft_{x,y,z}[\alpha(x),[y,z]_{p}]_{q}
=\frac{1}{2}\sum\limits_{p+q=s,p>0,q>0}[[\cdot, \cdot]_{p},[\cdot, \cdot]_{q}]_{\alpha}^{\wedge}(x,y,z))}$$
\begin{eqnarray}
&\nonumber\delta^{2}\left(\sum_{i=1}^{N}
\phi_{i}([\cdot, \cdot]_{\mathcal{L},N+1-i}-\sum'[\phi_{i},\phi_{j}]_{\mathcal{G},k})\right)(x,y,z)
%\\[1pt]
\nonumber\label{1}=-[\phi(\alpha(x)),\sum'[\phi_{i}(y),\phi_{j}(z)]_{\mathcal{G},k}]_{\mathcal{G},0}\\[1pt]
\nonumber\label{2}&+[\phi(\alpha(y)),\sum'[\phi_{i}(x),\phi_{j}(z)]_{\mathcal{G},k}]_{\mathcal{G},0}
%\\[1pt]
\nonumber\label{3}-[\phi(\alpha(z)),\sum'[\phi_{i}(x),\phi_{j}(y)]_{\mathcal{G},k}]_{\mathcal{G},0}\\[1pt]
\label{4}&+\sum'[\phi_{i}([x,y]_{\mathcal{L},0}),\phi_{j}(\alpha(z))]_{\mathcal{G},k}\\[1pt]
\label{5}&-\sum'[\phi_{i}([x,z]_{\mathcal{L},0}),\phi_{j}(\alpha(y))]_{\mathcal{G},k}\\[1pt]
\label{6}&+\sum'[\phi_{i}([y,z]_{\mathcal{L},0}),\phi_{j}(\alpha(x))]_{\mathcal{G},k}\\[1pt]
\label{7}&+[\phi(\alpha(x)),\sum_{i=1}^{N}\phi_{i}([y,z]_{\mathcal{L},N+1-i}))]_{\mathcal{G},0}\\[1pt]
\label{8}&-[\phi(\alpha(y)),\sum_{i=1}^{N}\phi_{i}([x,z]_{\mathcal{L},N+1-i}))]_{\mathcal{G},0}\\[1pt]
\label{9}&+[\phi(\alpha(z)),\sum_{i=1}^{N}\phi_{i}([x,y]_{\mathcal{L},N+1-i}))]_{\mathcal{G},0}\\[1pt]
\label{10}&-\sum_{i=1}^{N}\phi_{i}([[x,y]_{\mathcal{L},0},\alpha(z)]_{\mathcal{L},N+1-i})\\[1pt]
\label{11}&+\sum_{i=1}^{N}\phi_{i}([[x,z]_{\mathcal{L},0},\alpha(y)]_{\mathcal{L},N+1-i})\\[1pt]
\label{12}&-\sum_{i=1}^{N}\phi_{i}([[y,z]_{\mathcal{L},0},\alpha(x)]_{\mathcal{L},N+1-i}).
\end{eqnarray}
Using the fact that
$$\begin{array}{lcl}
\phi_{i}([x,y]_{\mathcal{L},0})&=&\sum\limits_{\begin{subarray}{l} \beta+\alpha+\gamma=i\\
\alpha,\beta,\gamma>0\end{subarray}}
[\phi_{\beta}(x),\phi_{\gamma}(z)]_{\mathcal{G},\alpha}
-\sum\limits_{\begin{subarray}{l} \lambda+\mu=i\\
1\leq\mu\leq i\end{subarray}}\phi_{\lambda}([x,z]_{\mathcal{L},\mu}),
\end{array}$$
then \eqref{5} becomes
\begin{eqnarray}
  \label{A}-\sum'[\phi_{i}([x,z]_{\mathcal{L},0}),\phi_{j}(\alpha(y))]_{\mathcal{G},k}&=&
 -\sum'_{\begin{subarray}{l}\alpha+\beta+\gamma=i\\ \alpha,\beta,\gamma\geq0
  \end{subarray}}[[\phi_{\beta}(x),\phi_{\gamma}(z)]_{\mathcal{G},\alpha},\phi_{j}(\alpha(y))]_{\mathcal{G},k}\\[1pt]
 \label{B}&+&\sum'_{\begin{subarray}{l}\lambda+\mu=i\\ 1\leq\mu\leq i\end{subarray}}
  [\phi_{\lambda}([x,z]_{\mathcal{L},\mu}),\phi_{j}(\alpha(y))]_{\mathcal{G},k},
\end{eqnarray}
where the sum on the right hand side of \eqref{A} is\\
\begin{equation}\label{somme1}\sum'_{\begin{subarray}{l}\alpha+\beta+\gamma=i\\ \alpha,\beta,\gamma\geq0
\end{subarray}}=\sum\limits_{\begin{subarray}{l}\alpha+\beta+\gamma,k>0\\ ~\alpha,\beta,\gamma\geq0\\~~~j=0
\end{subarray}}+\sum\limits_{\begin{subarray}{l}\alpha+\beta+\gamma,j>0\\ ~\alpha,\beta,\gamma\geq0\\
~~~k=0\end{subarray}}
+\sum\limits_{\begin{subarray}{l}\alpha=\beta=\gamma=0\\ ~~j,k>0
\end{subarray}}
+\sum\limits_{\begin{subarray}{l}\alpha+\beta+\gamma,k,j>0\\ ~~~\alpha,\beta,\gamma\geq0
\end{subarray}}\end{equation}
and \eqref{B} is of the form
\begin{equation*}
%\label{somme2}
\sum'_{\begin{subarray}{l}\lambda+\mu=i\\ 1\leq\mu\leq i\end{subarray}}=
\sum\limits_{\begin{subarray}{l}\lambda+\mu+k=N+1\\~k,\mu>0,\lambda\geq0\\ ~~~j=0\end{subarray}}
+\sum\limits_{\begin{subarray}{l}\lambda+\mu+j=N+1\\ ~j,\mu>0,\lambda\geq0\\ ~~~k=0\end{subarray}}
+\sum\limits_{\begin{subarray}{l}\lambda+\mu+j+k=N+1\\ ~j,\mu,k>0,\lambda\geq0\end{subarray}}.
\end{equation*}
The first term of \eqref{somme1} can be written as
\begin{align}
&&-\sum\limits_{\begin{subarray}{l}\alpha+\beta+\gamma,k>0\\ ~\alpha,\beta,\gamma\geq0\\~~~j=0
\end{subarray}}[[\phi_{\beta}(x),\phi_{\gamma}(z)]_{\mathcal{G},\alpha},\phi(\alpha(y))]_{\mathcal{G},k}
%\\[1pt]
\label{*}=-\sum\limits_{\begin{subarray}{l}k+\alpha=N+1\\~~~k,\alpha>0\end{subarray}}[[\phi(x),\phi(z)]
 _{\mathcal{G},\alpha},\phi(\alpha(y))]_{\mathcal{G},k}\\[1pt]
\nonumber
&&- \label{T}\sum\limits_{\begin{subarray}{l}\beta+\gamma+\alpha+k=N+1\\~~~k,\beta+\gamma>0\\ ~~~\alpha,\beta,\gamma\geq0\end{subarray}}
[[\phi_{\beta}(x),\phi_{\gamma}(z)]_{\mathcal{G},\alpha},\phi(\alpha(y))]_{\mathcal{G},k}.
\end{align}
The term \eqref{*} vanishes with the second term of $-\mathcal{O}b_{\mathcal{G}}\circ\phi$.
We may write \eqref{4} as follows
\begin{eqnarray}
  \label{C}\sum'[\phi_{i}([x,y]_{\mathcal{L},0}),\phi_{j}(\alpha(z))]_{\mathcal{G},k}&=&
 \sum'_{\begin{subarray}{l}\alpha+\beta+\gamma=i\\ \alpha,\beta,\gamma\geq0
  \end{subarray}}[[\phi_{\beta}(x),\phi_{\gamma}(y)]_{\mathcal{G},\alpha},\phi_{j}(\alpha(z))]_{\mathcal{G},k}\\[1pt]
 \label{D}&-&\sum'_{\begin{subarray}{l}\lambda+\mu=i\\ 1\leq\mu\leq i\end{subarray}}
  [\phi_{\lambda}([x,y]_{\mathcal{L},\mu}),\phi_{j}(\alpha(z))]_{\mathcal{G},k},
\end{eqnarray}
where the sum on the right hand side of \eqref{C} is\\
\begin{equation}\label{somme3}
\begin{array}{llcl}\sum'_{\begin{subarray}{l}\alpha+\beta+\gamma=i\\ \alpha,\beta,\gamma\geq0
\end{subarray}}=\sum\limits_{\begin{subarray}{l}\alpha+\beta+\gamma,k>0\\ ~\alpha,\beta,\gamma\geq0\\~~~j=0
\end{subarray}}+\sum\limits_{\begin{subarray}{l}\alpha+\beta+\gamma,j>0\\ ~\alpha,\beta,\gamma\geq0\\
~~~k=0\end{subarray}}
+\sum\limits_{\begin{subarray}{l}\alpha=\beta=\gamma=0\\ ~~j,k>0
\end{subarray}}
+\sum\limits_{\begin{subarray}{l}\alpha+\beta+\gamma,k,j>0\\ ~~~\alpha,\beta,\gamma\geq0
\end{subarray}}\end{array}.\end{equation}
The first term of \eqref{somme3} can be written as
\begin{align}
&\sum\limits_{\begin{subarray}{l}\alpha+\beta+\gamma,k>0\\ ~\alpha,\beta,\gamma\geq0\\~~~j=0
  \end{subarray}}[[\phi_{\beta}(x),\phi_{\gamma}(y)]_{\mathcal{G},\alpha},\phi(\alpha(z))]_{\mathcal{G},k}
 \label{**}=\sum\limits_{\begin{subarray}{l}k+\alpha=N+1\\~~~k,\alpha>0\end{subarray}}[[\phi(x),\phi(y)]
  _{\mathcal{G},\alpha},\phi(\alpha(z))]_{\mathcal{G},k}\\[1pt]
\nonumber
\label{M}&+ \sum\limits_{\begin{subarray}{l}\beta+\gamma+\alpha+k=N+1\\~~~k,\beta+\gamma>0\\ ~~~\alpha,\beta,\gamma\geq0\end{subarray}}
  [[\phi_{\beta}(x),\phi_{\gamma}(y)]_{\mathcal{G},\alpha},\phi(\alpha(z))]_{\mathcal{G},k}.
\end{align}
The term \eqref{**} vanishes with the first term of $-\mathcal{O}b_{\mathcal{G}}\circ\phi$.
We can write \eqref{6} as follows
\begin{eqnarray}
\label{E}\sum'[\phi_{i}([y,z]_{\mathcal{L},0}),\phi_{j}(\alpha(x))]_{\mathcal{G},k}&=&
\sum'_{\begin{subarray}{l}\alpha+\beta+\gamma=i\\ \alpha,\beta,\gamma\geq0
\end{subarray}}[[\phi_{\beta}(y),\phi_{\gamma}(z)]_{\mathcal{G},\alpha},\phi_{j}(\alpha(x))]_{\mathcal{G},k}\\[1pt]
\label{F}&-&\sum'_{\begin{subarray}{l}\lambda+\mu=i\\ 1\leq\mu\leq i\end{subarray}}
[\phi_{\lambda}([y,z]_{\mathcal{L},\mu}),\phi_{j}(\alpha(x))]_{\mathcal{G},k},
\end{eqnarray}
where the sum on the right hand side of \eqref{E} is\\
\begin{equation}\label{somme4}
\sum'_{\begin{subarray}{l}\alpha+\beta+\gamma=i\\ \alpha,\beta,\gamma\geq0
\end{subarray}}=\sum\limits_{\begin{subarray}{l}\alpha+\beta+\gamma,k>0\\ ~\alpha,\beta,\gamma\geq0\\~~~j=0
\end{subarray}}+\sum\limits_{\begin{subarray}{l}\alpha+\beta+\gamma,j>0\\ ~\alpha,\beta,\gamma\geq0\\
~~~k=0\end{subarray}}
+\sum\limits_{\begin{subarray}{l}\alpha=\beta=\gamma=0\\ ~~j,k>0
\end{subarray}}
+\sum\limits_{\begin{subarray}{l}\alpha+\beta+\gamma,k,j>0\\ ~~~\alpha,\beta,\gamma\geq0
  \end{subarray}}.\end{equation}
The first term of \eqref{somme4} can be written as
\begin{align}
&&\sum\limits_{\begin{subarray}{l}\alpha+\beta+\gamma,k>0\\ ~\alpha,\beta,\gamma\geq0\\~~~j=0
  \end{subarray}}[[\phi_{\beta}(y),\phi_{\gamma}(z)]_{\mathcal{G},\alpha},\phi(\alpha(x))]_{\mathcal{G},k}
 \label{***}=\sum\limits_{\begin{subarray}{l}k+\alpha=N+1\\~~~k,\alpha>0\end{subarray}}[[\phi(y),\phi(z)]
_{\mathcal{G},\alpha},\phi(\alpha(x))]_{\mathcal{G},k}\\[1pt]
\nonumber
&&+ \label{K}\sum\limits_{\begin{subarray}{l}\beta+\gamma+\alpha+k=N+1\\~~~k,\beta+\gamma>0\\ ~~~\alpha,\beta,\gamma\geq0\end{subarray}}
[[\phi_{\beta}(y),\phi_{\gamma}(z)]_{\mathcal{G},\alpha},\phi(\alpha(x))]_{\mathcal{G},k}.
\end{align}
The term \eqref{***} vanishes with the third term of $-\mathcal{O}b_{\mathcal{G}}\circ\phi$.\\
On the other hand, one
uses the fact that
\begin{equation*}
 \circlearrowleft_{x,y,z} \bigg(\sum\limits_{i=0}^{N}[[x,z]_{\mathcal{L},0},\alpha(y)]_{\mathcal{L},N+1-i}+
 \sum\limits_{j=0}^{N}[[x,z]_{\mathcal{L},N+1-j},\alpha(y)]_{\mathcal{L},j}\bigg)=0,
\end{equation*}
then
\begin{eqnarray}
\nonumber && \eqref{10}+\eqref{11}+\eqref{12}=-\sum_{i=1}^{N}\phi_{i}([[x,y]_{\mathcal{L},0},\alpha(z)]_{\mathcal{L},N+1-i})\\[1pt]
\nonumber&&+\sum_{i=1}^{N}\phi_{i}([[x,z]_{\mathcal{L},0},\alpha(y)]_{\mathcal{L},N+1-i})
-\sum_{i=1}^{N}\phi_{i}([[y,z]_{\mathcal{L},0},\alpha(x)]_{\mathcal{L},N+1-i})\\[1pt]
\nonumber&=&-\sum\limits_{i=1}^{N}\circlearrowleft_{x,y,z}\left(\phi_{i}
([[x,y]_{\mathcal{L},0},\alpha(z)]_{\mathcal{L},N+1-i})\right)
=\sum\limits_{\begin{subarray}{l}i+j+k=N+1\\~~i,j>0,k\geq0\end{subarray}}\circlearrowleft_{x,y,z}
\phi_{i}\big([[x,y]_{\mathcal{L},j},\alpha(z)]_{\mathcal{L},k}\big)\\[1pt]
\label{somme}&=&\sum\limits_{j=1}^{N}\circlearrowleft_{x,y,z}\big
(\sum\limits_{\begin{subarray}{l}i+k=N+1-j\\~~~~i,k\geq0\end{subarray}}
\phi_{i}([[x,y]_{\mathcal{L},j},\alpha(z)]_{\mathcal{L},k})\big)\\[1pt]
\label{rotation}&&-\phi\bigg(\sum\limits_{\begin{subarray}{l}j+k=N+1\\~~~j,k>0\end{subarray}}
\circlearrowleft_{x,y,z}[[x,y]_{\mathcal{L},j},\alpha(z)]_{\mathcal{L},k}\bigg)\bigg).
\end{eqnarray}
The term \eqref{rotation} vanishes with $\phi\circ\mathcal{O}b_{\mathcal{L}}$.
The term \eqref{somme} may be  written as
\begin{align}
\nonumber&\circlearrowleft_{x,y,z}\sum\limits_{j=1}^{N}\bigg
(\sum\limits_{\begin{subarray}{l}i+k=N+1-j\\~~~~i,k\geq0\end{subarray}}
\phi_{i}([[x,y]_{j},\alpha(z)]_{\mathcal{L},k}\bigg)
\\[1pt]
\nonumber&
=\sum\limits_{j=1}^{N}\circlearrowleft_{x,y,z}\bigg
(\sum\limits_{\begin{subarray}{l}\beta+\gamma+\alpha=N+1-j\\ ~~~~\alpha,\beta,\gamma\geq0\end{subarray}}
[\phi_{\beta},([x,y]_{j}),\phi_{\gamma}(\alpha(z))]_{\mathcal{G},\alpha}\bigg)\\[1pt]
\label{somme123}&=\circlearrowleft_{x,y,z}\sum\limits_{i=1}^{N}[\phi_{i}([x,y]_{\mathcal{L},N+1-i}),
\phi(\alpha(z))]_{\mathcal{G},0}\\[1pt]
\label{sommeBFD}&+\circlearrowleft_{x,y,z}\sum'\limits_{\begin{subarray}{l}\lambda+\mu+j+\alpha=N+1\\
~~~~~\lambda+\mu=j\\~~~~1\leq\mu\leq j\end{subarray}}[\phi_{\lambda}([x,y]_{\mathcal{L},\mu}),\phi_{k}(\alpha(z))]_{\mathcal{G},\alpha}.
\end{align}
The term
\eqref{sommeBFD} may be written as
\begin{eqnarray*}
&\circlearrowleft_{x,y,z}\sum'\limits_{\begin{subarray}{l}\lambda+\mu+j+\alpha=N+1\\
~~~~~\lambda+\mu=j\\~~~~1\leq\mu\leq j\end{subarray}}
[\phi_{\lambda}([x,y]_{\mathcal{L},\mu}),\phi_{k}(\alpha(z))]_{\mathcal{G},\alpha}
%\\[1pt]
=\sum'\limits_{\begin{subarray}{l}\lambda+\mu+j+\alpha=N+1\\
~~~~~\lambda+\mu=j\\~~~~1\leq\mu\leq j\end{subarray}}
[\phi_{\lambda}([x,y]_{\mathcal{L},\mu}),\phi_{k}(\alpha(z))]_{\mathcal{G},\alpha}\\[1pt]
&+\sum'\limits_{\begin{subarray}{l}\lambda+\mu+j+\alpha=N+1\\
~~~~~\lambda+\mu=j\\~~~~1\leq\mu\leq j\end{subarray}}
[\phi_{\lambda}([z,x]_{\mathcal{L},\mu}),\phi_{k}(\alpha(y))]_{\mathcal{G},\alpha}
+\sum'\limits_{\begin{subarray}{l}\lambda+\mu+j+\alpha=N+1\\
~~~~~\lambda+\mu=j\\~~~~1\leq\mu\leq j\end{subarray}}
[\phi_{\lambda}([y,z]_{\mathcal{L},\mu}),\phi_{k}(\alpha(x))]_{\mathcal{G},\alpha}\\[1pt]
&=-(\eqref{B}+\eqref{D}+\eqref{F}).
\end{eqnarray*}
Then $\eqref{sommeBFD}+\eqref{B}+\eqref{D}+\eqref{F}=0$
and \eqref{somme123} may be  written as
\begin{eqnarray*}
&\circlearrowleft_{x,y,z}\sum\limits_{i=1}^{N}[\phi_{i}([x,y]_{\mathcal{L},N+1-i}),
\phi(\alpha(z))]_{\mathcal{G},0}
=\sum\limits_{i=1}^{N}[\phi_{i}([y,z]_{\mathcal{L},N+1-i}),
\phi(\alpha(x))]_{\mathcal{G},0}\\[1pt]
&+\sum\limits_{i=1}^{N}[\phi_{i}([z,x]_{\mathcal{L},N+1-i}),
\phi(\alpha(y))]_{\mathcal{G},0}
+\sum\limits_{i=1}^{N}[\phi_{i}([x,y]_{\mathcal{L},N+1-i}),
\phi(\alpha(z))]_{\mathcal{G},0}\\[1pt]
&=-(\eqref{7}+\eqref{8}+\eqref{9}).
\end{eqnarray*}
Then $\eqref{somme123}+(\eqref{7}+\eqref{8}+\eqref{9})=0$.\\
The remaining components of $\delta^{2}\mathcal{O}b_{\phi}+\phi\mathcal{O}b_{\mathcal{L}}
-\mathcal{O}b_{\mathcal{G}}\phi$ can be written as
$$\circlearrowleft_{\phi_{\beta},\phi_{\gamma},\phi_{j}}\widetilde{\sum}
[[\phi_{\beta}(x),\phi_{\gamma}(y)]_{\mathcal{G},\alpha},\beta(\phi_{j}(z))]_{\mathcal{G},k}=0,$$
where
$\widetilde{\sum}=\sum\limits_{\begin{subarray}{l}j+\beta=N+1\\ ~~~\beta,j>0\end{subarray}}
+\sum\limits_{\begin{subarray}{l}j+\gamma=N+1\\ ~~~\gamma,j>0\end{subarray}}
+\sum\limits_{\begin{subarray}{l}\gamma+\beta=N+1\\ ~~~\beta,\gamma>0\end{subarray}}
+\sum\limits_{\begin{subarray}{l}j+\beta+\gamma=N+1\\~~~ \beta,j,\gamma>0\end{subarray}}
+\sum\limits_{\begin{subarray}{l}j+\beta+\gamma+j+\alpha=N+1\\~~~1\leq\beta+\gamma+j\leq N\\
~~~~~\alpha,\beta,\gamma,j,k\geq 0\end{subarray}}.
$\\
Then $\delta^{2}\mathcal{O}b_{\phi}+\phi\mathcal{O}b_{\mathcal{L}}
-\mathcal{O}b_{\mathcal{G}}\phi=0$.\\
One has moreover
$\delta^{2}([\cdot, \cdot]_{\mathcal{L},N+1},[\cdot, \cdot]_{\mathcal{G},N+1},\phi_{N+1})=
(\mathcal{O}b_{\mathcal{L}},\mathcal{O}b_{\mathcal{G}},\mathcal{O}b_{\phi})$.
Then, the order $N$ formal deformation extends to an order $N+1$ formal deformation
whenever $(\mathcal{O}b_{\mathcal{L}},\mathcal{O}b_{\mathcal{G}},\mathcal{O}b_{\phi})$
is a coboundary in $C_{HL}^{3}(\phi,\phi)$.
\end{proof}
%%%%%%%%%%%%%%%%%%%%%%%%%%%%%%%%%%%%%%%%%%%%%%%%%%%%%%%%%%%%%%%%%%%%%%%%%%%%%%%%%%%%%%
\section{Example}
We compute in this section    a   cohomology of a given Hom-Lie algebra morphism and  discuss some deformations.
Let $\mathfrak{g}_{1}=(\mathfrak{g}_{1},[\cdot,\cdot]_{1},\alpha_1)$ and $\mathfrak{g}_{2}=(\mathfrak{g}_{2},[\cdot,\cdot]_{2},\alpha_2)$ be two Hom-Lie algebras defined with respect to the basis $\{e_{1},e_{2},e_{3}\}$
(resp. $\{f_{1},f_{2},f_{3}\}$) by 
$$\mathfrak{g}_{1}:\small{\left\{\begin{array}{llll}
& [e_{1},e_{2}]_{1}=e_{3}\\
& [e_{2},e_{3}]_{1}=0\\
& [e_{1},e_{3}]_{1}=0
\end{array}\right.} \  \alpha_{1}=\small{\begin{pmatrix}&p_{1}&0&0\\&0&p_{2}&0\\&0&0&p_{1}p_{2}
\end{pmatrix}},
\ \mathfrak{g}_{2}:\small{\left\{\begin{array}{llll}
&[f_{1},f_{2}]_{2}=f_{1}+f_{3}\\
&[f_{2},f_{3}]_{2}=f_{2}\\
&[f_{1},f_{3}]_{2}=f_{1}+2f_{3}
\end{array}\right.} \ \alpha_{2}=\small{\begin{pmatrix}&1&0&0\\&0&2&0\\&0&0&2
\end{pmatrix}},$$
where $ p_{1},p_{2},a,b,c,d$ are parameters.\\
%%%%%%%%%%%%%%%%%%%%%%%%%%%%%%%%%%%%%%%%%%%%%%%%%%%%%%%%%%%%%%%%%%%%%%%%%%%%%%%%%%%%%%%%%%%%%
%%%%%%%%%%%%%%%%%%%%%%%%%%%%%%%%%%%%%%%%%%%%%%%%%%%%%%%%%%%%%%%%%%%%%%%%%%%%%%%%%%%%%%%%%%%%%%%%
Let $\phi_{1,2}:\mathfrak{g}_{1}\rightarrow \mathfrak{g}_{2}$ be a  Hom-Lie algebra morphism.
%that is $\phi_{1,2}\circ[\cdot, \cdot]_{1}=[\phi_{1,2},\phi_{1,2}]_{2}$ and $\phi_{1,2}\circ\alpha_{1}=\alpha_{2}\circ\phi_{1,2}$.
The morphism $\phi_{1,2}$ is wholly determined by a set of structure constants $\lambda_{i,j}$.
%then for each $i$ from $1$ to $3,~\phi_{1,2}(e_{i})$ is an element of $\mathfrak{g}_{2}$
%and these components $\lambda_{i,j}$ are the structures constants:
%$$\phi_{1,2}(e_{j})=\sum\limits_{j=1}^{3}\lambda_{i,j}f_{i}$$
It turns out that we have the following cases\\
$$if ~p_{1}=p_{2}=2\qquad\left\{\begin{array}{ll}
&\phi_{1,2}^{1}(e_{1})=\lambda_{2,1}f_{2}
+\lambda_{3,1}f_{3}\\
&\phi_{1,2}^{1}(e_{2})=\lambda_{2,2}f_{2}+\frac{\lambda_{2,2}\lambda_{3,1}}{\lambda_{2,1}}f_{3}
\\
&\phi_{1,2}^{1}(e_{3})=0
\end{array}\right.$$
and $$~if ~p_{1}=2 ~and ~p_{2}=0 \qquad \left\{\begin{array}{ll}
&\phi_{1,2}^{2}(e_{1})=\lambda_{2,1}f_{2}
+\lambda_{3,1}f_{3}\\
&\phi_{1,2}^{2}(e_{2})=0\\
&\phi_{1,2}^{2}(e_{3})=0
\end{array}\right.$$

We come now to the computation of $\mathfrak{g}_{1}$ cohomology. A $2$-cochain is given by a triple
$(\psi,\varphi,\rho),$ where $\psi:\mathfrak{g}_{1}\times\mathfrak{g}_{1}\rightarrow\mathfrak{g}_{1}$,
$ \varphi:\mathfrak{g}_{2}\times\mathfrak{g}_{2}\rightarrow\mathfrak{g}_{2}$,
$\rho:\mathfrak{g}_{1}\rightarrow\mathfrak{g}_{2}$. The  $2$-cochains of the Hom-Lie
algebras $\mathfrak{g}_{1}$ are defined by $\psi_{i},i=1,\cdots ,8$,  \\
\begin{small}
 $$\left\{\begin{array}{lll}&\psi_{1}(e_{1},e_{2})=a_{1}e_{2}\\
&\psi_{1}(e_{2},e_{3})=0\\
&\psi_{1}(e_{1},e_{3})=0
\end{array}\right.\quad\left\{\begin{array}{lll}&\psi_{2}(e_{1},e_{2})=a_{2}e_{3}\\
&\psi_{2}(e_{2},e_{3})=0\\
&\psi_{2}(e_{1},e_{3})=0
\end{array}\right.\quad\left\{\begin{array}{lll}&\psi_{3}(e_{1},e_{2})=0\\
&\psi_{3}(e_{2},e_{3})=a_{3}e_{3}\\
&\psi_{3}(e_{1},e_{3})=0
\end{array}\right.\quad\left\{\begin{array}{lll}&\psi_{4}(e_{1},e_{2})=0\\
&\psi_{4}(e_{2},e_{3})=0\\
&\psi_{4}(e_{1},e_{3})=a_{4}e_{2}
\end{array}\right.$$

$$
\left\{\begin{array}{lll}&\psi_{5}(e_{1},e_{2})=0\\
&\psi_{5}(e_{2},e_{3})=0\\
&\psi_{5}(e_{1},e_{3})=a_{5}e_{3}
\end{array}\right.\quad\left\{\begin{array}{lll}&\psi_{6}(e_{1},e_{2})=a_{6}e_{1}\\
&\psi_{6}(e_{2},e_{3})=0\\
&\psi_{6}(e_{1},e_{3})=0
\end{array}\right.\quad\left\{\begin{array}{lll}&\psi_{7}(e_{1},e_{2})=0\\
&\psi_{7}(e_{2},e_{3})=a_{7}e_{1}\\
&\psi_{7}(e_{1},e_{3})=0
\end{array}\right.\quad\left\{\begin{array}{lll}&\psi_{8}(e_{1},e_{2})=0\\
&\psi_{8}(e_{2},e_{3})=-\frac{p_{2}}{p_{1}}a_{8}e_{2}\\
&\psi_{8}(e_{1},e_{3})=a_{8}e_{1}
\end{array}\right.
$$
\end{small}
where $a_{1},\cdots,a_{8}$ are  parameters. \\
We obtain the following results
\begin{enumerate}
  \item If $p_{1}=0$ then $Z^{2}_{HL}(\mathfrak{g}_{1},\mathfrak{g}_{1})
=\big<\psi_{2},\psi_{3},\psi_{5},\psi_{6},\psi_{7}\big>$. Hence, dim$Z^{2}_{HL}(\mathfrak{g}_{1},\mathfrak{g}_{1})=5$.
  \item If $p_{1}=1$ and $p_{2}\not\in\{-1,0,1\}$, $Z^{2}_{HL}(\mathfrak{g}_{1},\mathfrak{g}_{1})
=\big<\psi_{1},\psi_{2},\psi_{4},\psi_{5}\big>$, with dimension $4$.
\begin{enumerate}
  \item If $p_{2}=1$, $Z^{2}_{HL}(\mathfrak{g}_{1},\mathfrak{g}_{1})$ is generated in addition by $\{\psi_{3},\psi_{6},\psi_{7},\psi_{8}\}$.
  \item If $p_{2}=0$, $Z^{2}_{HL}(\mathfrak{g}_{1},\mathfrak{g}_{1})$ is generated in addition by $\{\psi_{3}\}$.
  \item If $p_{2}=-1$, $Z^{2}_{HL}(\mathfrak{g}_{1},\mathfrak{g}_{1})$ is generated in addition by $\{\psi_{7}\}$.
\end{enumerate}
  \item If $p_{1}=-1$ and $p_{2}\not\in\{-1,0,1\}$ then $Z^{2}_{HL}(\mathfrak{g}_{1},\mathfrak{g}_{1})
=\big<\psi_{2},\psi_{4}\big>$, with dimension $2$.
\begin{enumerate}
  \item If $p_{2}=1$, $Z^{2}_{HL}(\mathfrak{g}_{1},\mathfrak{g}_{1})$ is generated in addition by $\{\psi_{3},\psi_{6},\psi_{7}\}$.
  \item If $p_{2}=0$, $Z^{2}_{HL}(\mathfrak{g}_{1},\mathfrak{g}_{1})$ is generated in addition by $\{\psi_{1},\psi_{3},\psi_{5}\}$.
  \item If $p_{2}=-1$, $Z^{2}_{HL}(\mathfrak{g}_{1},\mathfrak{g}_{1})$ is generated in addition by $\{\psi_{7},\psi_{8}\}$.
\end{enumerate}
\item If $p_{1}\not\in\{-1,0,1\}$ and $p_{2}=\frac{1}{p_{1}}$, then $Z_{HL}^{2}(\mathfrak{g}_{1},\mathfrak{g}_{1})
=\big<\psi_{2},\psi_{8}\big>$, with dimension $2$.
\begin{enumerate}
  \item If $p_{2}=1$, $Z^{2}_{HL}(\mathfrak{g}_{1},\mathfrak{g}_{1})$ is generated by $\{\psi_{2},\psi_{3},\psi_{6},\psi_{7}\}$.
  \item If $p_{2}=0$, $Z^{2}_{HL}(\mathfrak{g}_{1},\mathfrak{g}_{1})$ is generated by $\{\psi_{1},\psi_{2},\psi_{3},\psi_{4},\psi_{5}\}$.
  \item If $p_{2}=-1$, $Z^{2}_{HL}(\mathfrak{g}_{1},\mathfrak{g}_{1})$ is generated by $\{\psi_{2},\psi_{7}\}$.
\end{enumerate}
\item If $p_{1}\neq\{-1,0,1\}$ and $p_{2}\neq\{-1,0,1\}$, such  that $p_{2}\neq\frac{1}{p_{1}}$, then $Z^{2}_{HL}(\mathfrak{g}_{1},\mathfrak{g}_{1})
=\big<\psi_{2}\big>$.
\end{enumerate}
All cochains are not coboundaries except $\psi_{2}$.
Therefore, in summary, we have
\begin{equation*}
dim~H^{2}_{HL}(\mathfrak{g}_{1},\mathfrak{g}_{1})
= \left\{
\begin{array}{llllll}
&4 \qquad \hbox{ if } p_{1}=0 \hbox{ with } a_{3}\neq a_{6}\\
&3 \quad \hbox{ if } p_{1}=1 \hbox{ and } p_{2}\not\in\{-1,0,1\} \hbox{ with } a_{1}\neq -a_{5} .\\
&7 \quad \hbox{ if } p_{1}=1 \hbox{ and } p_{2}=1 \hbox{ with } a_{1}\neq -a_{5} \hbox{ and } a_{3}\neq a_{6}.\\
&4 \quad \hbox{ if } p_{1}=1 \hbox{ and } p_{2}=0 \hbox{ with } a_{1}\neq -a_{5}.\\
&4 \quad \hbox{ if } p_{2}=-1 \hbox{ and } p_{2}=-1\hbox{ with } a_{1}\neq -a_{5}.\\[3pt]
&1 \quad \hbox{ if } p_{1}=-1 \hbox{ and } p_{2} \not\in\{-1,0,1\} \\
&4 \quad \hbox{ if } p_{1}=-1 \hbox{ and } p_{2}=1 \hbox{ with } a_{3}\neq a_{6}.\\
&4 \quad \hbox{ if } p_{1}=-1 \hbox{ and } p_{2}=0.\\
&3 \quad \hbox{ if } p_{1}=-1 \hbox{ and } p_{2}=-1.\\[3pt]
&1 \quad \hbox{ if } p_{1} \not\in\{-1,0,1\} \hbox{ and } p_{2}=\frac{1}{p_{1}}.\\
&3 \quad \hbox{ if } p_{1} \not\in\{-1,0,1\} \hbox{ and } p_{2}=1 \hbox{ with } a_{3}\neq a_{6}.\\
&4 \quad \hbox{ if } p_{1} \not\in\{-1,0,1\} \hbox{ and } p_{2}=0 \hbox{ with } a_{1}\neq -a_{5}. \\
&1 \quad \hbox{ if } p_{1} \not\in\{-1,0,1\} \hbox{ and } p_{2}=-1.\\
&0 \quad \hbox{ if } p_{1} \not\in\{-1,0,1\} \hbox{ and } p_{2}\not\in\{-1,0,1\}.
\end{array}\right.
\end{equation*}
Now, we consider  the second Lie algebra $\mathfrak{g}_{2}$. The $2$-cocycles are defined as
%\begin{flushleft}
$$\left\{\begin{array}{lll}\varphi(f_{1},f_{2})&=k_{1}f_{2}+k_{2}f_{3}\\
\varphi(f_{2},f_{3})&=0\\
\varphi(f_{1},f_{3})&=k_{3}f_{2}-k_{1}f_{3}.
\end{array}\right.$$
%\end{flushleft}
%for $k_{2}=0, ~H^{2}_{HL}(\mathfrak{g}_{2},g_{2})$ is $1$-dimensional.
%The generators are $(\varphi_{1},\varphi_{3})$ and defined as follow
%$$\left\{\begin{array}{lll}&\varphi_{1}(f_{1},f_{2})=f_{2}\\
%&\varphi_{1}(f_{2},f_{3})=0\\
%&\varphi_{1}(f_{1},f_{3})=-f_{3}
%\end{array}\right.\quad \left\{\begin{array}{lll}&\varphi_{3}(f_{1},f_{2})=0\\
%&\varphi_{3}(f_{2},f_{3})=0\\
%&\varphi_{3}(f_{1},f_{3})=f_{2}
%\end{array}\right.$$
%Otherwise
Therefore,  dim$H^{2}_{HL}(\mathfrak{g}_{2},\mathfrak{g}_{2})=1$ and it is generated by
%$\varphi_{1},\varphi_{2},\varphi_{3}$, where
$$\left\{\begin{array}{lll}\varphi_{2}(f_{1},f_{2})&=f_{3}\\
\varphi_{2}(f_{2},f_{3})&=0\\
\varphi_{2}(f_{1},f_{3})&=0.
\end{array}\right.$$
Now, we consider the third component. Set $\rho(e_{i})=\sum_{k=1}^{3}a_{k,i}f_{k}$, we start with
the first morphism $\phi_{1,2}^{1}$. There is only one $1$-cocycle corresponding to $\psi=\psi_{2}$ and $\varphi$,
and it is  given by
$$\rho=\left\{\begin{array}{llll}
\rho(e_{1})&=(-\frac{a_{3,2}\lambda_{2,1}^{2}}{\lambda_{2,2}\lambda_{3,1}}
+\frac{a_{2,2}\lambda_{2,1}}{\lambda_{2,2}}+\frac{a_{3,1}\lambda_{2,1}}{\lambda_{3,1}})f_{2}
+a_{3,1}f_{3}\\
\rho(e_{2})&=a_{2,2}f_{2}+a_{3,2}f_{3}\\
\rho(e_{3})&=0.
\end{array}\right.$$

Therefore $H^{1}_{HL}(\mathfrak{g}_{1},\mathfrak{g}_{2})$ is $6$-dimensional.\\
%It is turn out that $\rho$ is a cobord. Therefore $[\rho]=0$.\\
For the second morphism $\phi_{1,2}^{2}$, we have 
$$\rho=\left\{\begin{array}{llll}
\rho(e_{1})&=a_{2,1}f_{2}+a_{3,1}f_{3}\\
\rho(e_{2})&=0\\
\rho(e_{3})&=0.
\end{array}\right.
$$
Therefore $H^{1}_{HL}(\mathfrak{g}_{1},\mathfrak{g}_{2})$ is $2$-dimensional.\\

%which is a cobord. Therefore
%$H^{1}_{HL}(\mathfrak{g}_{1},\mathfrak{g}_{2})$ is equal to $0$.\\
%%%%%%%%%%%%%%%%%%%%%%%%%%%%%%%%%%%%%%%%%%%%%%%%%%%%%%%%%%%%%%%%%%%%%%%%%%%%
Now, we provide examples of deformations. For   $\mathfrak{g}_{1}$ we consider the deformed brackets  $[~,~]_{1}'$.
%For $p_{1}=-1, ~p_{2}=-1$
$$[e_{1},e_{2}]'_{1}=e_{3};~[e_{2},e_{3}]'_{1}=0;~[e_{1},e_{3}]'_{1}=wte_{2},$$ where $w$ is a parameter, or
%And for $p_{1},~p_{2}=-1$
$$[e_{1},e_{2}]'_{1}=e_{3}+te_{2};~[e_{2},e_{3}]'_{1}=0;~[e_{1},e_{3}]'_{1}=te_{3}.$$
Let $[~,~]_{2}'$ be a deformation of $\mathfrak{g}_{2}$ defined by
$$[f_{1},f_{2}]'_{2}=af_{1}+(b+tk_{2})f_{3};~[f_{2},f_{3}]'_{2}=cf_{2};~[f_{2},f_{3}]'_{2}=df_{1}+2af_{3}$$
and $\widetilde{\phi}$ a deformation of the second morphism given by
$$\widetilde{\phi}(e_{1})=(\lambda_{2,1}+a_{2,1}t)f_{2}+(\lambda_{3,1}+a_{3,1}t)f_{3};~\widetilde{\phi}
(e_{2})=0;~\widetilde{\phi}(e_{3})=0.$$
Then $([~,~]'_{1},[~,~]'_{2},\widetilde{\phi})$ is a deformation of $\phi_{1,2}^{2}$.\\
Let $\widetilde{\phi}$ be a deformation of the first morphism, where
\begin{eqnarray*}
\widetilde{\phi}(e_{1})&=&(\lambda_{21}+(-\frac{a_{3,2}\lambda^{2}_{2,1}}{\lambda_{2,2}\lambda_{3,1}}
+\frac{a_{22}\lambda_{2,1}}{\lambda_{2,2}}+\frac{a_{3,1}\lambda_{2,1}}{\lambda_{3,1}})t)f_{2}+(\lambda_{3,1}+ta_{3,1})f_{3}\\
\widetilde{\phi}(e_{2})&=&(\lambda_{2,2}+ta_{2,2})f_{2}+(\frac{\lambda_{2,2}\lambda_{3,1}}{\lambda_{2,1}}+ta_{3,2})f_{3}\\
\widetilde{\phi}(e_{3})&=&0
\end{eqnarray*}
Then $([~,~]_{1},[~,~]'_{2},\widetilde{\phi})$ is a deformation of $\phi_{1,2}^{1}$.
%%%%%%%%%%%%%%%%%%%%%%%%%%%%%%%%%%%%%%%%%%%%%%%%%%%%%%%%%%%%%%%%%%%%%%%%%%%%%%%%%%%%%%%%%%%%%%%%%%%
%%%%%%%%%%%%%%%%%%%%%%%%%%%%%%%%%%%%%%%%%%%%%%%%%%%%%%%%%%%%%%%%%%%%%%%%%%%%%%%

%%%%%%%%%%%%%%%%%%%%%%%%%%%%%%%%%%%%%%%%%%%%%%%%%%%%%%%%%%%%%%%%%%%%%%%%%%%%%%%%%%%%%%%%%%%%%%%%
%\eject

\end{document}